\documentclass[12pt,centertags,oneside]{amsart}

\usepackage{amsmath,amstext,amsthm,amssymb,amsfonts,amscd}
\usepackage{mathrsfs}
\usepackage[colorlinks=true]{hyperref}
\usepackage{graphicx}
\usepackage{color,tocvsec2}
\usepackage{typearea}
\usepackage[shortlabels]{enumitem}
\usepackage{lscape}
\usepackage[all]{xy}
\usepackage{fouriernc}
\usepackage{amsrefs}
\usepackage{slashed}
\usepackage[size=\tiny]{todonotes}

\usepackage{multicol}        
\makeindex             

\usepackage[a4paper,width=14.66cm,top=2.52cm,bottom=2.52cm]{geometry}

\theoremstyle{plain}
\newtheorem{lem}{Lemma}[section]
\newtheorem{thm}[lem]{Theorem}
\newtheorem{prop}[lem]{Proposition}
\newtheorem{cor}[lem]{Corollary}

\theoremstyle{definition}
\newtheorem{defin}[lem]{Definition}
\newtheorem{exa}[lem]{Example}

\theoremstyle{remark}
\newtheorem{re}[lem]{Remark}

\numberwithin{equation}{section}
\numberwithin{figure}{section}

\newcommand{\cU}{\mathcal{U}}

\newcommand{\bR}{\mathbf{R}}
\newcommand{\bC}{\mathbf{C}}

\newcommand{\fg}{\mathfrak{g}}
\newcommand{\fh}{\mathfrak{h}}
\newcommand{\fk}{\mathfrak{k}}
\newcommand{\fm}{\mathfrak{m}}
\newcommand{\fl}{\mathfrak{l}}

\newcommand{\fp}{\mathfrak{p}}
\newcommand{\fq}{\mathfrak{q}}
\newcommand{\ft}{\mathfrak{t}}
\newcommand{\fz}{\mathfrak{z}}
\newcommand{\fu}{\mathfrak{u}}
\newcommand{\fa}{\mathfrak{a}}

\newcommand{\Yok}{Y_{0}^{\mathfrak k}}

\DeclareMathOperator{\Tr}{\mathrm Tr}
\DeclareMathOperator{\Ad}{\mathrm Ad}

\DeclareMathOperator{\vol}{\mathrm vol}

\DeclareMathOperator{\ad}{\mathrm ad}

\DeclareMathOperator{\End}{\mathrm End}
\DeclareMathOperator{\Hom}{\mathrm Hom}

\DeclareMathOperator{\rk}{\mathrm rk}

\newcommand{\<}{\langle}
\renewcommand{\>}{\rangle}
\newcommand{\ol}{\overline}
\newcommand{\ul}{\underline}

\renewcommand{\(}{\left(}
\renewcommand{\)}{\right)}
\renewcommand{\[}{\left[}
\renewcommand{\]}{\right]}
\renewcommand{\l}{\leqslant}

\newcommand{\g}{\geqslant}

\newcommand{\e}{\epsilon}

\begin{document}

\title{Heat kernel, large-time behavior, and  
Representation theory}
\author{Shu Shen}
 \address{Institut de Math\'ematiques de Jussieu-Paris Rive Gauche, 
Sorbonne Universit\'e,  4 place Jussieu, 75252 Paris Cedex 5, France.}
\email{shu.shen@imj-prg.fr}
\author{Yanli Song}
\address{Department of Mathematics,
1 Brookings Dr, St. Louis, MO, 63130, USA}
\email{yanlisong@wustl.edu}

\author{Xiang Tang}
\address{Department of Mathematics,
1 Brookings Dr, St. Louis, MO, 63130, USA}
\email{xtang@math.wustl.edu}

\thanks {}

\makeatletter
\@namedef{subjclassname@2020}{%
  \textup{2020} Mathematics Subject Classification}
\makeatother
\begin{abstract}
Given a real reductive group $G$, the purpose of this paper is to show an asymptotic formula of the large-time behavior of the $G$-trace of the heat operator on the associated symmetric spaces. Together with Carmona's proof on Vogan's lambda map, our results provide a geometric counterpart  of Vogan's minimal $K$-type theory.
\end{abstract}



\date{\today}

\dedicatory{}

\maketitle
\tableofcontents

\settocdepth{section}
\section*{Introduction}

\subsection{Background}
The purpose of this paper is to study the large-time behavior of the von Neumann trace of the heat operator on symmetric spaces, and to relate it to the representation theory of the underlying Lie groups.

Let $G$ be a connected real reductive Lie group with Lie algebra $\fg$, let $K \subset G$ be a maximal compact subgroup, and let $X = G/K$ be the associated symmetric space.
Let $\tau^{E} : K \to {\rm U}(E)$ be a finite-dimensional unitary representation of $K$, and let $F = G \times_K E$ be the associated $G$-equivariant Hermitian vector bundle on $X$.

The study of the spectral theory of the Casimir operator $C^{\fg,X}$ acting on $C^{\infty}(X,F)$ is deeply connected to the tempered representation theory of $G$ (see, e.g., \cite{AS77, Connes82, BCH}) and to Harish-Chandra's Plancherel formula for $G$ \cite{HC75,HC76,HC1976}.

Discrete series representations are the building blocks of tempered representations.
In \cite{HC65,HC_DSII}, Harish-Chandra classified all discrete series representations of $G$. They exist only when $G$ and $K$ have equal complex rank, and in this case, the discrete series representations of $G$ are in one-to-one correspondence with certain properly regular irreducible $K$-representations.

Following Harish-Chandra's classification, the geometric realisation of discrete series representations became a central problem in the 1970s.
In the case where $X = G/K$ is a Hermitian symmetric space, and $F$ is the holomorphic vector bundle associated to such a properly regular irreducible $K$-representation, Narasimhan and Okamoto \cite{NarasimhanOkamoto70} realised the discrete series representation as the $F$-valued $L^{2}$-Dolbeault cohomology.
More generally, when $G$ and $K$ have equal complex rank, Parthasarathy \cite{Parthasarathy72} realised the discrete series representation as $F$-twisted $L^{2}$-harmonic spinors.
In this case, $X$ has even dimension.
If $ \fg=\fp \oplus \mathfrak{k}$ is the Cartan decomposition of $\fg$, assume the adjoint $K$-action on $\mathfrak{p}$ lifts to the spinor $S^{\fp}$ (passing to a double cover if necessary).
If $S^{TX} = G\times_{K} S^{\mathfrak{p}}$ and if $D^{X} : C^{\infty}\left(X,S^{TX} \otimes F\right) \to C^{\infty}\left(X,S^{TX} \otimes F\right)$ is the twisted Dirac operator, then the discrete series representation can be realised as the $L^{2}$-kernel of $D^{X}$.
Parthasarathy's formula states that $\left(D^{X}\right)^{2}$ is a shift of the Casimir operator $C^{\fg,X}$ acting on $C^{\infty}\left(X,S^{TX} \otimes F\right)$.

In a more general context, Atiyah \cite{AtiyahL2} established the $L^{2}$-index theorem, which gives a geometric/topological formula for the $L^{2}$-index of twisted Dirac operators.
Applying these tools, Atiyah and Schmid \cite{AS77} gave a systematic treatment of the discrete series representations.
Index-theoretic methods are not limited to equal-rank groups and discrete series; they also provide a natural abelian group homomorphism from the representation ring of $K$ to the K-theory of the reduced $C^*$-algebra of $G$. 
The Connes–Kasparov isomorphism theorem \cite{wassermann, Larfforgue, CHS, CHST} asserts that this homomorphism is, in fact, an isomorphism.

Beyond discrete series representations, Vogan's minimal $K$-type theory provides a similar description for a class of tempered representations, namely those with real infinitesimal character.
As an analogue and generalisation of Harish-Chandra's description of the discrete series, Vogan \cite{Vogan79} established a one-to-one correspondence between the set of irreducible tempered representations of $G$ with real infinitesimal character and the set of all irreducible representations of $K$.

If $\tau^{E}$ is an irreducible representation of $K$, Harish-Chandra's Plancherel formula implies that the spectrum of the Casimir operator $C^{\fg,X}$ acting on $C^{\infty}(X,F)$ consists of a finite set of discrete points and finitely many half-lines, each bounded from below.
The discrete spectrum and the bottom of these half-lines are precisely realised by the irreducible tempered representations with real infinitesimal character that contain $\tau^{E}$ as a $K$-type.
Among these, those with the lowest spectrum correspond to the representations with minimal $K$-type \cite[Proposition 6.6.2, Lemma 6.6.5]{Vogan-book}.

Some aspects of these deep results can be recovered from the asymptotics of the heat operator $\exp\left(-\frac{t}{2} C^{\fg,X}\right)$ as $t \to \infty$.
In this paper, using Bismut's formula \cite{B09}, we develop a geometric approach to study the large-time behavior of the von Neumann $G$-trace of $\exp\left(-\frac{t}{2} C^{\fg,X}\right)$.
Our results are independent of Harish-Chandra's Plancherel theory, and serve as a geometric counterpart to Vogan's minimal $K$-type theory.

\subsection{The $ G$-trace and Bismut's formula}
The heat operator $ \exp\left(-\frac{t}{2} C^{\fg,X}\right)$ is a $ G$-invariant integral operator on $ C^{\infty}(X,F)$.
Since $ X$ is non-compact, the operator $ \exp\left(-\frac{t}{2} C^{\fg,X}\right)$ is not trace-class, but it has a well-defined von Neumann $ G$-trace $ \Tr_{G} \left[ \exp\left(-\frac{t}{2} C^{\fg,X}\right) \right]\in \mathbf{R}$.

If $ p_{t} \left(x,x^{\prime } \right)\in \Hom\left(F_{x^{\prime } }, F_{x} \right)$ is the integral kernel of $ \exp\left(-\frac{t}{2} C^{\fg,X}\right)$ with respect to a Riemannian volume form on $ X$, the $ G$-invariance property of $ \exp\left(-\frac{t}{2} C^{\fg,X}\right)$ implies that $ \Tr \left[ p_{t} \left(x,x\right) \right] $ is independent of $ x \in X$, so that the von Neumann $G$-trace of $\exp\left(-\frac{t}{2} C^{\fg,X}\right)$ has the following expression, 
\begin{align}\label{eq:xvuz}
 \Tr_{G} \left[ \exp\left(-\frac{t}{2} C^{\fg,X}\right) \right] =  \Tr \left[ p_{t} \left(x,x\right) \right]. 
\end{align}


Motivated by his previous works \cite{B05,B08a,B08,BL08},
Bismut developed an approach reminiscent of local index theory, constructing a family of differential operators $\mathcal{L}_b|_{b> 0 } $ on the extended tangent bundle $G\times_{K} \mathfrak{g}$. These operators fit within H\"ormander's theory of hypoelliptic operators of generalized Kolmogorov type, interpolating between $\frac{1}{2} C^{\mathfrak{g}, X} $ on  $X$  and the generator of the extended geodesic flow on $G\times_{K} \mathfrak{g}$. 
Bismut\footnote{Indeed, Bismut's results hold for all the semisimple orbital integrals.} introduced  a super $ G$-trace $ {\rm Tr}_{{\rm s},G} \left[\exp\left(-t \mathcal{L}_{b} \right)\right]$ and observed that $ {\rm Tr}_{{\rm s},G} \left[\exp\left(-t \mathcal{L}_{b} \right)\right]$ is independent of  $b$, 
a property that follows from the standard supersymmetry trick. 
By taking the limit as  $b \to 0$, he showed
\begin{align}
\lim _{b \rightarrow 0}\Tr_{{\rm s},G}\left[\exp \left(-t \mathcal{L} _b\right)\right]=\Tr_G\left[\exp\left(-\frac{t}{2}C^{\fg, X}\right)\right].  
\end{align}
Using a variation of Getzler's rescaling argument, 
he further computed the limit as  $b \to \infty$ and obtained a remarkable formula (Theorem \ref{thmB09}), which is given by  an integral on 
$\fk$ 
using the Lie theory data.
Various generalisations are obtained in \cite{B19,BSh22,Liu23}.
Successful applications of the Bismut formula have been found in the study of geometry of (locally) symmetric spaces \cite{B19,BMZ,Liu21,Liu24,Shen18,Shen21,Shen23,ShenYu22}.

\subsection{Our main results}
The following is the  main result of this paper. 
\begin{thm}\label{thmit}
(Theorems \ref{thm1} and \ref{thm2}) 
If $ \tau^{E} $ is irreducible, the $G$-trace of the heat operator has the following asymptotic formula,
	\begin{align}\label{eq:asymptotics}
		\Tr_{G}\[\exp\(-\frac{t}{2}C^{\fg,X}\)\]\sim \ul \alpha_{0} 
		t^{\ul\beta_{1}}e^{\ul \gamma_{2}t}, \quad \text{as } t \to \infty, 
	\end{align} 
where the constants $\ul \alpha_{0} > 0, \ul\beta_{1} \in -\frac{\mathbf{N}}{2}$, and $ \ul \gamma_{2} \in \bR$ are explicitly defined in \eqref{eqa0}. 

The condition $ \underline{\beta }_{1}  = 0 $ is realised if and only if $ G$ and $ K$ have the same complex rank, and $ \tau^{E} $ is regular in the sense of \eqref{eq:regHC}.
In this case,  there is $ \epsilon_{0} > 0 $ such that as $ t \to \infty$, 
 \begin{align}\label{eq:introtawi}
	\Tr_{G}\[\exp\(-\frac{t}{2}C^{\fg, 
		X}\)\]= \ul \alpha_{0} e^{t \ul \gamma_{2}} \left(1+\mathcal{O} \left(e^{-\epsilon_{0} t} \right)\right).
 \end{align}
\end{thm}


The second part of our Theorem \ref{thmit} is compatible with the theory of discrete series representations.
Indeed, under one of the equivalent conditions of Theorem \ref{thmit}, 
as  $t \to \infty$, the heat operator $\exp\(-\frac{t}{2}C^{\fg,X}\)$ concentrates on the discrete series representation $ \pi $ with the lowest $ K$-type $\tau^{E} $. 
In this case, the constants have the following representation theoretic interpretations:
\begin{itemize}
\item $\ul \alpha_{0} $ recovers the formal degree of $ \pi $;
\item $\ul\beta_{1}$ vanishes;
\item $\ul \gamma_{2}$ is the infinitesimal character of the negative half Casimir associated to the representation $ \pi $. 
\end{itemize} 

In general, as  $t \to \infty$, the $ G$-trace of the heat operator concentrates on the tempered representation with real infinitesimal character containing  $\tau^{E} $  as its lowest  $K$-type \cite{Vogan79}. 
Moreover, 
\begin{itemize} 
	\item $\ul \alpha_{0} $ is a generalisation of the Mehta-Macdonald integral \cite{Macdonald82}, whose explicit computation  is challenging;
	\item $\ul\beta_{1}$  serves as a Novikov-Shubin-type invariant, encoding local properties of the Plancherel measure on the tempered dual of $G$;
	\item $\ul \gamma_{2}$  is given by Vogan's  $\Lambda$-map \cite{Vogan79}. 
\end{itemize}



Our analysis uses crucially the Lie-theoretic information associated with the root system. This proof reveals a deep connection between the Bismut's hypoelliptic Laplacian on $X$ and Vogan's theory of unitary  $G$-representations.

As an application, Corollary \ref{cor:uemi} allows us to reduce the computation of $ \underline{\alpha} _{0}, \underline{\beta} _{1}, \underline{\gamma} _{2} $ to the special case where $ G$  is quasi-split and  $E$  is small, paralleling Vogan's \cite{Vogan98, Vogan-book} construction of  $G$-representations via cohomological induction. 
Additionally, we apply our main theorem to compute the Novikov-Shubin type invariant for smooth locally symmetric spaces (Proposition \ref{prop:novikovshubin} and Theorem \ref{thm:jwob}), which generalizes the results \cite{Lo-Me, Olbrich}.


Note that the structure of the above asymptotics can be also deduced from Harish-Chandra's Plancherel formula \cite[Theorem 13.11]{Knappsemi}.
The constants $\ul \alpha_{0}, \ul\beta_{1}, \ul \gamma_{2}$ can be determined by the Plancherel measure of $G$ and some $ K$-multiplicities of tempered $ G$ representations, which a priori are difficult to compute explicitly.
Our contribution consists in providing an explicit geometric formula for these constants in terms of the Lie-theoretic data of $\fg$.	

In Corollary \ref{cor:vtpj}, we will see that by combining our results with the (non-explicit) Plancherel formula, we obtain certain existence results for the tempered representations of $ G$.
This provides a geometric counterpart of Vogan's minimal $ K$-type theory. 
In a forthcoming paper, we will give a more precise connection between Bismut's formula and the Plancherel measure. 

\subsection{Organisation of the article}
This article is organised as follows.

In Section \ref{SBformula}, we recall the definition of $G$-trace and  the statement of Bismut's $G$-trace formula.

In Section \ref{SRoot}, we review some constructions related to root systems. 

In Section \ref{Smainresult}, we state our main result Theorem \ref{thmit}.

In Section \ref{sec:application}, we discuss applications of our main results in studying small representations in Vogan's theory and computing the Novikov-Shubin type invariant. 

The purpose of Sections \ref{Spf1} and \ref{SJtda} is to show Theorem \ref{thmit}.

In Section \ref{Spf1}, using Weyl groups, we reduce the proof to establishing a corresponding result for an integral over convex cones. 

Finally, in Section \ref{SJtda}, we study this integral by the Laplace Method. 

\subsection{Acknowledgment}
We would like to thank Alexandre Afgoustidis, Jean-Michel Bismut, and Nigel Higson for encouragement and inspiring discussion.
We owe special thanks to Michèle Vergne for her insightful suggestions, which greatly improved the first version of this paper.
Shen's research was  partially supported by ANR grant ANR-20-CE40-0017.
Song's research was partially supported by the NSF grant  DMS-1952557. Tang's research was partially supported the NSF grants DMS-1952551, DMS-2350181, and Simons Foundation grant MPS-TSM-00007714. 

\settocdepth{subsection}

\section{A geometric formula for $G$-trace}\label{SBformula}
This section aims to define the $G$-trace for the heat semigroup on the symmetric space associated with a connected real reductive Lie group $G$ and to review Bismut's geometric $G$-trace formula \cite[Theorem 6.1.1]{B09}. 

This section is organised as follows. 
In Sections \ref{sNotation}-\ref{sSyms}, we introduce some necessary notation, the reductive group $G$, a maximal compact subgroup $K$, the symmetric space $X=G/K$, Casimir operator $C^{\fg, X}$ on $X$, and some related constructions. 

In Section \ref{sGOIF}, we define the $G$-trace for the heat semigroup associated to the Casimir operator $ C^{\fg,X} $. 

In Section \ref{sGOIF1}, we recall Bismut's formula for the above $G$-trace. 

Finally, in Section \ref{subsec:sl2}, we illustrate the idea behind the proof of the main theorem using the example of ${\rm SL}(2, \bR)$.

\subsection{Notation}\label{sNotation}
We use the convention 
	\begin{align}
		&\mathbf N=\{0,1,2,\ldots\}, &\mathbf N^{*}=\{1,2,\ldots\}, 
		&&\bR_{+}=[0,\infty),&&&\mathbf{R}_{-} = (-\infty,0 ],&&&&\bR^{*}=\mathbf R\backslash 
\{0\}.
	\end{align} 

If $V$ is a real vector space, we denote its dimension by $\dim V$.
We use the notation $V_{\mathbf{C}}$ for its complexification.

If $M$ is a topological  group, $M^{0}$ denotes the connected component of the identity in $M$. 
If $M$ acts on a set $E$, and if $e\in E$ or $F\subset E$, then $M(e)\subset M$ and $M(F)\subset M$ denote the corresponding stabilizers.

\subsection{Real reductive groups}\label{sRed}
Let $G$ be a linear connected real reductive group \cite[p.~3]{Knappsemi}, and let $\theta \in  {\rm Aut}(G)$ be the Cartan 
involution. Let $K\subset G$ be the fixed point set of $\theta$ in $G$. Then $K$ 
is a maximal compact  subgroup of $G$, which is also connected.

Let $\fg$ and $\fk$ be the Lie algebras of $G$ and $K$.  Then 
$\theta$ acts as an automorphism on $\fg$, so that $\fk$ is the  
eigenspace of $\theta$ associated with the eigenvalue $1$. Let 
$\fp\subset \fg$ be the eigenspace of $\theta$ associated with the eigenvalue $-1$. We have the Cartan decomposition
	\begin{align}\label{Cartandecomp}
		\fg=\fp\oplus \fk. 
	\end{align} 
Set
	\begin{align}
&		m=\dim \fp,&n=\dim \fk. 
	\end{align} 
	
Let $B:\fg\times \fg\to \mathbf R$ be a  symmetric nondegenerate 
bilinear form, which is $G$ and $\theta$-invariant. We assume that 
$B$ is positive on $\fp$ and negative on $\fk$. Therefore, $B$ 
induces a Euclidean metric on $\fp\oplus \sqrt{-1}\fk$, which will be denoted by $|\cdot|^{2}$.

Let $\cU(\fg)$ be  the enveloping algebra of $\fg$. Let $C^{\fg}\in 
\cU(\fg)$ be the Casimir operator associated to $B$. If 
$e_{1},\ldots,e_{m+n}$ is a basis of $\fg$, and if 
$e_{1}^{*},\ldots,e_{m+n}^{*}$ is the dual basis of $\fg$ with 
respect to $B$, then 
	\begin{align}
		C^{\fg}=-\sum_{i=1}^{m+n}e_{i}^{*}e_{i}. 
	\end{align} 
Assume that $e_{1},\ldots,e_{m}$ 
is an orthonormal  basis of $\fp$ with respect to $B_{|\fp}$, and if 
$e_{m+1},\ldots,e_{m+n}$ is an orthonormal basis of $\fk$ with 
respect to $-B_{|\fk}$, then 
	\begin{align}\label{eqCfg}
		C^{\fg}=-\sum_{i=1}^{m}e_{i}^{2}+\sum_{i=m+1}^{m+n}e_{i}^{2}. 
	\end{align} 
Classically, $C^{\fg}$ is in the centre of $\cU(\fg)$. 

We define the Casimir   $C^{\fk}\in \cU(\fk)$ of $\fk$ in the same way. 
	If $E$ is a Euclidean or Hermitian space, and if $\tau^{E}:K\to {\rm GL}(E)$ is a finite dimensional  unitary  
representation of $K$. Denote by $C^{\fk,E}\in \End(E)$ the 
corresponding Casimir operator acting on $E$. If $\tau^E$ is 
irreducible, then $C^{\fk,E}\in \bR_{-}$ is a non-positive scalar. 

The group $K$ acts on $\fp$ and $\fk$ by adjoint action. Let 
$C^{\fk,\fp}\in \End(\fp)$ and $C^{\fk,\fk}\in \End(\fk)$ be the 
associated Casimir.  Set 
\begin{align}\label{eqcg}
		c_{\fg}=-\frac{1}{8}\Tr\[C^{\fk,{\fp}}\]-\frac{1}{24}\Tr\[C^{\fk,\fk}\]\in \bR_{+}. 
	\end{align} 
As we will see later in Section \ref{sGOIF1}, the constant  $c_{\fg}$ appears in Bismut's $G$-trace formula. 	

\subsection{Symmetric spaces}\label{sSyms}
Let $\omega^{\fg}\in \Omega^{1}(G,\fg)$ be the canonical 
left-invariant $1$-form on $G$ with values in $\fg$. By 
\eqref{Cartandecomp}, we have a splitting 
	\begin{align}\label{eq:omega-splitting}
		\omega^{\fg}=\omega^{\fp}+\omega^{\fk}. 
	\end{align} 

Let 
	\begin{align}
		X=G/K
	\end{align} 
be the symmetric space associated to $G$.	
The natural projection $p:G\to X$ defines a $K$-principal bundle, and $\omega^{\fk}$ is a connection form. 

Recall that $K$ acts on $\fp$ by adjoint action. The  tangent bundle $TX$ is given by
	\begin{align}
T X = G \times_{K} \fp. 		
	\end{align} 
Then $TX$ is equipped with the scalar product  induced by $B_{|\fp}$, so 
that $X$ is a Riemannian manifold. The connection $\nabla^{TX}$ on 
$TX$ induced by $\omega^{\fk}$ coincides  the Levi-Civita connection 
of $TX$, and its curvature is parallel and non-positive. Moreover,  $G$ acts 
isometrically on the left on $X$. Also, $\theta$ acts as an isometry of  $X$.

More generally, let $\tau^{E}:K\to {\rm U}(E)$ be a unitary representation of $K$. Set
	\begin{align}
		F=G\times_{K}E. 
	\end{align} 
Then, $F$ is a Hermitian vector bundle equipped with a 
connection $\nabla^{F}$ induced by $\omega^{\fk}$.  The $G$ action on 
$X$ lifts to $F$, so that if $g\in G$, the diagram 
	\begin{align}\label{eqgrs}
		\begin{aligned}
		\xymatrix{
F \ar[d] \ar[r]^{g_{*}} & F\ar[d]\\ 
X \ar[r]^g &X
} 
\end{aligned}
	\end{align} 
commutes.
We have the canonical identification of $G$-spaces 
	\begin{align}\label{eqCXFK}
		C^{\infty}(X,F)=\(C^{\infty}(G)\otimes E\)^{K}. 
	\end{align} 
	
The enveloping algebra $\cU(\fg)$ acts on $C^{\infty}(G)$ as left 
invariant differential operators. Using the fact that $C^{\fg}$ is in the centre of 
$\cU(\fg)$ and that  $K$ is connected, we see that $C^{\fg}\otimes 
1$ preserves the $K$-invariant space $\(C^{\infty}(G)\otimes 
E\)^{K}$. By \eqref{eqCXFK}, we obtain an operator $C^{\fg, X}$ 
acting on $C^{\infty}(X, F)$, which is $G$-invariant. 

We equip  $C^{\infty}(X,F)$ with the standard $G$-invariant $L^{2}$-metric. If $-\Delta^{X}$ denotes the Bochner Laplacian \cite[Definition 2.4]{BGV} associated to $\nabla^{TX},\nabla^{F}$,  the 
splitting \eqref{eqCfg} descends to  
	\begin{align}\label{eq:CgX}
		C^{\fg,X}=-\Delta^{X}+C^{\fk,E}. 
	\end{align} 
Then,  $C^{\fg,X}$ is a  self-adjoint generalised Laplacian 
\cite[Definition 2.2]{BGV}.

\subsection{Heat semigroup and  the $G$-trace}\label{sGOIF}
Let $\exp\(-tC^{\fg,X}/2\)_{|t\g0}$ be the heat semigroup of $C^{\fg, X}/2$. 
Since $C^{\fg, X}$  commutes with the $G$-action, $\exp\(-tC^{\fg,X}/2\)$ is $G$-invariant.

For $t>0,x,x'\in X$, denote $p_{t}(x,x')\in \Hom\(F_{x'}, F_{x}\)$  the corresponding  smooth integral kernel with respect to the Riemannian 
volume on $X$. 
Since $\exp\(-tC^{\fg,X}/2\)$ is $G$-invariant, if 
$g\in G$ and if $g_{*}:F_{x}\to F_{gx}$ is the obvious map, we have 
\begin{align}\label{eqptgxx}
	p_{t}(gx,gx')=g_*p_{t}(x,x')g_{*}^{-1}.
\end{align}
Since $G$ acts transitively on $X$, by \eqref{eqptgxx}, we see that $\Tr[p_{t}(x,x)]$ is independent of $x\in X$. 

\begin{defin}\label{def:cc1a}
For $ t> 0 $,  put
 \begin{align}\label{eqTrG}
	\Tr_{G}\[\exp\(-\frac{t}{2}C^{\fg,X}\)\]=\Tr\[p_{t}(x,x)\]. 	
\end{align} 
\end{defin}

\begin{re}
	By \cite[(2.13)]{AS77}, the quantity $\Tr_{G}\left[\exp\left(-tC^{\mathfrak{g},X}/2\right)\right]$ coincides with the $G$-trace of $\exp\left(-tC^{\mathfrak{g},X}/2\right)$ as introduced in \cite[Section 2]{AS77}.
  It also coincides with the orbital integral of $\exp\left(-tC^{\mathfrak{g},X}/2\right)$ at the identity element, as defined in \cite[p.~66]{Selberg56} and \cite[Section 4.2]{B09}. 
\end{re}

\begin{re}
	Since the semi-group $\exp\(-tC^{\fg,X}/2\)$ is 
	non-negative and self-adjoint, the endomorphism $p_{t}(x,x)\in \End(F_{x})$ is also non-negative and self-adjoint. 
	Then,  
	\begin{align}
		\Tr_{G}\[\exp\(-\frac{t}{2}C^{\fg,X}\)\]\g 0. 
	\end{align}
	We will see in Corollary \ref{corTrp} that this quantity is indeed positive.
\end{re}

\subsection{Bismut's formula for 
$\Tr_{G}\[\exp\(-\frac{t}{2}C^{\fg,X}\)\]$}\label{sGOIF1}
%
%

For $x\in \bR$, set 
\begin{align}
	\widehat{A}(x)=\frac{x/2}{\sinh(x/2)}. 
\end{align} 
Then, $\widehat A$ is a smooth even positive function on $\bR$. There 
exists  $C>0$ such that for $x\in \bR$,  
	\begin{align}\label{eqAx}
		 	\widehat{A}(x)\l C(1+|x|)e^{-|x|/2}. 
	\end{align} 

For a Hermitian matrix $H$, define
\begin{align}\label{eqAH12}
	\widehat{A}(H)={\det}^{1/2} \(\frac{H/2}{\sinh(H/2)}\).
\end{align}
Since $\det  \(\frac{H/2}{\sinh(H/2)}\)>0$, the square root in 
\eqref{eqAH12} is defined by the positive square root. 


 
Recall that  $\fk$ acts anti-symmetrically on $\fp$ and $\fk$. For 
$\Yok\in \sqrt{-1}\fk$,  
$\widehat{A}\(\ad\({\Yok}\)_{|\fp}\)$  and  
$\widehat{A}\(\ad\({\Yok}\)_{|\fk}\)$ are well defined positive number. By 
\eqref{eqAx}, there exist   $C_{1}>0$ and $C_{2}>0$ such that for $\Yok\in \sqrt{-1}\fk$,
	\begin{align}\label{eqJrexp}	
		\frac{\widehat{A}\(\ad\({\Yok}\)_{|\fp}\)}{\widehat{A}\(\ad\(\Yok\)_{|\fk}\)}\l C_{1} \exp\(C_{2} 
		\left|\Yok\right|\). 
	\end{align} 
	
 Denote $d\Yok$ the Lebesgue measure on the Euclidean space $\(\sqrt{-1}\fk, B_{|\sqrt{-1}\fk}\)$.  
	
 \begin{thm}\label{thmB09}
	 For $t>0$,  the following identity holds:
	\begin{multline}\label{eqTrr}
		\Tr_{G}\[\exp\(-\frac{t}{2}C^{\fg,X}\)\]=\frac{1}{(2\pi 
		t)^{m/2}}\exp\(-\frac{c_{\fg}}{2}t \)\\ \times 
		\int_{\sqrt{-1}\fk}\frac{\widehat{A}\(\ad\({\Yok}\)_{|\fp}\)}{\widehat{A}\(\ad\(\Yok\)_{|\fk}\)}\Tr\[\tau^{E}\(e^{-\Yok}\)\]\exp\(-\frac{\left|\Yok\right|^{2}}{2t}\)\frac{d\Yok}{(2\pi t)^{n/2}}. 
	\end{multline} 
 \end{thm} 
\begin{proof}
	This is Bismut's orbital integral formula \cite[Theorem 
	6.1.1]{B09} associated to the identity element. 
\end{proof}  
 
Note that by \eqref{eqJrexp}, the above integral converges. 
	
\begin{cor}	\label{corTrp}
	For $t>0$, we have 
	\begin{align}\label{eqTrG>0}
	\Tr_{G}\[\exp\(-\frac{t}{2}C^{\fg,X}\)\]>0. 
\end{align} 
\end{cor}
\begin{proof}	
If $\Yok\in \sqrt{-1}\fk$, $\tau^{E}\(\Yok\)$ acts on the Hermitian 
space $E$ as a self-adjoint operator. Therefore, 
	\begin{align}\label{eqtrp}
		\Tr\[\tau^{E}\(e^{-\Yok}\)\]>0. 
	\end{align} 
Moreover, we have 
		\begin{align}\label{eqtrp1}
		\frac{\widehat{A}\(\ad\({\Yok}\)_{|\fp}\)}{\widehat{A}\(\ad\(\Yok\)_{|\fk}\)}>0.
	\end{align} 
By \eqref{eqTrr}, \eqref{eqtrp}, and \eqref{eqtrp1}, we get our corollary. 
\end{proof}

\subsection{Motivating example of ${\rm SL}(2, \bR)$}\label{subsec:sl2}
Let $G = {\rm SL}(2, \bR)$ and $K = {\rm SO}(2)$.
We identify $ K$ with $ {\rm U}(1) $ and $ \mathfrak{k} $ with $ \sqrt{-1} \mathbf{R} $.
We choose the bilinear form $ B$, so that the induced metric on $ \sqrt{-1} \mathfrak{k} $ is the standard metric on $ \mathbf{R} $.

If $ \lambda \in \mathbf{Z} $, let $ E$ be the irreducible representation defined by $\tau^{E} : z\in {\rm U} \left(1\right) \to z^{\lambda }\in {\rm U}\left(1\right) $. 
The adjoint action of $ K$ on $ \mathfrak{p}_{\mathbf{C} } $ is a sum of two irreducible representations of $ K$ parametrized by $\lambda =  \pm 2$. 
Combining with the fact that $ \fk$ is commutative, by \eqref{eqcg}, we have
 \begin{align}\label{eq:sjsp}
	c_{\fg} = 1.
 \end{align}
Moreover, if $ x = \Yok \in \sqrt{-1} \mathfrak{k} = \mathbf{R} $, by \eqref{eqAH12}, we have 
 \begin{align}\label{eq:k2dm}
		\frac{\widehat{A}\(\ad\({\Yok}\)_{|\fp}\)}{\widehat{A}\(\ad\(\Yok\)_{|\fk}\)}= \frac{x}{\sinh\left(x\right)}.
 \end{align} 

 By the above normalisation and by \eqref{eq:k2dm}, Bismut's formula (\ref{eqTrr}) reduces to a well-known formula\footnote{The discrepancy with \cite[Section 8.3]{B09} and \cite[p.~233]{McKean72} comes form the different normalisation on the bilinear form $ B$.} (see \cite[Section 8.3]{B09} and \cite[p.~233]{McKean72}), 
 \begin{align}
 \label{eq sl2}
\Tr_{G}\[\exp\(-\frac{t}{2}C^{\fg,X}\)\]=\frac{1}{2\pi 
		t} e^{-t/2} \int_{\bR} \frac{x}{\sinh x}  e^{-\frac{x^2}{2t} - \lambda x}\; \frac{dx}{(2\pi 
		t)^{1/2}}. 
 \end{align}
In the following, we will compute the asymptotics of (\ref{eq sl2}) as $t \to \infty$ using an elementary method, which will be later generalized in the proof of the main Theorem \ref{thmit}.

The integral \eqref{eq sl2} is invariant when changing $ \lambda $ to $ -\lambda $.
We may and we will assume $ \lambda \ge 0$.
For $ w\in \left\{\pm 1\right\} $, put 
\begin{align}\label{eq:liy2}
	I^{\fg}_{t}(\lambda,w) = \frac{1}{2\pi 
	t}\int_{\bR_+} \frac{x}{\sinh x}  e^{-\frac{x^2}{2t} +w\lambda x}\; \frac{dx}{(2\pi 
	t)^{1/2}}.
\end{align}
By \eqref{eq sl2} and \eqref{eq:liy2}, we have
\begin{align}\label{eq:s2bk}
 \Tr_{G}\[\exp\(-\frac{t}{2}C^{\fg,X}\)\]= \left(I^{\fg}_{t}(\lambda,1)+I^{\fg}_{t}(\lambda,-1)\right)e^{-t/2} .
\end{align}

Put 
\begin{align}\label{eq:jnx3}
 \varrho^{\mathfrak{g} } = 1.
\end{align}
We can rewrite \eqref{eq:liy2} as 
  \begin{align} \label{sl2 It}
		I^{\fg}_{t}(\lambda,w) 
						=\frac{1}{\pi 
				t}\int_{\bR_+} \frac{x}{1-e^{-2x}}  e^{-\frac{x^2}{2t} +(w \lambda-\varrho^\fg)  x}\; \frac{dx}{(2\pi 
				t)^{1/2}}.
		\end{align}
We rescale the variable $x$ by $t$ so that,
   \begin{align}
I^{\fg}_{t}(\lambda,w) 
        =\frac{1}{\pi} t^{1/2}  e^{\frac{t}{2}\left(w \lambda-\varrho^\fg\right)^2}  \int_{\bR_+} \frac{x}{1-e^{-2tx}} e^{-\frac{t}{2}\left(x -\left(w \lambda-\varrho^\fg\right) \right)^2} \; \frac{dx}{(2\pi)^{1/2}}.
 \end{align}
Note that the function $ \frac{x}{1-e^{-2tx}}$ is positive on $\bR_+$.

When $ t\to \infty$, the asymptotics of the above integral can be evaluated by Laplace's method.
The leading term of the asymptotics is localised near the minimum point of the action 
\begin{align}\label{eq:3jhr}
 x\in \bR_{+} \to \frac{1}{2} \left(x -\left(w \lambda-\varrho^\fg\right) \right)^2.
\end{align}
Let us discuss case by case according to the location of $w \lambda-\varrho^\fg$.

\underline{Case I : $w \lambda-\varrho^\fg > 0 $}. 
In this case, the minimum point of \eqref{eq:3jhr} is $ w \lambda-\varrho^\fg$.
Let us compute the large-time behavior of the integral in three steps.
\begin{enumerate}[\indent 1)] 
\item  Fix a parameter $0<\epsilon <w \lambda-\varrho^\fg $.
We define 
\begin{align}
 I^{\fg}_{\epsilon, t}(\lambda,w) =\frac{1}{\pi} t^{1/2}  e^{\frac{t}{2}\left(w \lambda-\varrho^\fg\right)^2}  \int_{w \lambda-\varrho^\fg-\epsilon}^{w \lambda-\varrho^\fg+\epsilon}  \frac{x}{1-e^{-2tx}}  e^{-\frac{t}{2}\left(x -\left(w \lambda-\varrho^\fg\right) \right)^2} \;  \frac{dx}{(2\pi)^{1/2}}.
\end{align}
As $ t \to \infty$,  we can localise the integral $ I^{\fg}_{t}(\lambda,w)$, so that 
\begin{align}
I^{\fg}_{t}(\lambda,w)  \sim I^{\fg}_{\epsilon, t}(\lambda,w).
\end{align}
\item We change the variable $x = w \lambda-\varrho^\fg +\frac{y}{t^{1/2}}$,
\begin{align}
 I^{\fg}_{\epsilon, t}(\lambda,w)  =\frac{1}{\pi}  e^{\frac{t}{2}\left(w \lambda-\varrho^\fg\right)^2}  \int_{-\sqrt{t}\epsilon}^{\sqrt{t}\epsilon}  \frac{ w \lambda-\varrho^\fg + \frac{y}{\sqrt{t}}}{1-e^{-2t\left( w \lambda-\varrho^\fg + \frac{y}{\sqrt{t}}\right)}} \cdot e^{-\frac{1}{2}y^{2} } \; \frac{dy}{(2\pi)^{1/2}}.
\end{align}
\item As $t \to \infty$, we have 
 \begin{align}
 I^{\fg}_{\epsilon, t}(\lambda,w)  \sim & \frac{1}{\pi} \cdot  e^{\frac{t}{2}\left(w \lambda-\varrho^\fg\right)^2}   \int_{\bR}  \left(w \lambda-\varrho^\fg\right) \cdot e^{-\frac{1}{2}y^{2} } \; \frac{dy}{(2\pi)^{1/2}}\\
 =   & \frac{1}{\pi} \left(w \lambda-\varrho^\fg\right)  e^{\frac{t}{2}\left(w \lambda-\varrho^\fg\right)^2}.\notag
\end{align}
\end{enumerate}

 \underline{Case II : $w \lambda-\varrho^\fg =  0 $}.
In this case, the minimum point of \eqref{eq:3jhr} is $ w \lambda-\varrho^\fg= 0 $.
We can perform a similar localisation as before.
\begin{enumerate}[\indent 1)]
	\item For a fixed parameter $\epsilon > 0$, put
	\begin{align}
I^{\fg}_{\epsilon, t}(\lambda,w)  =\frac{t^{1/2}}{\pi}  \int_{0}^{\epsilon}  \frac{x}{1-e^{-2tx}} e^{-\frac{t}{2}x^2} \; \frac{dx}{(2\pi)^{1/2}}.
	\end{align}
	As $ t\to \infty$,  
	\begin{align}\label{eq:qrfa}
		I^{\fg}_{t}(\lambda,w)  \sim I^{\fg}_{\epsilon, t}(\lambda,w).
	\end{align}
	\item We  rescale the coordinate $x = \frac{y}{t^{1/2}}$, and obtain
	\begin{align}
I^{\fg}_{\epsilon, t}(\lambda,w)  = \frac{1}{\pi t^{1/2}} \int_{0}^{\sqrt{t}\epsilon}  \frac{y}{1-e^{-2\sqrt{t}y}} \cdot e^{-\frac{1}{2}y^{2} } \; \frac{dy}{(2\pi)^{1/2}}.
 \end{align}
 \item As $ t \to \infty$, by an easy computation, we have 
 \begin{align}\label{eq:pnq1}
	I^{\fg}_{\epsilon, t}(\lambda,w) \sim \frac{1}{\pi}  \left\{\int_{\bR_+} y e^{-\frac{1}{2}y^{2} } \; \frac{dy}{(2\pi)^{1/2}} \right\} t^{-1/2}= \frac{1}{\sqrt{2} \pi^{3/2} } t^{-1/2} .
 \end{align}
\end{enumerate}

\underline{Case III : $w \lambda-\varrho^\fg <  0 $}.  We can directly apply  the dominant convergence theorem to (\ref{sl2 It}): as $ t \to \infty$, 
\begin{align}
		I^{\fg}_{t}(\lambda,w)  \sim   &\frac{1}{\pi 
t}\int_{\bR_+} \frac{x}{1-e^{-2x}}  e^{(w \lambda-\varrho^\fg)  x}\; \frac{dx}{(2\pi 
t)^{1/2}}\\
		=&\frac{1}{2\pi}\left\{\int_{\bR_+} \frac{x}{\sinh x}  e^{w \lambda  x}\; \frac{dx}{(2\pi)^{1/2}} \right\} t^{-3/2}. \notag
\end{align}
 
Theorem \ref{thmit} for $ G= {\rm SL}_{2} \left(\mathbf{R} \right)$ can be deduced easily from the above computations.

\begin{prop}\label{prop:drep}
 There exist constants $ \alpha\in \mathbf{R}^{*}_{+}, \beta\in -\frac{1}{2} \mathbf{N} , \gamma\in \mathbf{R}  $ such that as $ t \to \infty$,
 \begin{align}\label{eqsl22}
	\Tr_{G}\[\exp\(-\frac{t}{2}C^{\fg,X}\)\] \sim \alpha  t^{\beta}  e^{\left(\gamma-1/2\right)t}.
	\end{align}
	Moreover, 
	\begin{enumerate}
		\item when $\lambda \geq 2$, 
		\begin{align}\label{case1}
			 &\alpha =  \frac{\lambda-\varrho^\fg}{\pi }, &\beta =0,  && \gamma =  \frac{1}{2}\left(\lambda-\varrho^\fg\right)^2;
		\end{align}
		\item when $\lambda = 1$, 
		\begin{align}
			 &\alpha = \frac{1}{\sqrt{2 }\pi^{3/2} }, & \beta =  -\frac{1}{2},  && \gamma = 0;
		\end{align}
		\item when $\lambda =0$, 
		\begin{align}\label{case3}
			& \alpha = \frac{1}{2\pi } \int_{\bR} \frac{x}{\sinh x} \; \frac{dx}{(2\pi)^{1/2}}= \frac{\pi^{1/2} }{4 \sqrt{2}}, & \beta =  -\frac{3}{2},  && \gamma =0. 
		\end{align}
		\end{enumerate}	
\end{prop}
\begin{proof} 
 If $ \lambda \ge 2$,  then $ \lambda - \varrho^{\fg} > 0  $ and $ -\lambda - \varrho^{\fg} < 0 $. 
 By \eqref{eq:s2bk}, only the term $ I^{\mathfrak{g} }_{t} \left(\lambda, 1\right)$ in case (I) will contribute to our asymptotics \eqref{eqsl22}. 
 If $ \lambda = 1$, then $ \lambda - \varrho^{\fg} =  0  $ and $ -\lambda - \varrho^{\fg} < 0 $, so that only the term $ I^{\mathfrak{g} }_{t} \left(\lambda, 1\right)$ in case (II) will contribute to the asymptotics.
 If $ \lambda = 0  $, then $ \pm \lambda - \varrho^{\fg} <  0  $,  both $ I^{\mathfrak{g} }_{t} \left(\lambda, \pm 1\right)$ in case (III) will contribute to the asymptotics.
 Our proposition now follows directly from the previous computations.
\end{proof}

 
As the notations indicate, the group $ \left\{\pm 1\right\}$ is the Weyl group, $ \mathbf{R}_{+} $ is a Weyl chamber, $ \rho^\fg $ is the half sum of positive roots associated to the Lie algebra $ \fg = \mathfrak{sl}(2, \mathbf{R})$.
Also, $ \mathbf{Z} $ is the weight lattice, and $ \lambda\in \mathbf{Z} $ is the (highest) weight of the representation $ \tau^{E}$.
Here, three distinct cases are based on the position of $ w \lambda - \rho^{\fg} $  relative to the Weyl chamber $ \mathbf{R}_{+} $ : (I) inside the chamber, (II) on its wall, and (III) outside the chamber. For a general real reductive Lie group $G$, these cases persist but can occur simultaneously in different directions due to the higher-dimensional structure of the Weyl chamber.


\section{Root systems and related constructions} \label{SRoot}
In this section, we review some basic construction related to a root system. 
Most of the results are well-known in Lie theory. 
We have recalled them in this section for the convenience of the readers. 
But we have omitted most of the proofs. 
 
%


%
%
%
%

This section is organised as follows.  In Section \ref{sMax}, we 
introduce the Cartan subalgebra $\ft$ of   $\fk$, the root system $R(\fk)$ of $\fk$, and $\mathfrak t$-restricted root 
system $R(\fg)$ of $\fg$, and the corresponding  Weyl groups $W(\fk)$, $W(\fg)$. 

In Section \ref{sPRS}, we introduce the positive  root systems 
$R_{+}(\fk)$, $R_{+}(\fg)$ and the associated positive Weyl chambers 
$C_{+}(\fk)$, $ C_{+}(\fg)$. We review also some classic geometric properties  
related to the Weyl chambers and Weyl groups. 

In Section \ref{sPSA}, we define various subspace of $\sqrt{-1}\ft$ associated to a 
subset $\Delta_{0}^{1}$ of a system of simple roots $\Delta_{0}^{\fg}$ in $ R_{+} \left(\mathfrak{g} \right)$. 

In Section \ref{sLeviPara}, we review the theory of 
$\theta$-invariant parabolic subalgebras of $\fg_{\bC}$ associated to $\Delta_{0}^{1}$. 

In Section \ref{sSubK1}, we introduce a semisimple subgroup $K^{1}_{s}$ of $K$ associated to $\Delta_{0}^{1}$. Given an irreducible representation of $\tau^{E}$ of $K$, we construct a corresponding  irreducible representation $\tau^{E,1}$ of $K^{1}_{s}$. 

In Section \ref{sLMAP}, we recall Langlands' combinatorial Lemma  following  Carmona \cite[Section 1]{Carmona83}.

Finally, in Section \ref{sVoganLambda}, we introduce the Lambda map of Vogan 
following Carmona \cite[Section 2]{Carmona83}.

\subsection{Maximal torus and root decompositions}\label{sMax}
Let $T$ be a maximal torus of $K$. Let $\ft\subset \fk$ be the Lie algebra of $T$.  Then, $\ft$ is a Cartan subalgebra of $\fk$. 

Recall that $\(\sqrt{-1}\ft, B_{|\sqrt{-1}\ft}\)$ is a Euclidean space. Let 
$\sqrt{-1}\ft^{*}$ be the dual space with induced Euclidean metric. Let $R(\fk)\subset \sqrt{-1}\ft^{*}$ be the root system\footnote{It 
is an abstract root system on the real subspace of  $\sqrt{-1}\ft^{*}$ spanned by $R(\fk)$.} of $\fk$ with respect to  
$\ft$. If $\alpha\in R(\fk)$, denote by $\fk_{\alpha}\subset 
\fk_{\bC}$ the root space associated to $\alpha$. We have the root  decomposition 
	\begin{align}\label{eqkTwei}
		\fk_{\bC}=\ft_{\bC}\oplus \bigoplus_{\alpha\in R(\fk)}\fk_{\alpha}.
	\end{align} 

We can similarly define $R(\fp)$ and $R(\fg)$ and use the corresponding obvious notation. Clearly, 
	\begin{align}
		R(\fg)=R(\fp)\cup R(\fk). 
	\end{align} 
Let $\fa\subset \fp$ be the centraliser of $\ft$ in $\fp$. Then,  
	\begin{align}\label{eqpTwei}
		&\fp_{\bC}=\fa_{\bC}\oplus \bigoplus_{\alpha\in R(\fp)}\fp_{\alpha},&\fg_{\bC}=\fa_{\bC}\oplus\ft_{\bC}\oplus \bigoplus_{\alpha\in 
		R(\fg)}\fg_{\alpha}. 
	\end{align} 
If $\alpha\in R(\fg)$, by \eqref{eqkTwei} and \eqref{eqpTwei}, we have 
	\begin{align}\label{eqgapaka}
	\fg_{\alpha}=\fp_{\alpha}\oplus \fk_{\alpha}. 
	\end{align} 
	
\begin{re}\label{re:hat}
	If $\fh=\fa\oplus \ft$, then $\fh$ is a $\theta$-invariant Cartan subalgebra of 
	$\fg$ \cite[p.~129]{Knappsemi}. 
\end{re}
	
The following three propositions are well-known \cite{Vogan-book,Vogan79}.	
    
\begin{prop}\label{propRg}
	If $\alpha\in R(\fp)$, we have $\dim \fp_{\alpha}=1$. If $\alpha\in R(\fg)$, we have 
	\begin{align}
		\dim \fg_{\alpha}=\begin{cases}
		2, &\alpha\in R(\fp)\cap R(\fk),\\
		1,&\text{otherwise}. 
		\end{cases}
	\end{align} 	
\end{prop}

 \begin{prop}\label{thmRg}
The set $R(\fg)$ forms an abstract root system on the real span of 
$R(\fg)$. 
 \end{prop}
	
Let $W(\fk), W\left(\mathfrak{g} \right)$ be the Weyl groups of $R(\fk), R\left(\mathfrak{g} \right)$. 	
If $\alpha \in 
R(\fg)$, denote by $s_{\alpha} $ the reflection on $ \sqrt{-1} \mathfrak{t}^{*} $ associated to $\alpha$.	
Then, $ W\left(\fk\right)$ and $ W\left(\mathfrak{g} \right)$ are generated respectively by the reflections $s_{\alpha}$ with $\alpha\in R(\fk)$ and $\alpha\in R(\fg)$.
Since $R(\fk)\subset R(\fg)$, we have  
	\begin{align}
		W(\fk)\subset W(\fg). 
	\end{align} 

\begin{prop}\label{propWgainv}
The function  $\alpha\in R(\fg)\to \dim \fg_{\alpha}$ is $W(\fg)$-invariant. 
\end{prop} 	
	
 \begin{re}
	 The root system $R(\fg)$ is possibly  non reduced. Indeed, if  
	 $\fg=\mathfrak{sl}(3,\bR)$ and  
	$\fk=\mathfrak{so}(3,\bR)$, we can take $\ft=\mathfrak{so}(2,\bR)$. If 
	$R(\fk)=\{\pm \alpha\}$, then   $R(\fp)=\{\pm \alpha, \pm 2\alpha\}$. In 
	particular, 
	\begin{align}
	R(\fg)=	\{\pm \alpha, \pm 2\alpha\}.
	\end{align} 
\end{re}

\begin{re}
	The sets $R(\fk), R(\fg)$ are   abstract root systems on their 
	own real spans. These two real spans do not always coincide. Indeed, if $\fg=\mathfrak{sl}(2,\bR)$ and 	
	$\fk=\mathfrak{so}(2,\bR)$, we have $\ft=\fk$, so $R(\fk)$ is 
	empty which spans $\{0\}$. However,  $R(\fg)$ spans $\sqrt{-1}\ft^{*}$.
\end{re}

In the  sequel, we will denote 
	\begin{align}
		\ft_{0}=\sqrt{-1}\ft.
	\end{align} 
Then, $ \left(\mathfrak{t}_{0}, B_{|\mathfrak{t}_{0} } \right)$ is a Euclidean space.	
We will identify $\ft_{0}$ with $\ft_{0}^{*}=\sqrt{-1}\ft^{*}$ via the metric $B_{|\ft_{0}}$. 

\subsection{Positive root systems}\label{sPRS}
Let $R_{+}(\fk)\subset R(\fk)$ be a positive root system of $R(\fk)$. 
Let $C_{+}(\fk)\subset \ft_{0}$ be the closed positive Weyl chamber of $R_{+}(\fk)$. Then,
	\begin{align}\label{eqCk}
		C_{+}(\fk)=\left\{Y_{0}^{}\in \ft_{0}: \text{ for all 
		}\alpha\in R_{+}(\fk), \left\langle \alpha, 
		Y_{0}^{}\right\rangle\g 0 \right\}. 
	\end{align} 
We  denote by ${\rm Int}\(C_{+}(\fk)\)$ the interior of $C_{+}(\fk)$, which is an open cone. 

	The group $W(\fk)$ acts on $\ft_{0}$, so that $C_{+}(\fk)$ is a 
	fundamental domain. 
We have 
	\begin{align}\label{eqt0wk}
		\ft_{0}=\bigcup_{w\in W(\fk)} w^{-1}C_{+}(\fk),
	\end{align} 
where the intersection of two different closed Weyl chambers is negligible. 

Set 
	\begin{align}
		\varrho^{\fk}=\frac{1}{2}\sum_{\alpha\in R_{+}(\fk)}\alpha\in 	\ft_{0}. 
	\end{align} 
The $\pi$-function for $R_{+}(\fk)$  is  a real polynomial on 
$\ft_{0}$  defined  for $Y_{0}^{}\in \ft_{0}$ by 
\begin{align}\label{eqpikh}
\pi^{\fk}\(Y_{0}^{}\)=\prod_{\alpha\in 
R_{+}(\fk)}\left\<\alpha,Y_{0}^{}\right\>.	
\end{align}

%
If $w\in W(\fk)$, then $w$ acts  isometrically on $\ft_{0}$. Set
	\begin{align}
		\e_{w}=\det \(w_{|\ft_{0}}\)\in \{\pm1\}. 
	\end{align} 
Classically,
\begin{align}\label{eqew}
		\pi^{\fk}\(wY_{0}^{}\)=\e_{w}\pi^{\fk}\(Y_{0}^{}\). 
\end{align} 

We choose a compatible positive root system  $R_{+}(\fg)\subset R(\fg)$ for $R(\fg)$ in the sense that
\begin{align}\label{eqRkR+}
	R_{+}(\fk)\subset R_{+}(\fg). 
\end{align}
It is worth noting that a compatible positive root system $R_+(\fg)$ always exists.
	
Let $C_{+}(\fg)\subset \ft_{0}$ be the positive closed Weyl Chambre of $R_{+}(\fg)$, i.e.,
	\begin{align}\label{eqCg}
		C_{+}(\fg)=\left\{Y_{0}^{}\in \ft_{0}: \text{ for all 
		}\alpha\in R_{+}(\fg), \left\langle \alpha, 	Y_{0}^{}\right\rangle\g 0 \right\}.
	\end{align} 
As in \eqref{eqt0wk}, we have a decomposition 
	\begin{align}\label{eqt0wg}
		\ft_{0}=\bigcup_{w\in W(\fg)} w^{-1}C_{+}(\fg),
	\end{align} 
and the intersection of two different closed Weyl chambers is negligible.

By \eqref{eqCk},  \eqref{eqRkR+}, and \eqref{eqCg}, we have 
	\begin{align}
		C_{+}(\fg)\subset C_{+}(\fk). 
	\end{align} 
	
Set  	
	\begin{align}
		W(\fg, \fk)=\left\{w\in W(\fg): w^{-1}C_{+}(\fg)\subset C_{+}(\fk)\right\}. 
	\end{align} 	
By \eqref{eqt0wk} and \eqref{eqt0wg}, we have 
	\begin{align}
		C_{+}(\fk)=\bigcup_{w\in W(\fg, \fk)}w^{-1}C_{+}(\fg). 
	\end{align} 	
If $w\in W(\fg, \fk)$, using $w^{-1}C_{+}(\fg)\subset 
C_{+}(\fk)$, we see that
$R_{+}(\fk)$ is nonnegative on $w^{-1}C_{+}(\fg)$, so that 
	\begin{align}\label{eqwR+kR+g}
		wR_{+}(\fk)\subset R_{+}(\fg). 
	\end{align} 
And therefore, 
	\begin{align}\label{eqwR+kR+g1}
		R_{+}(\fk)=R(\fk) \cap w^{-1}R_{+}(\fg). 
	\end{align} 
	
In general, $W(\fg, \fk)$ is not a group. 
It is a system of representatives of the quotient space $W(\fg)/W(\fk)$. 
Indeed, if $w\in W(\fg)$, there is a unique $w_{2}\in W(\fk)$ sends $w^{-1}C_{+}(\fg)$ into $C_{+}(\fk)$, and a unique $w_{1}\in W(\fg, 
\fk)$ such that $w_{1}^{-1}C_{+}(\fg)=w_{2}w^{-1}C_{+}(\fg)$. This 
way gives a unique decomposition 
	\begin{align}\label{eqwww12}
		&w=w_{1}w_{2}, & w_{1}\in W(\fg,\fk), &&w_{2}\in 
		W(\fk).
	\end{align}

	
	


Put
	\begin{align}\label{eqCgchek}
\check C_{+}(\fg)	=\left \{u\in \ft_{0}: \text{ for all } Y_{0}^{}\in 
C_{+}(\fg), \left\langle u,Y_{0}^{}\right\rangle \g 0\right\}. 
	\end{align} 
 
The following two propositions are elementary consequences of the positivity of the cones and Chevalley's Lemma \cite[Proposition 2.72]{KnappLie}. 
We omit the proof.  
\begin{prop}\label{Propv2v2}
	If $u\in C_+(\fg)$ and if $v\in u+ \check C_+(\fg)$, then 
	\begin{align}
		|v|\g \left|u\right|,
	\end{align} 
	where the equality holds if and only if $u=v$. 
\end{prop} 	

\begin{prop}\label{propuvC}	
If $u,v\in C_{+}(\fg)$, and if $w\in W(\fg)$, then   
	\begin{align}\label{eqw-}
		\<u,v\>\g \<u,wv\>.
	\end{align} 
The	equality holds if and only if there exist $w',w''\in W(\fg)$ such that 
\begin{align}\label{eqwww}
	&w=w'w'',	&w'u=u,&&w''v=v. 
	\end{align} 
In this case, $ w^{\prime } $ is generated by the reflection $s_\alpha$ where $ \alpha\in R\left(\mathfrak{g} \right)$ such that 
 		\begin{align}
 		\<\alpha,u\>=0,
 	\end{align} 
 	and $w''$ is generated by $s_{\alpha}$ where $ \alpha\in R\left(\mathfrak{g} \right)$ such that 
		\begin{align}
  \<\alpha,v\>=0. 
 	\end{align} 
\end{prop} 	

Set 
	\begin{align}\label{eqrhog}
		\varrho^{\fg}=\frac{1}{2}\sum_{\alpha\in 
		R_{+}(\fg)}\dim \fg_{\alpha}\cdot \alpha
        = \frac{1}{2}\sum_{\alpha\in 
		R_{+}(\fp)} \alpha+\frac{1}{2}\sum_{\alpha\in 
		R_{+}(\fk)} \alpha \in \ft_{0}. 
	\end{align}

	\begin{re}\label{rerhott}
    Since a root $ \alpha $ of $ \left(\mathfrak{g}, \mathfrak{h} \right)$ has non zero restriction on $ \mathfrak{t} $, and since the vector $ \left(\pm \alpha_{|\mathfrak{a} }, \pm \alpha_{\mathfrak{t} } \right)$ are root of  $ \left(\mathfrak{g}, \mathfrak{h} \right)$, we see that $ R_{+} \left(\fg\right)$ induces a system of positive roots $ R_{+} \left(\mathfrak{g}, \mathfrak{h} \right)$ for $ \left(\mathfrak{g}, \mathfrak{h} \right)$.
		Also, the associated $\varrho$-vector are 
		\begin{align}\label{eq:bgr3}
		 \left(0, \varrho^{\mathfrak{g} } \right).
		\end{align}
		By \cite[(2.47)]{BSh22}, \eqref{eqcg}, and \eqref{eq:bgr3}, we have 
		\begin{align}\label{eq:zh3k}
		 c_{\mathfrak{g} } = \left| \varrho^{\mathfrak{g} } \right|^{2}. 
		\end{align}
		
\end{re}

If $w\in W(\fg)$, 
$w^{-1}R_{+}(\fg)$ is another positive root system. By Proposition 
\ref{propWgainv},  the corresponding $\varrho^{\fg}$ vector is 
$w^{-1}\varrho^{\fg}$.

\subsection{Subsets of $\Delta_{0}^{\fg}$}\label{sPSA}
Recall that  $R_{+}(\fg)\subset R(\fg)$ is a positive root system of $R(\fg)$, and that  $\Delta_{0}^{\fg}\subset R_{+}(\fg)$ is the associated system of simple roots. Notation and convention of this section follows \cite[Section 1.2]{LabesseWaldspruger13}. 

Let $\ft_{0}^{\fg}$ be the subspace of $\ft_{0}$ 
spanned by $\Delta_{0}^{\fg}$. 
Let $\ft_{\fg}$ be the orthogonal complement of $\ft_{0}^{\fg}$ in $\ft_{0}$, so that 
\begin{align}\label{eqft012}
\ft_{0}=\ft_{0}^{\fg}\oplus \ft_{\fg}. 
\end{align}
Clearly, $ \ft_{\mathfrak{g} } $ is just the compact part of the centre of $ \mathfrak{g} $.

Let $\Delta_{0}^{1}\subset \Delta_{0}^{\fg}$ be a subset of $\Delta_{0}^{\fg}$. 
Let $\ft_{0}^{1}$ be the subspace of $\ft_{0}^{\fg}$ spanned by  $\Delta_{0}^{1}$. 
Let $\ft_{1}^{\fg}$ be the orthogonal complement  $\ft_{0}^{1}$ in $\ft^{\fg}_{0}$, so that 
\begin{align}\label{eqft012g}
\ft_{0}^{\fg}=\ft_{0}^{1}\oplus \ft_{1}^\fg.
\end{align}

Let $\Delta_{0}^{2}\subset \Delta_{0}^{\fg}$ be another subset of $\Delta_{0}^{\fg}$ such that $\Delta_{0}^{1}\subset \Delta_{0}^{2}$. 
Then, $\ft_{0}^{1}\subset \ft_{0}^{2}$. Let $\ft_{1}^{2}\subset \ft_{0}^{2}$ be the orthogonal complement of $\ft_{0}^{1}$ in $ \ft_{0}^{2}$.
Then, we have the orthogonal decomposition 
\begin{align}\label{eqt012orth}
\ft_{0}^{\fg}=\ft_{0}^{1}\oplus \ft_{1}^{2}\oplus \ft_{2}^{\fg}. 
\end{align} 
Let $P_{0}^{1}$, $P_{1}^{2}$, and $P_{2}^{\fg}$ be respectively the corresponding orthogonal projections onto the above three spaces.

Note that $\Delta_{0}^{\fg}$ forms a basis of $\ft_{0}^{\fg}$. 
Denote by $\check\Delta_{0}^{\fg}$ the basis of $\ft_{0}^{\fg}$ which is dual to $\Delta_{0}^{\fg}$ with respect to the Euclidean metric. 
If $\alpha\in \Delta_{0}^{\fg}$, the corresponding element in $\check \Delta_{0}^{\fg}$ will be 
denoted by $\omega_{\alpha}$, so that for $\alpha,\beta\in 
\Delta_{0}^{\fg}$, 
\begin{align}
\left\langle \alpha,\omega_{\beta}\right\rangle = 
\begin{cases}
	1, &\text{if } \alpha=\beta,\\
	0, &\text{otherwise.}
\end{cases}
\end{align}

\begin{defin}\label{defDelta12}
	Set 
\begin{align}
&\Delta_{1}^{2}=\left\{P_{1}^{2}\alpha: \alpha\in \Delta_{0}^{2}\backslash 
\Delta_{0}^{1}\right\}, &\check\Delta_{1}^{2}=\left\{P_{1}^{2}\omega_\alpha: \alpha\in \Delta_{0}^{2}\backslash 
\Delta_{0}^{1}\right\}. 
\end{align} 
In particular, we define $\check\Delta_{0}^{1}, \Delta_{1}^{\fg}, 
\check \Delta_{1}^{\fg}$ by considering  the pair $\(\varnothing, 
\Delta_{0}^{1}\)$ or $\(\Delta_{0}^{1}, \Delta_{0}^{\fg}\).$
\end{defin} 

A simple argument in linear algebra shows that $\Delta_{1}^{2}$ and 
$\check\Delta_{1}^{2}$ are two bases of $\ft_{1}^{2}$. 

 Since 
$\ft_{0}^{2}$ is generated by $\Delta_{0}^{2}$, if $\alpha\in 
\Delta_{0}^{2}$, we have 
	\begin{align}\label{eqan01}
P_{0}^{2}\alpha=\alpha. 
	\end{align} 	
 If $\alpha\in \Delta_{0}^{\fg}\backslash \Delta_{0}^{1}$, then $\omega_\alpha$ is orthogonal to $\Delta_{0}^{1}$, and therefore 
contains in $\ft_{1}^{\fg}$, so that 
	\begin{align}\label{eqan02}
P_{1}^{\fg}\omega_\alpha=\omega_\alpha.		
	\end{align} 
By \eqref{eqan01} and \eqref{eqan02}, if $\alpha,\beta\in 
\Delta_{0}^{2}\backslash \Delta_{0}^{1}$, we have 
\begin{align}
\left\langle P_{1}^{2}\alpha, P_{1}^{2}\omega_{\beta}\right\rangle =\left\langle P_{0}^{2}\alpha, 
P_{1}^{\fg}\omega_{\beta}\right\rangle =\left\langle \alpha, \omega_{\beta}\right\rangle. 
\end{align} 
Thus, $\Delta_{1}^{2}$ and $\check\Delta_{1}^{2}$  are dual to 
each other. 

We have the classical well-known results \cite[Lemmas 1.2.4-1.2.6]{LabesseWaldspruger13}.
\begin{prop}\label{propobact}
	Given $\Delta_{0}^{1}, \Delta_{0}^{2}\subset \Delta_{0}^{\fg}$ 
	such that $\Delta_{0}^{1}\subset  \Delta_{0}^{2}$, then 
	 $\Delta_{1}^{2}$  forms an obtuse basis of $\ft_{1}^{2}$, i.e., if $\alpha, 
	 \alpha'\in \Delta_{1}^{2}$ with $\alpha\neq\alpha'$, then
	 	\begin{align}\label{eqaa-0}
		\left\langle \alpha,\alpha'\right\rangle\l 0.  
	\end{align} 
	 Moreover, $\check\Delta_{1}^{2}$ forms an  acute basis of 
	 $\ft_{1}^{2}$, i.e., if 
	 $\omega,\omega'\in \check\Delta_{1}^{2}$, then 
	 	\begin{align}
		\left\langle \omega,\omega'\right\rangle \g 0. 
	\end{align} 
\end{prop} 

\begin{defin}
	Set 
		\begin{align}\label{eqC12}
			\begin{aligned}
		C^{2}_{1}&=\left\{Y\in \ft_{1}^{2}: \text{ for all }\alpha\in 
		\Delta_{1}^{2},  \left\langle 
		\alpha,Y\right\rangle \g 0\right\},\\
		\check C^{2}_{1}&=\left\{Y\in \ft_{1}^{2}: \text{ for all 
		}\omega\in 
		\check\Delta_{1}^{2},  \left\langle 
		\omega,Y\right\rangle \g 0\right\}.
		\end{aligned}
	\end{align} 
\end{defin} 

Then, $C^{2}_{1}$ is an acute cone generated by nonnegative  linear 
combination of $\check\Delta_{1}^{2}$, and $\check C^{2}_{1}$ is 
an obtuse cone generated by nonnegative  linear combination of $\Delta_{1}^{2}$.
By \eqref{eqan01} and \eqref{eqan02}, we have
	\begin{align}\label{eqC12g0}
		&\check C^{1}_{0}\subset \check C^{2}_{0},& 	
		C_{2}^{\fg}\subset C_{1}^{\fg}.
	\end{align} 
By \eqref{eqCg}, \eqref{eqCgchek}, and \eqref{eqC12}, we have 
	\begin{align}\label{eqCgot}
		&C_{+}(\fg)=C^{\fg}_{0}\times \ft_{\fg},&\check 
		C_{+}(\fg)=\check  C^{\fg}_{0}. 
	\end{align}

If $Y_{0}^{\fg}=\sum_{\alpha\in \Delta_{0}^{\fg}}y^{\alpha}\omega_{\alpha}\in 
\ft_{0}^{\fg}$, write 
\begin{align}\label{eqy123}
& Y_{1}=\sum_{\alpha\in \Delta_{0}^{1}} 
	y^{\alpha}\omega_{\alpha}, &Y_{2}=	\sum_{\alpha\in 
	\Delta_{0}^{2}\setminus \Delta_{0}^{1}}y^{\alpha} 
	{\omega_{\alpha}}, &&Y_{3}=\sum_{\alpha\in \Delta_{0}^{\fg}\setminus 
	\Delta_{2}^{\fg}}y^{\alpha} {\omega_{\alpha}}, 
\end{align} 
so that 
	\begin{align}\label{eqY0123g}
			Y_{0}^{\fg}=Y_{1}+Y_{2}+Y_{3}. 
	\end{align} 		
Using  $\(Y_{0}^{1}, Y_{1}^{2}, 
Y_{2}^{\fg}\)=Y_{0}^{\fg}$ and \eqref{eqY0123g}, we get a linear map 
	\begin{align}\label{eqpyFG}
\(Y_{0}^{1}, Y_{1}^{2}, Y_{2}^{\fg}\)\to 
\(P_{0}^{1}Y_{1},P_{1}^{2}Y_{2}, Y_{3}\)
\end{align} 
on $\ft_{0}^{1}\oplus \ft_{1}^{2}\oplus 
	\ft_{2}^{\fg} $.

Let $dY_{0}^{\fg}, dY_{0}^{1},dY_{1}^{2}$,  and  $dY_{2}^{\fg}$ be 
the Euclidean volumes on $\ft_{0}^{\fg}$, $\ft_{0}^{1}$,  $\ft_{1}^{2}$, and $\ft_{2}^{\fg}$. Then, 
	\begin{align}
		dY_{0}^{\fg}=  dY_{0}^{1} dY_{1}^{2}dY_{2}^{\fg}. 
	\end{align}

\begin{prop}\label{propintVVp}
	The linear map \eqref{eqpyFG}	is a volume preserving 
	isomorphism of $\ft_{0}^{1}\oplus \ft_{1}^{2}\oplus 
	\ft_{2}^{\fg}$. 
\end{prop} 
\begin{proof}
	By \eqref{eqan02}, we have 
	\begin{align}
&P_{0}^{1}Y_{0}^{\fg}=P_{0}^{1}Y_{1}, 
&	P_{1}^{2}Y_{0}^{\fg}=P_{1}^{2}(Y_{1}+Y_{2}),  
&&	P_{2}^{\fg}Y_{0}^{\fg}=P_{2}^{\fg}(Y_{1}+Y_{2})+Y_{3}. 
\end{align} 
Therefore,  we have 
	\begin{align}\label{eqY123id123}
	\(Y_{0}^{1}, Y_{1}^{2}, Y_{2}^{\fg}\)=\(P_{0}^{1}Y_{1}, 
	P_{1}^{2}(Y_{1}+Y_{2}), P_{2}^{\fg}(Y_{1}+Y_{2})+Y_{3}\). 
\end{align} 

Since $\check \Delta_{0}^{1}$ forms a basis of $\ft_{0}^{1}$, we see 
that $P_{0}^{1}Y_{1}\to P_{1}^{2}Y_{1}$, $P_{0}^{1}Y_{1}\to P_{2}^{\fg}Y_{1}$ are well-defined linear maps. 
By \eqref{eqY123id123}, we see that 
	\begin{align}
	\(Y_{0}^{1}, Y_{1}^{2}, Y_{2}^{\fg}\)\to \(P_{0}^{1}Y_{1}, P_{1}^{2}Y_{2}, P_{2}^{\fg}Y_{2}+Y_{3}\)
\end{align} 
is a volume preserving isomorphism of $\ft_{0}^{1}\oplus \ft_{1}^{2}\oplus 
	\ft_{2}^{\fg}$.  By a similar arguments using $P_{1}^{2}Y_{2}\to P_{2}^{\fg}Y_{2}$, we get our proposition. 
\end{proof} 

\subsection{Levi subalgebras and parabolic subalgebras}\label{sLeviPara}
Let us recall the theory of standard $\theta$-invariant parabolic subalgebras, which is an analogue of the real standard parabolic subalgebras.  

We use the notation of the previous section. Recall that 
we have fixed $\Delta_{0}^{\fg}$. Let  
$\Delta_{0}^{1}\subset \Delta_{0}^{\fg}$ be a subset of 
$\Delta_{0}^{\fg}$.  

\begin{defin}
Let $\mathfrak l^{1}\subset \fg$ be the centraliser of 
$\ft_{1}^{\fg}$ in $\fg$. Let $K^{1}\subset K$ be the centraliser of 
$\ft_{1}^{\fg}$ in $K$. 
\end{defin} 

Then,  $\mathfrak l^{1}$ is a $\theta$-invariant real Lie subalgebra of $\fg$, which is reductive by \cite[Corollary 4.61 (b)]{KnappCohomology}.
We use the superscript $1$ to emphasis that the map $\Delta^{1}_{0}\to \mathfrak l^{1}$ is order preserving with respect 
to inclusion. 
	
By \cite[Corollary 4.51]{KnappLie}, $K^{1}$ is a connected compact Lie 
subgroup of $K$. Denote by $\fk^{1}$ the Lie algebra of $K^{1}$. 
Write 
	\begin{align}
		\mathfrak l^{1}=\fp^{1}\oplus \fk^{1}
	\end{align} 
the  Cartan decomposition of $\fl^{1}.$

Clearly, $\ft\subset \mathfrak k^{1}$ is the Cartan subalgebra. The 
root systems of $\mathfrak l^{1}$ and $\mathfrak k^{1}$ with respect 
to $\ft$ are given by 
	\begin{align}
&	R\(\mathfrak l^{1}\) =\left\{\alpha\in R(\fg): \alpha_{|\ft_{1}^{\fg}}= 
0\right\}, &	R\(\fk^{1}\)=\left\{\alpha\in R(\fk): 
\alpha_{|\ft_{1}^{\fg}}=
0\right\}. 
	\end{align} 
Let $W\(\mathfrak l^{1}\), W\(\fk^{1}\)$ be the corresponding Weyl 
groups.  Sometimes, we will also use the notations   
\begin{align}\label{eq:RW01}
&R_{0}^{1}=R\(\mathfrak l^{1}\),& W_{0}^{1}=W\(\mathfrak l^{1}\).  
\end{align} 

By Chevalley's Lemma\footnote{In \cite[Proposition 2.72]{KnappLie}, Chevalley's 
Lemma is stated for reduced root systems. Using   \cite[Lemma 
2.91]{KnappLie}, it is easy to extend the lemma to  nonreduced root systems.}, 
$W\(\mathfrak l^{1}\)\subset W(\fg)$ is the centraliser  of $\ft_{1}^{\fg}$ in 
$W(\fg)$, and $W\(\fk^{1}\)\subset W(\fk)$ is the centraliser of 
$\ft_{1}^{\fg}$ in $W(\fk)$, so that 
\begin{align}
W\(\fk^{1}\)=W\(\mathfrak l^{1}\)\cap W(\fk).
\end{align}

Put 
	\begin{align}
		&R_{+}\(\mathfrak l^{1}\)=R_{}\(\mathfrak l^{1}\)\cap 
		R_{+}(\fg), 	&	R_{+}\(\mathfrak k^{1}\)=R_{}\(\mathfrak 
		k^{1}\)\cap R_{+}(\fg). 
	\end{align} 
Then, 
\begin{align}
R_{+}\(\mathfrak k^{1}\)\subset R_{+}\(\mathfrak l^{1}\), 
\end{align} 
so that $R_{+}\(\mathfrak l^{1}\), R_{+}\(\mathfrak k^{1}\)$ are compatible
positive root systems.

Moreover, $\Delta^{1}_{0}$ is the  system of  simple roots of 
$R_{+}\(\mathfrak l^{1}\)$.   By \cite[Proposition 2.62]{KnappLie}, 
$W\(\mathfrak l^{1}\)$ is the subgroup of $W(\fg)$ generated by the 
reflections $s_{\alpha}$ with  $\alpha\in \Delta^{1}_{0}$. 

\begin{prop}\label{propw12l1}
If $w\in W\(\mathfrak l^{1}\)$ such that $w=w_{1}w_{2}$ with $w_{1}\in W(\fg, 
\fk)$ and $w_{2}\in W(\fk)$, then
\begin{align}
&w_{1}\in W\(\fl^{1},\fk^{1}\),&w_{2}\in  W\(\mathfrak k^{1}\).
\end{align} 
\end{prop} 
\begin{proof}
	Firstly, let us show $w_{2}\in W\(\fk^{1}\)$. Take $u\in {\rm 
	Int}\(C_{1}^{\fg}\)$. Since $w\in W\(\fl^{1}\)$, we have $w u=u$. Since 
	$w=w_{1}w_{2}$, we have
	\begin{align}\label{eqw2w1u}
w_{2}u=w_{1}^{-1}u. 
\end{align} 
Since $u\in  {\rm Int}\(C_{1}^{\fg}\)\subset C_{+}(\fg)$, by 
definition of $W(\fg, \fk)$, 	we have 
$w_{1}^{-1}u\in C_{+}(\fk)$. By \eqref{eqw2w1u}, we get 
\begin{align}\label{eqw2w1u1}
w_{2}u\in C_{+}(\fk). 
\end{align} 

By $u\in C_{+}(\fg) \subset C_{+}(\fk)$, and by \eqref{eqw2w1u1}, we get 
\begin{align}
w_{2}u=u.
\end{align} 
By Chevalley's Lemma \cite[Proposition 2.72]{KnappLie}, we get $w_{2}\in 
W\(\fk^{1}\)$.

Then, $w_{1}=ww_{2}^{-1}\in W\(\fl^{1}\)$. It remains to show 
\begin{align}
w_{1}^{-1}C_{+}\(\fl^{1}\)\subset C_{+}\(\fk^{1}\). 
\end{align} 
Or equivalently, for any $Y\in C_{+}\(\fl^{1}\)$ and $\beta\in 
R_{+}\(\fk^{1}\)$, we need to show 
\begin{align}\label{eq277}
\left\<w_{1}^{-1}Y, \beta\right\>\g 0. 
\end{align} 

By \eqref{eqCgot}, we have 
\begin{align}
C_{+}\(\fl^{1}\)=C_{0}^{1}\times \ft_{1}. 
\end{align} 
Since $w_{1}$ acts as identity on $\ft_{1}$, and 
$\beta$ vanishes on $\ft_{1}$, we see that \eqref{eq277} holds when $Y\in 
\ft_{1}$.  Note that $C_{0}^{1}$ is given by the nonnegative linear 
combinations of  $\check \Delta_{0}^{1}$. If 
$\alpha\in \Delta_{0}^{1}$, using again that $w_{1}$ acts as identity on $\ft_{1}$, and using $P_{0}^{1}\beta=\beta$, we have 
\begin{align}\label{eq279}
\left\<w_{1}^{-1} P_{0}^{1}\omega_\alpha, \beta\right\>= \left\< 
w_{1}^{-1}\omega_\alpha, \beta\right\>.
\end{align} 
The right hand side of \eqref{eq279} is nonnegative since $w_{1}\in W(\fg, \fk)$. 
This finishes the proof of \eqref{eq277} and completes the proof of our proposition. 
\end{proof} 

\begin{re}
By Proposition \ref{propw12l1}, if $w\in W\(\fl^{1}\)$, the decomposition  $w=w_{1}w_{2}$ in \eqref{eqwww12} is just the one associated to smaller Lie algebras $\fl^{1},\fk^{1}$ and to $R_{+}\(\fl^{1}\)$, $R_{+}\(\fk^{1}\)$.
By the uniqueness of the decomposition, we have 
\begin{align}
W(\fg,\fk)\cap W\(\fl^{1}\)=W\(\fl^{1},\fk^{1}\). 
\end{align} 
\end{re}

Set 
	\begin{align}\label{eq:R1fg}
		&R_{1}=R(\fg)\backslash R^{1}_{0},&R_{1,+}=R_{+}(\fg)\backslash 
		R^{1}_{0,+}. 
	\end{align} 
Since an element in $R(\fg)$ is a nonnegative or nonpositive linear combination of $\Delta_{0}^{\fg}$, for a given $ \alpha\in R_{1} $, we observe that $\alpha\in R_{1, +}$ if and only if $\alpha$ takes positive values on the interior of $C_{1}^{\fg}$.

\begin{defin}
Put 
	\begin{align}\label{equ12}
		\fu_{1}=\bigoplus_{\alpha\in R_{1,+}}\fg_{\alpha}.
	\end{align} 	
\end{defin} 	
Then, $\fu_{1}$ is a $\theta$-invariant complex Lie algebra, which is 
 nilpotent by \cite[Corollary 4.61]{KnappCohomology}. Here, we use the subscript $1$ to emphasize the map $\Delta^{1}_{0}\to \mathfrak \fu_{1}$ is order reversing with respect 
to inclusion.

\begin{defin}
Put 
	\begin{align}\label{eqq1}
		\fq^{1}=\mathfrak l_{\bC}^{1}\oplus \fu_{1}. 
	\end{align} 	
\end{defin} 
	
Then, $\fq^{1}$ is a $\theta$-invariant complex parabolic subalgebra 
of $\fg_{\bC}$. The map $\Delta^{1}_{0}\to \fq^{1}$ is order preserving as indicated by the superscript. 

The parabolic subalgebra $\fq^{1}$ constructed in this way will be 
called standard parabolic $\theta$-invariant subalgebra of 
$\fg_{\bC}$. It is equivalent to the one constructed in \cite[Proposition 4.76]{KnappCohomology}.

If  $\Delta^{1}_{0}=\varnothing$, the corresponding object will be  given 
the superscript or subscript $0$. Then, 
	\begin{align}\label{eqflq0}
&		\fl^{0}=\fa\oplus\ft, & \fu_{0}=\bigoplus _{\alpha\in R_{+}(\fg)} 
\fg_{\alpha}, && \fq^{0}= \fl^{0}_{\mathbf{C} } \oplus \fu_{0}. 
	\end{align} 

Let $\Delta_{0}^{2}$ be another subset of $\Delta_{0}^{\fg}$, so that 
	\begin{align}
		\fq^{2}=\mathfrak l ^{2}_{\bC}\oplus \fu_{2}. 
	\end{align} 
Assume $\Delta_{0}^{1}\subset \Delta_{0}^{2}$. Then, 
	\begin{align}
		\fq^{1}\subset \fq^{2}. 
	\end{align}

Consider  $\mathfrak l^{2}$ as the ambient reductive Lie algebra with 
system of simple root $\Delta_{0}^{2}$.  We  can define a standard $\theta$-invariant parabolic subalgebra of $\mathfrak l^{2}$ associated to $\Delta_{0}^{1}$, 
	\begin{align}\label{eqqst}
\mathfrak q_{*}=	\mathfrak l^{1}_{\bC}\oplus \fu_{1}^{2}. 
	\end{align}
	
Applying \eqref{equ12} with $\fg$ replaced by $\mathfrak l^{2}$, we obtain 
	\begin{align}\label{equu12}
		\fu_{1}^{2}= \bigoplus_{\alpha\in R^{2}_{0,+}\backslash 
		R^{1}_{0,+}} 
		\fg_{\alpha}. 
	\end{align}
By \eqref{equ12} and \eqref{equu12}, we get 
	\begin{align}\label{equ112}
\fu_{1}= \fu_{1}^{2}\oplus \fu_{2}. 		
	\end{align} 
	By \eqref{eqq1}, \eqref{eqqst}, and \eqref{equ112},  we have 
	\begin{align}
		\mathfrak q^{1}=\mathfrak q_{*}\oplus \fu_{2}. 
	\end{align}

From the above observations, it is clear that, given $\fq^{2}$, the map 
	\begin{align}\label{eqq1qs1}
		\fq^{1}\to \fq_{*}
	\end{align} 
establishes a bijection between the set of standard $\theta$-invariant parabolic subalgebras of $\fg_{\bC}$ contained in $\mathfrak q^{2}$ and the set of the standard $\theta$-invariant parabolic subalgebras of $\mathfrak l^{2}_{\mathbf{C} } $. 
The inverse map is given by 
	\begin{align}\label{eqq1qs2}
		\fq_{*}\to \fq_{*}\oplus \fu_{2}. 
	\end{align} 


%

\subsection{A subgroup of  $K^{1}$}\label{sSubK1}

Let $\Delta_{0}^{1}\subset \Delta_{0}^{\fg}$. Then, $\sqrt{-1}\ft_{1}\subset \fk^{1}$. Let  
$\fk_{s}^{1}$ be the orthogonal complement of $\sqrt{-1}\ft_{1}$ in $ \fk^{1}$ with respect to $B$, so that 
	\begin{align}
		\fk^{1}=\fk^{1}_{s}\oplus \sqrt{-1}\ft_{1}. 
	\end{align} 
Then,  $\fk^{1}_{s}$ is a Lie subalgebra of $\fk^{1}$, and $\sqrt{-1}\ft_{0}^{1}\subset \fk^{1}_{s}$ is a Cartan subalgebra of $\fk^{1}_{s}$, so 
that we have the root decomposition 
	\begin{align}
		\fk^{1}_{s\bC}=\ft_{0\bC}^{1}\oplus \bigoplus_{\alpha\in 
		R\(\fk^{1}\)}\fk_{\alpha}. 
	\end{align} 
Since $\ft_{0}^{1}$ is generated by  $R\(\fk^{1}\)$, we know that 
$\fk^{1}_{s}$ is a semisimple Lie algebra and $\sqrt{-1}\ft_{1}$ is the 
center  of $\fk^{1}$. 

Let $K^{1}_{s}, T^{1}_{0}, T_{1}\subset K^{1}$ be the Lie subgroups of $K^{1}$ 
associated to the Lie algebras $\fk^{1}_{s}, \sqrt{-1}\ft_{0}^{1}, 
\sqrt{-1}\ft_{1}$. By \cite[Theorem 4.29]{KnappLie}, $K^{1}_{s}$ is 
compact. Moreover, $T^{1}_{0}$ is a maximal torus of $K^{1}_{s}$, and 
$T_{1}$ is  the connected component of the center of $K^{1}$ that contains the identity. 

Let $\fm^{1}$  be the orthogonal subspace to 
$\sqrt{-1}\ft_{1}$ in $\fl^{1}$ with respect to $B$, so that 
	\begin{align}\label{eq:l1m1t1}
		\fl^{1}=\fm^{1}\oplus \sqrt{-1}\ft_{1}. 
	\end{align} 
As before, $\fm^{1}$ is a Lie subalgebra of $\fl^{1}$. 
It is easy to see that the compact component of the centre of $\fl^{1}$ is $\sqrt{-1}\ft_{1}$, and $\fm^{1}$ is the direct sum of the semisimple part of $\fl^{1}$ with the noncompact component of the centre of $\fl^{1}$. 
In particular, $\fm^{1}$ is reductive with the Cartan decomposition 
	\begin{align}\label{eqmfks}
		\fm^{1}=\fp^{1}\oplus \fk^{1}_{s}. 
	\end{align} 

Clearly, $R\(\fk^{1}_{s}\)=R\(\fk^{1}\),$ $ R\(\fm^{1}\)=R\(\fl^{1}\)$. 
Take 
	\begin{align}
		&R_{+}\(\fk^{1}_{s}\)=R_{+}(\fk^{1}),& R_{+}\(\fm^{1}\)=R_{+}(\fl^{1}).  
	\end{align} 
	
\begin{prop}
We have 
	\begin{align}\label{eqk1rho}
		&\varrho^{\fk^{1}_{s}}=\varrho^{\fk}{}_{|\ft_{0}^{1}},&  
		\varrho^{\fm^{1}}=\varrho^{\fg}{}_{|\ft_{0}^{1}}. 
	\end{align} 
\end{prop} 	
\begin{proof}
The first identity of \eqref{eqk1rho} is 
		equivalent to 
	\begin{align}\label{eqk1rho1}
		\sum_{\alpha\in R_{+}(\fk)\backslash 
		R_{+}\(\fk^{1}\)}\alpha_{|\ft_{0}^{1}}=0. 
	\end{align} 
By \eqref{eq:R1fg} and \eqref{equ12}, if $Y_{0}^{1}\in \ft_{0}^{1}$, we have  
	\begin{align}\label{eqk1rho2}
		\sum_{\alpha\in R_{+}(\fk)\backslash 
		R_{+}\(\fk^{1}\)}\left\langle \alpha, 
		Y_{0}^{1}\right\rangle =\Tr^{\fu_{1}\cap \fk_{\bC}}\[\ad\(Y_{0}^{1}\)\]. 
	\end{align} 

Since $\fk_{s}^{1}$ is semisimple, then 
$[\fk_{s}^{1},\fk_{s}^{1}]=\fk_{s}^{1}$. In particular, $\sqrt{-1} Y_{0}^{1}\in 
\sqrt{-1} \ft_{0}^{1}$ is a commutator  of elements in $\fk_{s}^{1}$. 
Since $\fk_{s}^{1}$ acts on $\fu_{1}\cap \fk_{\bC}$ and since the trace vanishes on commutators, we see that 
	\begin{align}\label{eqk1rho3}
		\Tr^{\fu_{1}\cap \fk_{\bC}}\[\ad\(Y_{0}^{1}\)\]=0. 
	\end{align} 
By \eqref{eqk1rho2} and \eqref{eqk1rho3}, we get \eqref{eqk1rho1}. 
	
By a similar method, we get  the second identity of \eqref{eqk1rho}. 
\end{proof} 

If $E$ is an irreducible representation of $K$ with highest weight 
$\lambda^{E}$. Set 
	\begin{align}\label{eqK1zL}
		\lambda^{E,1}=\lambda^{E}_{|\ft_{0}^{1}}. 
	\end{align} 
Since  $\lambda^{E}$ is a $T$-weight, we know that $\lambda^{E,1}$ is 
a $T_{0}^{1}$-weight. Moreover,  
	$\lambda^{E,1}\in C_{0}^{1}$. Therefore, $\lambda^{E,1}$ is a 
	$R_{+}\(\fk^{1}_{s}\)$-dominant 
	$T_{0}^{1}$-weight.

\begin{defin}\label{deftauE1}
Let $\tau^{E,1}: K^{1}_{s}\to {\rm U}\(E^{1}\)$ be the irreducible 
unitary representation of $K_{1}^{s}$ of highest weight $\lambda^{E, 1}$. 
\end{defin}

\subsection{Langlands' combinatorial Lemma}\label{sLMAP}
We have fixed $R_{+}(\fg)\subset R(\fg)$. 
We use the notation in Section \ref{sPSA}. 
Recall that 
	\begin{align}
		&C_{+}(\fg)=C^{\fg}_{0}\times \ft_{\fg},&\check 
		C_{+}(\fg)=\check  C^{\fg}_{0}. 
	\end{align} 


Take $v\in \ft_{0}$.  Since $C_{+}(\fg)$ is closed and convex, the function 
	\begin{align}\label{eqv00}
		Y_{0}\in C_+(\fg)\to 
\left|Y_{0}-v\right|\in \bR_{+}
	\end{align}  
	has  one and only one minimal point, which is  called the 
	projection of $v$ onto $C_{+}(\fg)$. 
	By the classical  property of the projection onto a closed  cone,  the projection of $v$ is the unique element 
	$v_{*}\in C_{+}(\fg)$ such that 
\begin{align}\label{eqvv0}
&\<v-v_{*},v_{*}\>=0,	& v_{*}-v\in \check C_{0}^{\fg}. 
\end{align} 
	
%

The following proposition is due to Langlands \cite[Lemma 
4.4]{Langlands89}. 
The details can be found in 
\cite[Corollaire 1.4]{Carmona83}. 

\begin{prop}\label{propLC1}
	For $v\in \ft_{0}$, there exist uniquely  two 
	subsets $\Delta_{0}^{1},\Delta_{0}^{2}$  of $\Delta_{0}^{\fg}$ with 
	$\Delta_{0}^{1}\subset \Delta_{0}^{2}$, such that 
		\begin{align}\label{eqvFG}
&		v=v_{0}^{1}+v_{2}, &v_{0}^{1}\in -{\rm Int}\(\check 
C_{0}^{1}\),&& 
v_{2}\in {\rm Int}\(C_{2}^{\fg}\) \times \ft_{\fg}. 
	\end{align} 
 Moreover, 
\begin{align}\label{eqv0}
	v_{*}=v_{2}. 
\end{align} 
\end{prop}

Let $\ul v\in \ft_{0}$ be another element. Denote by $\ul 
\Delta_{0}^{1}$, $\ul \Delta_{0}^{2}$, $\ul v_{0}^{1}$, and $\ul 
v_{2}$ the associated objects.  The proof of the following proposition is essentially due to Langlands \cite[Corollary 4.6]{Langlands89}. 
\begin{prop}\label{propv22delta}
Let $v,\ul v\in \ft_{0}$ be such that $v\in \ul v + \check C_{0}^{\fg}$. 
Then,  
	\begin{align}\label{eqLanC}
&		v_{2}-\ul v_{2}\in \check C^{\fg}_{0},&\left|v_{2}\right|\g \left|\ul v_{2}\right|.
	\end{align} 
Moreover, 
	\begin{align}\label{eqLanC0}
		|v_{2}|= \left|\ul 
v_{2}\right|\iff v_{2}=\ul v_{2} \iff \left\<	v-\ul v,v_{2}\right\>=0. 
	\end{align} 
Also, if one of the conditions of \eqref{eqLanC0} holds, we have 	
	\begin{align}\label{eqDinD}
		&\Delta_{0}^{1}\subset \ul \Delta_{0}^{1},&\Delta_{0}^{2} =  \ul \Delta_{0}^{2}.
	\end{align} 
\end{prop}

\subsection{Vogan's Lambda map}\label{sVoganLambda}
%
Let $\mu\in \ft_{0}$. Choose a positive root system 
$R_{+}(\fg)\subset R(\fg)$ of $R(\fg)$ such that  
\begin{align}\label{eqmuCpp}
		\mu\in C_{+}(\fg). 
	\end{align} 
Clearly, the choice of $R_{+}(\fg)$ is not unique. 
	
Let $\varrho^{\fg}$ be the vector  defined in \eqref{eqrhog}. By Proposition 
\ref{propLC1}, there exist $\Delta_{0}^{1},\Delta_{0}^{2}\subset 
\Delta_{0}^{\fg}$ and a  decomposition 
	\begin{align}\label{eqmuvarholl}
		\mu-\varrho^{\fg}=\(\mu-\varrho^{\fg}\)_{0}^{1}+\(\mu-\varrho^{\fg}\)_{2}. 
	\end{align} 
Note that $\varrho^{\fg}$ and the decomposition \eqref{eqmuvarholl}  depend on the choice of $R_{+}(\fg)$. 

\begin{re}\label{re:n12h}
 If $ w\in W\left(\mathfrak{g} \right)$, applying $ w^{-1} $ to \eqref{eqmuvarholl}, we get 
 \begin{align}\label{eqmuvarholl1}
	w^{-1} \mu-w^{-1} \varrho^{\fg}=w^{-1}\(\mu-\varrho^{\fg}\)_{0}^{1}+w^{-1}\(\mu-\varrho^{\fg}\)_{2}. 
\end{align} 
Since $ w^{-1} \mu\in w^{-1} C\left(\mathfrak{g} \right)$, by \eqref{eqvv0} and \eqref{eqvFG}, we see that \eqref{eqmuvarholl1} is just the decomposition of \eqref{eqmuvarholl} associated to $ w^{-1} \mu $ and $ w^{-1} C_{+} \left(\mathfrak{g} \right)$. 
\end{re}

Let $W_{0}^{1}, W_{0}^{2}$ be the Weyl groups of the root systems $R_{0}^{1}$ and $R_{0}^{2}$. Clearly, 
\begin{align}
W_{0}^{1}\subset W_{0}^{2}. 
\end{align} 

By Proposition \ref{propuvC}, $ w^{-1} C_{+} \left(\mathfrak{g} \right)$ is another Weyl chambre satisfying \eqref{eqmuCpp} if and only if 
\begin{align}\label{eq:pam1}
 w \mu = \mu. 
\end{align}
The following proposition is due to Carmona \cite[Proposition 2.2]{Carmona83}.

\begin{prop}\label{propCVogan}
 If $ w\in W\left(\mathfrak{g} \right)$ such that  $ w\mu= \mu$, then 
 \begin{align}\label{eq:lsos}
&w\in W_{0 }^{1}, &w^{-1} \left(\varrho^{\fg}+\(\mu-\varrho^{\fg}\)_{0}^{1}\right)= \varrho^{\fg}+\(\mu-\varrho^{\fg}\)_{0}^{1},&&w^{-1} \(\mu-\varrho^{\fg}\)_{2}=	\(\mu-\varrho^{\fg}\)_{2}. 
 \end{align}
In particular, the vectors  $\varrho^{\fg}+\(\mu-\varrho^{\fg}\)_{0}^{1}$ and  $\(\mu-\varrho^{\fg}\)_{2}\in \ft_{0}$ are independent of the choice of the positive root system $R_{+}(\fg)$ with the property \eqref{eqmuCpp}. 
\end{prop}

\begin{defin}\label{defLambda}
Let $\Lambda: \ft_{0}\to \ft_{0}$ be the map  defined by  $\mu\in \ft_{0}\to \(\mu-\varrho^{\fg}\)_{2}\in \ft_0$.
\end{defin} 


In \cite[Proposition 2.2]{Carmona83}, Carmona showed that $\Lambda$ coincides with Vogan's Lambda map  \cite[Proposition 
4.1]{Vogan-book}. Vogan's constructions are obtained  by a recurrence 
procedure.   The results of our paper do not rely  on Vogan's constructions.

\section{Statement of our main results}\label{Smainresult}
Assume $E$ is an irreducible representation of $K$. We fix a positive 
system $R_{+}(\fk)$ of $R(\fk)$. Let $\lambda^{E}\in \ft_{0}$ be the 
highest weight of $E$. Then, 
	\begin{align}
		\lambda^{E}+\varrho^{\fk}\in {\rm Int}\  C_{+}(\fk) 
	\end{align} 
is in the open Weyl chamber.

We fix $R_{+}(\fg)\subset R(\fg)$ a positive root system such that 
\begin{align}\label{eqlambdaEinC}
\lambda^{E}+2\varrho^{\fk}\in C_{+}(\fg). 
	\end{align} 
Since $\lambda^{E}+2\varrho^{\fk}\in C_{+}(\fk)$, we see that 
\begin{align}
C_{+}(\fg)\subset C_{+}(\fk), 
\end{align} 
so that $R_{+}(\fg)$ and $R_{+}(\fk)$ are compatible.  
As we have already observed in Section \ref{sVoganLambda}, $R_{+}(\fg)$ is generally not unique. 	
  
By Proposition \ref{propLC1}, associated with the vector $\lambda^{E}+2\varrho^{\fk}-\varrho^{\fg}$, there exist subsets $\Delta_{0}^{1}$, $\Delta_{0}^{2}$ of $\Delta_{0}^{\fg}$ such that $\Delta_{0}^{1}\subset \Delta_{0}^{2}$. 
This leads to the following orthogonal decomposition, 
	\begin{align}\label{eq:t0d1}
		\ft_{0}= \ft_{0}^{1}\oplus \ft_{1}^{2}\oplus \ft_{2}. 
	\end{align} 

\begin{exa} Suppose that $ G = {\rm SL}(2, \bR)$ and $K = {\rm SO}(2)$.
We use the notations in Section \ref{subsec:sl2}.
The decomposition \eqref{eq:t0d1} depends on the highest weight $\lambda^E\in \mathbf{Z} $. 
If $ \lambda^{E} \ge 0 $, we can choose $ C_{+} \left(\mathfrak{g} \right)$ such that $ \rho^{\fg} = 1$.
The three cases in \eqref{case1}-\eqref{case3} give the decompositions of $\ft_0$ as follows:
\begin{align}
    \begin{cases}
   \ft_{0}^{1} = 0, \quad \ft_{1}^{2} = 0,\quad \ft_{2} = \ft_0,  & \text{if } \lambda^E \geq 2 ;\\
   \ft_{0}^{1} =0, \quad \ft_{1}^2 = \ft_0, \quad\ft_{2} = 0,  & \text{if } \lambda^E = 1; \\
   \ft_{0}^{1} = \ft_0, \quad \ft_{1}^2  = 0,  \quad\ft_{2} = 0,   & \text{if } \lambda^E = 0. \\
\end{cases}
\end{align}
\end{exa}
Set 
	\begin{align}\label{eq:rdecp}
&		r_{0}=\dim \ft_{0}, &r_{0}^{1}=\dim \ft_{0}^{1}, && 
r_{1}^{2}=\dim \ft_{1}^{2}, &&& r_{2}=\dim \ft_{2}. 
	\end{align} 
By \eqref{eq:t0d1} and \eqref{eq:rdecp}, we have  
	\begin{align}
		r_{0}=r_{0}^{1}+r_{1}^{2}+r_{2}. 
	\end{align} 

As in \eqref{eqmuvarholl}, we write 
	\begin{align}
		\lambda^{E}+2\varrho^{\fk}-\varrho^{\fg}=\(\lambda^{E}+2\varrho^{\fk}-\varrho^{\fg}\)_{0}^{1}+\(\lambda^{E}+2\varrho^{\fk}-\varrho^{\fg}\)_{2},
	\end{align} 
such that 
	\begin{align}
&		\(\lambda^{E}+2\varrho^{\fk}-\varrho^{\fg}\)_{0}^{1}\in -\check 
		C_{0}^{1},& \(\lambda^{E}+2\varrho^{\fk}-\varrho^{\fg}\)_{2}\in 
		C_{2}^{\fg}\times \ft_{\fg}. 
	\end{align} 

Recall that $ R_{0 }^{1}, R_{0 }^{2}, R_{1}, R_{2} $ are defined in \eqref{eq:RW01} and \eqref{eq:R1fg}.
If $Y_{1}^{2}\in \ft_{1}^{2}, Y_{2}\in \ft_{2}$, set  
	\begin{align}
&		\pi_{1}^{2}\(Y_{1}^{2}\)=\prod_{\alpha\in R_{0, +}^{2} \backslash 
R_{0 ,+}^{1}}\left\langle \alpha, 
Y_{1}^{2}\right\rangle,&\pi_{2}\(Y_{2}\)=\prod_{\alpha\in R_{2, 
+}}\left\langle \alpha, Y_{2}\right\rangle. 
	\end{align} 
	
Let $\fq^{1}=\fl^{1}\oplus \fu_{1}, \fq^{2}=\fl^{2}\oplus \fu_{2}$ be the standard $\theta$-invariant  parabolic 
subalgebras of $\fg_{\bC}$ associated to $\Delta_{0}^{1}$ and 
$\Delta_{0}^{2}$. 
Let $K^{1}_{s}$ be the  compact  semisimple Lie group defined in Section \ref{sSubK1}. 
Write 
	\begin{align}
		n_{s}=\dim K_{s}^{1}. 
	\end{align} 
Recall that $\fm^{1}$ is the reductive Lie 
algebra with Cartan decomposition 
	\begin{align}
		\fm^{1}=\fp^{1}\oplus \fk^{1}_{s}. 
	\end{align} 
And $\tau^{E,{1}}$ is the induced representation of $K^{1}_{s}$ defined in Definition \ref{deftauE1}.

Recall also that $c_{\mathfrak{g}}$ is defined in (\ref{eqcg}). 

\begin{defin}\label{defin31}
	Set 
		\begin{align}\label{eq313}
			\notag
		\alpha_{0}^{1}&=\pi^{\fk^{1}_{s}}\(\varrho^{\fk^{1}_{s}}\)\int_{ \sqrt{-1}{\fk}^{1}_{s}}\frac{\widehat 
		A\(\ad\(Y^{\fk^1_{s}}_{0}\)_{|{\fp}^{1}}\)}{\widehat 
A\(\ad\(Y^{\fk^{1}_{s}}_{0}\)_{\fk^{1}_{s}}\)}\Tr\[\tau^{E,1}\(e^{-Y^{\fk^1_{s}}_{0}}\)\]\frac{dY_{0}^{\fk^{1}_{s}}}{(2\pi)^{n_{s}^{1}/2}}.\\
\alpha_{1}^{2}&=\int_{ C^{2}_{1}}\pi_{1}^{2}\(Y_{1}^{2}\) 
		\exp\(-\frac{1}{2}\left|Y_{1}^{2}\right|^{2}\)\frac{dY_{1}^{2}}{(2\pi)^{r_{1}^{2}/2}},\\
		\alpha_{2}&=\pi_{2}\(\(\lambda^{E}+2\varrho^{\fk}-\varrho^{\fg}\)_{2}\).\notag
	\end{align} 
Define 
	\begin{align}\label{eqa0}
	&\alpha_{0} = \alpha_{0}^{1}\alpha_{1}^{2}\alpha_{2},& 
	\ul\alpha_{0}=\frac{\alpha_{0}}{(2\pi)^{m/2}}\[\pi^{\fk}\(\rho^{\fk}\)\]^{-1},\notag\\
		&	\beta_{1}=	-\frac{1}{2}r_{0}^{1}+\frac{1}{2}\dim_{\bC} 
\fu_{1}^{2}+\dim_{\bC} 
\fu_{2},&  \ul 
\beta_{1}=\beta_{1}-\frac{m+n-r_{0}}{2},\\
	&\gamma_{2}=\frac{1}{2}\left\|\(\lambda^{E}+2\varrho^{\fk}-\varrho^{\fg}\)_{2}\right\|^{2}, &\ul\gamma_{2}=\gamma_{2}-\frac{c_{\fg}}{2}. \notag
	\end{align} 
\end{defin} 
\begin{re}
	If $\Delta_{0}^{1}=\varnothing$, then $C_{1}^{2}=C_{0}^{2}$ is a Weyl 
chamber, so that 
	\begin{align}\label{eqa12}
		\alpha_{1}^{2}=\frac{1}{\left|W_{0}^{2}\right|}\int_{\ft_{0}^{2}} \left|\pi_{0}^{2}\(Y_{0}^{2}\)\right|\exp\(-\frac{1}{2}\left|Y_{0}^{2}\right|^{2}\)\frac{dY_{0}^{2}}{(2\pi)^{r_{0}^{2}/2}}.
	\end{align} 
This is  the Mehta-Macdonald integral \cite{Macdonald82}.  An  explicit evaluation  is given by Opdam \cite[Theorem 6.4]{Opdam89}. 
\end{re}

\begin{re}\label{re:b1}
	Using $ m+n = \dim \mathfrak{l}^{1} +2 \dim_{\mathbf{C} } \mathfrak{u}_{1} $, by \eqref{eq:l1m1t1} and \eqref{eqa0}, we have  
	\begin{align}
		\ul\beta_{1}=-\frac{1}{2}\(\dim \fm^{1}+\dim_{\bC} \fu_{1}^{2}  \)\in - \frac{1}{2}\mathbf N. 
	\end{align} 
\end{re}

\begin{re}
	By Definition \ref{defLambda} and by \eqref{eqa0}, we can write   
	\begin{align}
		\gamma_{2}=\pi_{2}\(\Lambda\(\lambda^{E}+2\varrho^{\fk}\)\).  
	\end{align} 
\end{re}

\begin{re}
	All the constants defined in Definition \ref{defin31} are independent of the choice of $R_{+}(\fg)$. 
	Indeed, if $R_{+}(\fg)$ is replaced by another positive root system $ w^{-1} R_{+} \left(\mathfrak{g} \right)$ satisfying \eqref{eqlambdaEinC}, by Proposition \ref{propCVogan}, we know that
	$\(\lambda^{E}+2\varrho^{\fk}-\varrho^{\fg}\)_{2}$ remains unchanged, so that $ \gamma_{2} $ is also unchanged.
	Moreover, $\ft_{0}^{1},\ft_{1}^{2}, \ft_{2}$ remain unchanged, so that $ \mathfrak{l}^{1}, \mathfrak{l}^{2} $ are unchanged, and therefore
	$ \beta_{1} $ is unchanged.
	Additionally, since $ w $ preserves $ R_{0 }^{1}, R_{0 }^{2}  $ and acts as the identity on  $ \ft_{1}^{2},\mathfrak{t}_{2} $, we know that $ \alpha_{1}^{2},\alpha_{2}  $ are unchanged.	
	Furthermore, $ K^{1}_{s} $ and $ R_{+} \left(\fk^{1}_{s} \right)$ are unchanged, so that 
	$\tau^{E,1}$ is also unchanged.

	Let us note, however, that $\Delta_{0}^{1}$ changes to $w^{-1}\Delta_{0}^{1}$.  
\end{re}


Now we can state the main results of our paper. 

 \begin{thm}\label{thm1} 
	If $ E$ is irreducible, with the notations in this section, as $t\to \infty$, 
	\begin{align}
		\Tr_{G}\[\exp\(-\frac{t}{2}C^{\fg, 
		X}\)\]\sim 
		\ul \alpha_{0} t^{\ul \beta_{1}}e^{t \ul \gamma_{2}}. 
	\end{align} 
 \end{thm} 
\begin{proof}
		The proof of our theorem will be given in Section 	\ref{sPfthm1}. 
\end{proof}   

\begin{thm}\label{thm2}
 The condition $ \underline{\beta }_{1}  = 0 $ is realised if and only if $ E$ is irreducible such that $ \Delta_{0}^{1} = \Delta_{0}^{2} = \varnothing $ and $ G$ is equal rank, i.e., $ {\rm rk}_{\mathbf{C} } G= {\rm rk}_{\mathbf{C} } K$.
 In this case,  there is $ \epsilon_{0} > 0 $ such that as $ t \to \infty$, 
 \begin{align}\label{eq:tawi}
	\Tr_{G}\[\exp\(-\frac{t}{2}C^{\fg, 
		X}\)\]= \ul \alpha_{0} e^{t \ul \gamma_{2}} \left(1+\mathcal{O} \left(e^{-\epsilon_{0} t} \right)\right).
 \end{align}
\end{thm}
\begin{proof}
	By Remark \ref{re:b1}, we see that 
	\begin{align}\label{eq:hiko}
	 \underline{\beta }_{1}=0 \iff \dim \mathfrak{m}^{1} = 0,\ \dim_{\mathbf{C} }  \mathfrak{u}_{1}^{2} = 0. 
	\end{align}
	By the first equation of \eqref{eqflq0} and by \eqref{eq:l1m1t1}, we know that 
	\begin{align}\label{eq:pwhn}
	 \dim \mathfrak{m}^{1} = 0 \iff \ft_{0 }^{1}  = \varnothing, \ \mathfrak{l}^{1} = \ft \iff \Delta _{0 }^{1} = \varnothing,\  \fa= 0 .
	\end{align}
	By \eqref{equu12}, we have 
	\begin{align}\label{eq:c3ar}
		\mathfrak{u}_{1}^{2} = 0 \iff R_{0 }^{1} = R_{0 }^{2} \iff \Delta_{0}^{1} = \Delta_{0}^{2}.  
	\end{align}
	By \eqref{eq:pwhn} and by \eqref{eq:c3ar}, we get our first statement.
	
	The proof of the second statement will be given in Section \ref{sPfthm1}. 
 \end{proof}

\section{Applications of Theorems \ref{thm1} and \ref{thm2}}\label{sec:application}
The purpose of this section is to discuss some immediate applications of our main result Theorems \ref{thm1} and \ref{thm2}.

This section is organised as follows. 
In Section \ref{App ds repn}, we provide a 
representation-theoretic interpretation of 
Theorem \ref{thm2}. 

In Section \ref{SApplication}, we discuss how the asymptotics in Theorem \ref{thm1} is related to the corresponding asymptotics associated to some quasi-split subreductive group and small representations.

In Section \ref{sSpecMeas}, we study the spectral measure of the operator $ C^{\mathfrak{g}, X} $.

In Section \ref{NS invariant}, we discuss how to use Theorem \ref{thm1} to study the Novikov-Shubin-type invariant on locally symmetric spaces. 

In Section \ref{sNSinvtwist}, we study the classical Novikov-Shubin invariant for a class of Hermitian flat vector bundles on locally symmetric spaces.

\subsection{Application to discrete series representations} \label{App ds repn}
We assume that $G$ has a compact Cartan subgroup, or equivalently, $\rk_{\mathbf{C} }  G =  \rk_{\mathbf{C} }  K $. 
In this case,  $G$ has discrete series representations. 

Let $E$ be an irreducible representation of $K$ with a highest weight $\lambda^E \in \ft_0^{*} $. 
We assume that $\lambda^E$ is regular in the following sense:
\begin{align}\label{eq:regHC}
\left\langle \lambda^{E}+2\varrho^{\fk}-\varrho^{\fg}, \alpha \right\rangle > 0, \quad \forall \alpha \in R_{+}(\fg). 
\end{align}
In particular, we have  
\begin{align}
\lambda^{E}+2\varrho^{\fk}-\varrho^{\fg} =  \(\lambda^{E}+2\varrho^{\fk}-\varrho^{\fg}\)_{2},   
\end{align}
and $\Delta_{0}^{1}= \Delta_{0}^{2} =\emptyset$.

By Theorem \ref{thm2}, we have 
\begin{align}
		\Tr_{G}\[\exp\(-\frac{t}{2}C^{\fg, 
		X}\)\]= 	\ul \alpha_{0} e^{t \ul \gamma_{2}}\left(1+\mathcal{O} \left(e^{-\epsilon_{0} t} \right)\right),
\end{align}
where
\begin{align}\label{eq:disccase}
&\ul\alpha_{0}= \frac{\pi^{\mathfrak{g} } \left(\frac{\lambda^{E}+2\varrho^{\fk}-\varrho^{\fg}}{2\pi } \right)}{\pi^{\fk}\(\frac{\rho^{\fk}}{2\pi } \)},
&\ul\gamma_{2}= \frac{1}{2}\left\|\lambda^{E}+2\varrho^{\fk}-\varrho^{\fg}\right\|^{2} -\frac{c_{\fg}}{2}.
\end{align}
Up to a constant depending on the choice of the Haar measure on $G$,  $\ul\alpha_{0}$ equals to the formal degree of the discrete series representation with Harish-Chandra parameter $\lambda^{E}+2\varrho^{\fk}-\varrho^{\fg}$.


\subsection{Reduction to small representations}\label{SApplication}
 

	

Let $\Delta_{0}^{2}\subset \Delta_{0}^{\fg}$ be the subset of simple roots associated to the vector $\lambda^{E}+2\varrho^{\fk}-\varrho^{\fg}$ as in Section \ref{Smainresult}.   
We have the corresponding orthogonal decomposition following \eqref{eqft012} and (\ref{eqft012g}), 
	\begin{align}
		\ft_{0}=\ft_{0}^{2}\oplus \ft_{2}.
	\end{align} 
Then, 
$\left(\lambda^{E}+2\varrho^{\fk}-\varrho^{\fg}\right)_2$ is in $ {\rm Int}\left(C_{2} \right)$. Let $\mathfrak l^{2}\subset \fg$ be the centraliser of $\ft_{2}^{\fg}$ in $\fg$. 
Let $\fq^{2}=\fl^{2}\oplus \fu_{2}$ be the standard $\theta$-invariant parabolic subalgebras of $\fg_{\bC}$ associated to $\Delta_{0}^{2}$. 
Let $ L^{2} $ be the connected Lie subgroup associated to $ \fl^{2} $.
Let $K^{2}\subset K$ be the centraliser of $\ft_{2}^{\fg}$ in $K$.
By \cite[Corollary 4.51]{KnappLie}, $ K^{2} $ is connected, so that $ K^{2} $ is a maximal compact subgroup of $ L^{2} $.
Then $ \left(L^{2}, K^{2} \right)$ is a subreductive pair of $ \left(G,K\right)$.
We have the corresponding Cartan decomposition
\begin{align}
		\mathfrak l^{2}=\fp^{2}\oplus \fk^{2}.
\end{align} 
Put 
\begin{align}\label{eq:byso}
 X^{2} = L^{2} /K^{2}. 
\end{align}

Set 
	\begin{align}
		\lambda^{E,2}=\lambda^{E}_{|\ft_{0}^{2}}. 
	\end{align} 
As in \eqref{eqK1zL}, $\lambda^{E,2}$ is a $R_{+}\(\fk^{2}\)$-dominant $T_{0}^{2}$-weight.
Let 
    \begin{align*}
       \tau^{E,2}: K^{2}\to {\rm U}\(E^{2}\) 
    \end{align*} be the irreducible 
unitary representation of $K^{2}$ of highest weight $\lambda^{E, 2}$. 

Let $ C^{\mathfrak{l}^{2}, X^{2} } $ be the associated Casimir operator acting on the sections of $ L^{2} \times_{K^{2} } E^{2} $. 
Let $ c_{\mathfrak{l}^{2} } $ be the constant  defined in \eqref{eqcg} while replacing $ \fg$ by $ \mathfrak{l}^{2} $.

\begin{prop}\label{prop:cckw}
 As $ t \to \infty$, we have 
 \begin{align}\label{eq:ngmm}
	\Tr_{L^{2} } \[\exp\(-\frac{t}{2}C^{\fl^{2} , 
	 X_{2} }\)\] \sim \frac{\alpha_{0}^{1} \alpha_{1}^{2}  }{\left(2\pi \right)^{\dim \mathfrak{p}^{2} /2} } \left[\pi^{\mathfrak{k}^{2} } \left(\varrho^{\fk^{2} } \right)\right]^{-1} t^{\underline{\beta }_{1} } e^{- c_{\mathfrak{l}^{2}} t/2}. 
 \end{align}
\end{prop}
\begin{proof}
	It is enough to apply Theorem \ref{thm1} to $ \left(L^{2}, K^{2},\tau^{E,2}  \right)$ together with the observations that the  constant $ \left(\alpha_{0}^{1},\alpha_{1}^{2},\underline{\beta }_{1}\right) $ associated to  $ \left(L^{2}, K^{2},\tau^{E,2}  \right)$ are unchanged\footnote{Here, our new $ \alpha_{1}^{2}$ is an integral over larger domain $ C_{1}^{2} \times \mathfrak{t}_{2}^{\mathfrak{g} } $, while its value is unchanged.
	}, while the corresponding $\left( \alpha_{2},\gamma_{2} \right) $ now becomes $\left( 1,0 \right)$.
\end{proof}

\begin{cor}[Reduction to $ \left(L^{2}, K^{2}, \tau^{E,2} \right)$]
	\label{cor:uemi}
 As $ t \to \infty$, we have 
 \begin{align}\label{eq:kjbh}
	\frac{\Tr_{G}\[\exp\(-\frac{t}{2}C^{\fg, 
	X}\)\]}{\Tr_{L^{2} } \[\exp\(-\frac{t}{2}C^{\fl^{2} , 
	X_{2} }\)\] } \sim \alpha_{2} \frac{\pi^{\mathfrak{k}^{2} }\left({2\pi }{\varrho^{\mathfrak{k}^{2} } } \right) }{\pi^{\fk} \left({2\pi }{\varrho^{\mathfrak{k} } } \right)}  e^{\left(\underline{\gamma} _{2} +c_{\mathfrak{l}^{2} }/2 \right)t}. 
 \end{align}
\end{cor}
\begin{proof}
 By Theorem \ref{thm2} and Proposition \ref{prop:cckw}, it is enough to show 
 \begin{align}\label{eq:wlcn}
	\dim \mathfrak{p} -\dim \fp^{2} = \dim \mathfrak{k} -\dim \fk^{2}. 
 \end{align}
	This follows from the fact that a generic element of $ \mathfrak{t}_{2}^{\mathfrak{g}} $ induces an isomorphism between the orthogonal complements of $ \fp^{2} $ in $ \mathfrak{p} $ and of $ \fk^{2} $ in $ \mathfrak{k} $.	
\end{proof}

\begin{re}\label{re:vko2}
	Following Vogan \cite[Definition 6.1, Theorem 6.4]{Vogan79}, $ \mathfrak{l}^{2} $ is quasi-split and  $\tau^{E,{2}}$ is small.
\end{re}

\subsection{The $ G$-spectral measure of $ C^{\mathfrak{g},X} $ }\label{sSpecMeas}
By the abstract spectral theory, there exists a Radon measure $ \mu $ on $ \bR $ such that
\begin{align}\label{eq:x2ki}
 \Tr_{G}  \left[ \exp\left(-\frac{t}{2} C^{\mathfrak{g}, X} \right) \right] = \int_{\bR} e^{-\frac{t}{2} \lambda} d\mu(\lambda).
\end{align}

\begin{prop}\label{prop:quth}
	The measure $ \mu $ is supported on $ \left[-2\underline{\gamma} _{2},+\infty \right)$. Moreover, $ -2\underline{\gamma} _{2}$ is contained in the support of $ \mu $.
	Also, if $ \underline{\beta} _{1}= 0  $, then $ \mu $ is supported on $ \left\{-2 \underline{\gamma }_{2} \right\}\cup \left[-2\underline{\gamma} _{2}+2\epsilon_{0},+\infty \right)$, and 
	\begin{align}\label{eq:un1m}
	 \mu \left(\left\{ -2\underline{\gamma }_{2} \right\}\right)= \underline{\alpha} _{0}. 
	\end{align}
	Furthermore, if $ \underline{\beta }_{1} < 0 $, then $ \mu $ has non spectral gap at $ -2 \underline{\gamma }_{2} $, and 
    \begin{align}\label{eq:un1mss}
	 \mu \left(\left\{ -2\underline{\gamma }_{2} \right\}\right)= 0.
	\end{align}
\end{prop}
\begin{proof}
Our proposition follows immediately from  Theorems \ref{thm1} and \ref{thm2}.
\end{proof}

The following corollary is useful to show the existence of unitary tempered representation.

\begin{cor}\label{cor:vtpj}
 There exists an irreducible unitary tempered representation $ \pi $ of $ G $ such that that 
 \begin{align}\label{eq:famz}
	&\pi \left(C^{\fg} \right) = -2 \underline{\gamma} _{2},& \left[\pi^{*}  \otimes E\right]^{K} \neq 0. 
 \end{align}
\end{cor}
\begin{proof}
Our corollary follows immediately from the second statement of Proposition \ref{prop:quth} and the Harish-Chandra's Plancherel formula.  
\end{proof}

\begin{re}\label{re:kaep}
According to Vogan \cite{Vogan-book,Vogan79},  the above $ E$ is the minimal $ K$-type of the unitary tempered representation $ \pi $.
\end{re}

\subsection{Novikov-Shubin type invariant for Casimir} \label{NS invariant}
We present an application of Theorem \ref{thm1} to study the Novikov-Shubin type invariant of locally symmetric spaces. 

The Novikov-Shubin invariants are topological invariants for closed manifolds defined by the large-time behavior of the heat operator of the Hodge Laplacian on the universal cover \cite{Novikov-Shubin, Efremov-Shubin, Lott}. 
Let us introduce a version of Novikov-Shubin type invariant for Casimir operators on locally symmetric spaces.
 
Let $\Gamma$ be a discrete cocompact and torsion free subgroup of $G$. Then, the quotient space $X_\Gamma =\Gamma \backslash G/K$ is a closed smooth manifold. 
Then, the $G$-equivariant Hermitian vector bundle $ F= G \times_{K} E$ on $X$ descends to a Hermitian vector bundle $ F_{\Gamma} $  on $X_\Gamma$. 


The $\Gamma$-trace of $\exp\(-\frac{t}{2}C^{\fg, 
		X}\)$  is computed by the following integral
\begin{align}\label{eq:kpzx}
	\Tr_{\Gamma} \left[\exp\(-\frac{t}{2}C^{\fg, 
	X}\)\right]=\int_{U} {\rm Tr}\left[p_t(x,x)\right]{\rm d}x,
\end{align}
where $U$ is a fundamental domain on $X$ with respect to the $\Gamma$-action, and $p_t(x, x)$ is defined in (\ref{eqptgxx}). 
By \eqref{eqTrG} and \eqref{eq:kpzx}, 
\begin{equation}\label{eq:gammatrace}
\Tr_{\Gamma} \left[\exp\(-\frac{t}{2}C^{\fg, 
		X}\)\right]= \operatorname{vol(X_\Gamma)}\operatorname{Tr}_G\left[\exp\(-\frac{t}{2}C^{\fg, 
		X}\)\right].
\end{equation}

We use the assumptions and notations in Theorem \ref{thm1}.
In the sequel, we assume that 
\begin{align}\label{eq:iody}
	\ul\gamma_{2} \le 0.
\end{align}
By Proposition \ref{prop:quth}, $ \mu$ is supported on $ \left[-2\ul\gamma_{2},\infty\right) \subset \mathbf{R}_{+} $.

\begin{defin}\label{def:suuy}
 The Novikov-Shubin-type invariant $ \alpha_{\Gamma }\in [0, \infty]$ of $X_\Gamma$ associated with $F_{\Gamma} $ is defined by\footnote{The supremum of an empty set by convention is $\infty$.}
 \begin{align}\label{eq:snu3}
	\alpha_{F_{\Gamma} }=\sup\left\{\beta \ge 0 :  \Tr_\Gamma\left[\exp\(-\frac{t}{2}C^{\fg, 
		 X}\)\right]-{\rm vol}\left(X_{\Gamma } \right)\mu\left(0 \right)=  \mathcal{O}\left( t^{-\frac{\beta}{2} }\right),\ t\to \infty\right\}.
 \end{align}
\end{defin}



\begin{prop}\label{prop:novikovshubin}
Suppose \eqref{eq:iody} holds.
\begin{itemize}
\item[i)] If $ \underline{\gamma }_{2} < 0 $ or if $ \underline{\beta }_{1} = \underline{\gamma }_{2} = 0 $,  then Novikov-Shubin invariant $ \alpha_{F_{\Gamma} } = \infty$.
\item[ii)] If $ \underline{\beta }_{1} < 0 $ and if $ \underline{\gamma }_{2} = 0 $, then
\begin{align}
  \alpha_{F_{\Gamma} }=-2\ul\beta_1 \in \mathbf{N}, 
\end{align}
\end{itemize}
\end{prop}
\begin{proof}
	We observe from (\ref{eq:gammatrace}) that $\Tr_\Gamma\left[\exp\(-\frac{t}{2}C^{\fg, 
	X}\)\right]$ and $\Tr_G\left[\exp\(-\frac{t}{2}C^{\fg, 
	X}\)\right]$ are related by the volume of $X_\Gamma$. 
	By Theorems \ref{thm1} and \ref{thm2}, 
	we get our proposition.
\end{proof}

\subsection{Novikov-Shubin invariant for twisted Hodge Laplacian}\label{sNSinvtwist}
Let $ \rho^{V} : G\to \End\left(V \right)$ be an irreducible finite dimensional $ G$ representation.
Assume that $ V$ is equipped with a Hermitian metric $ \left\langle \cdot, \cdot\right\rangle_{V} $ such that $ \mathfrak{p} $ acts symmetrically and $ \mathfrak{k} $ acts anti-symmetrically.

Then, $ W = G\times_{K} V$ is a Hermitian vector bundle on $ X$.
The map $\left(g,v\right)\in G \times V\to \left(g,gv\right)\in G \times V$ induces a $G$-equivariant canonical trivialization, 
\begin{align}\label{eq:rgvf}
 W \simeq X \times V.
\end{align}

Recall that $ \Gamma \subset G$ is cocompact and torsion free.
Then, $ W_{\Gamma } = \Gamma\backslash G\times_{K} V$ is a Hermitian vector bundle on $ X_{\Gamma } $.
By \eqref{eq:rgvf}, we have 
\begin{align}\label{eq:kkqj}
	W_{\Gamma }  = \Gamma \backslash \left(X\times V\right).
\end{align}
Therefore, $ W_{\Gamma } $ is a flat vector bundle with holonomy representation $ \rho^{V} _{|\Gamma} : \Gamma \to {\rm GL}\left(V\right)$. 
For a detailed discussion of the geometry of the flat vector bundle $ W_{\Gamma } $, we refer the reader to \cite{BMZ,Liu21,Shen21}.

We have identifications 
\begin{align}\label{eq:acjc}
 &\Omega^{\cdot} \left(X , W \right)=C^{\infty} \left( G,\Lambda^{\cdot} \left(\mathfrak{p}^{*} \right)\otimes_{\mathbf{R} } V\right)^{K},&\Omega^{\cdot} \left(X_{\Gamma } , W_{\Gamma } \right)=C^{\infty} \left(\Gamma \backslash G,\Lambda^{\cdot} \left(\mathfrak{p}^{*} \right)\otimes_{\mathbf{R} } V\right)^{K}. 
\end{align}
Classically, the associated Hodge Laplacian (see \cite[Proposition 4.2]{Shen21}) are given by 
\begin{align}\label{eq:p3ur}
 &\Box^{X} = C^{\fg,X} -C^{\fg,V},&\Box^{X_{\Gamma } } = C^{\fg,X_{\Gamma } } -C^{\fg,V}.
\end{align}

By spectral theory, we have 
\begin{align}\label{eq:j2vq}
 {\rm Sp}\left(\Box^{X} \right) \subset \mathbf{R}_{+}. 
\end{align}

If $0 \l i \l m$, the Novikov-Shubin invariant $ a_{i,V,\Gamma } $ is defined by a formula similar to \eqref{eq:snu3}, while replacing the Casimir by the Hodge Laplacian $ \Box^{X} $.

Put 
\begin{align}\label{eq:lb2z}
 \delta \left(G\right)= {\rm rk}_{\mathbf{C} } G-{\rm rk}_{\mathbf{C} } K.
\end{align}
Note that $ \delta \left(G\right)$ and $ m$ have the same parity.
Note also that $ \rho^{V}\circ\theta $ is another irreducible  representation of $ G$, where $\theta \in  {\rm Aut}(G)$ is the Cartan involution.

\begin{thm}\label{thm:jwob}
	The following statements hold.
	\begin{enumerate}[\indent a)]
		\item If $ \rho^{V} \not\simeq \rho^{V} \circ\theta $ or if $ i\not\in \left[\frac{m-\delta \left(G\right)}{2}, \frac{m+\delta \left(G\right)}{2}\right]$, then	the spectrum of $ \Box^{X} $ is bounded from below by a positive real number.
	 \item If $ \rho^{V} \simeq \rho^{V} \circ\theta $,  $ \delta \left(G\right)= 0 $, and $ i =  m/2$,  then $ 0\in {\rm Sp}\left( \Box^{X}\right) $ and $ \Box^{X} $ has a spectral gap at $ 0 $. 
	 \item If $ \rho^{V} \simeq \rho^{V} \circ\theta $,  $ \delta \left(G\right) > 0 $, and $ i\in \left[\frac{m-\delta \left(G\right)}{2}, \frac{m+\delta \left(G\right)}{2}\right]$,  then $ 0\in {\rm Sp}\left( \Box^{X}\right) $ and $ \Box^{X} $ has no spectral gap at $ 0 $.  
	 \item The Novikov-Shubin invariant $ \alpha_{i,V,\Gamma } $ is finite if and only if $\rho^{V} \simeq \rho^{V} \circ \theta, \delta\left(G\right)> 0  $ and $ i\in \left[\frac{m-\delta \left(G\right)}{2}, \frac{m+\delta \left(G\right)}{2}\right]$.
	In this case, 
	\begin{align}\label{eq:bcrp}
	 \alpha_{i,V,\Gamma } = \delta \left(G\right).
	\end{align}
	\end{enumerate}
 \end{thm}
 \begin{proof}
	Let $ \mathfrak{h} = \mathfrak{a} \oplus \mathfrak{t} $ be the Cartan subalgebra of $ \mathfrak{g} $ defined in Remark \ref{re:hat}.
	Let $ R_{+} \left(\mathfrak{g},\mathfrak{h}  \right)$ be the positive root system of introduced in Remark \ref{rerhott}.
	
	Let $ \lambda^{V}\in \mathfrak{a}^{*} \oplus \mathfrak{t}_{0}^{*} $ be the highest weight of $ \rho^{V}$.
	By \eqref{eq:bgr3}, we have the classical formula
	\begin{align}\label{eq:ghah}
	- C^{\mathfrak{g}, V} = \left| \varrho^{\mathfrak{g} } + \lambda^{V}_{|\mathfrak{t}_{0} }  \right|^{2} + \left| \lambda^{V}_{|\mathfrak{a} }  \right|^{2} - \left| \varrho^{\mathfrak{g} }  \right|^{2}. 
	\end{align}

	By our convention on $ R_{+} \left(\mathfrak{g}, \mathfrak{h} \right)$, the highest weight of $ \rho^{V}\circ \theta $ is 
	\begin{align}\label{eq:vb3z}
	 \left(-\lambda^{V}_{|\mathfrak{a} }, \lambda^{V}_{|\mathfrak{t}_{0 }  } \right) .
	\end{align}
	Then, we have 
	\begin{align}\label{eq:cjpb}
		\rho^{V} \simeq \rho^{V} \circ\theta \iff \lambda^{V}_{|\mathfrak{a} }= 0. 
	\end{align}

	Let $ E$ be an irreducible $ K$-representation of highest weight 
	\begin{align}\label{eq:mgez}
		\lambda^{E} = \sum_{\alpha \in R_{+} \left(\mathfrak{p} \right) }^{} \alpha + \lambda^{V}_{|\ft_{0 } }.  
	\end{align}
	Using $ \lambda^{E} +2\varrho^{\fk} -\varrho^{\fg} = \varrho^{\fg}+\lambda^{V} _{|\mathfrak{t}_{0 }  } $, by \eqref{eq:ghah}, we have  
	\begin{align}\label{eq:msji}
&	 \Delta_{0}^{1} = \Delta_{0}^{2} = \varnothing, &\underline{\beta }_{1} = - \frac{1}{2}\dim \mathfrak{a} =- \frac{1}{2} \delta(G),&& \underline{\gamma}_{2} + \frac{1}{2} C^{\mathfrak{g}, V}= -  \frac{1}{2} \left| \lambda^{V}_{|\mathfrak{a} }  \right|^{2}.   
	\end{align}
	
	By \eqref{eqpTwei}, using the weight decomposition of $ \Lambda^{i} \left(\mathfrak{p}^{*} \right)$ and $ V$, we see that $ E$ is a subrepresentation of  $ \Lambda^{i} \left(\mathfrak{p}^{*} \right)\otimes V$ if and only if $ i\in \left[\frac{m-\delta \left(G\right)}{2}, \frac{m+\delta \left(G\right)}{2}\right]$.
  Moreover, the highest weight of the other subrepresentations of $ \Lambda^{\cdot} \left(\mathfrak{p}^{*} \right)$ have the form
 \begin{align}\label{eq:kmif}
	\lambda^{E}   - \beta, 
 \end{align}
 where $ \beta $ is a sum of certain roots in $ R_{+} \left(\mathfrak{g} \right)$.
 By Proposition \ref{propv22delta}, the corresponding $\underline{\gamma}_{2} + \frac{1}{2} C^{\mathfrak{g}, V} $ is strictly smaller than $ - \frac{1}{2}\left| \lambda^{V}_{|\mathfrak{a} }  \right|^{2}$.
 Our theorem now follows immediately from \eqref{eq:cjpb}, \eqref{eq:msji}, and the above observations.
 \end{proof}

 \begin{re}\label{re:crxp}
	If $ \rho^{V} $ is trivial, the above results can be found in \cite[Proposition 11.1]{Lo-Me} and \cite[Theorem 1.1]{Olbrich}, whose proof relies on Harish-Chandra's Plancherel formula and Lie algebra cohomology.
	Our proof does not use these tools.
	\end{re}
	
	\begin{re}\label{re:sll1}
		The condition $ \rho^{V} \circ \theta \not\simeq \rho^{V} $ is closely related to Borel-Wallach's vanishing theorem \cite[Theorem VII.6.7]{Borel-Wallach}.
	\end{re}

\section{Large times behavior of the $G$-trace}\label{Spf1}
The purpose of this section is to show our main result Theorems \ref{thm1} and \ref{thm2}. 

This section is organised as follows.
In Section \ref{sWeylint}, using the Weyl's integral formula with respect to $K$, we rewrite our $ G$-trace as an integral $I_{t}$ over $\ft_{0}$.

In Section \ref{sItww}, choosing a compatible positive root system $R_{+}(\fg)$,  we write $I_{t}$ as a sum of integral $I_{t}(w)$ over the  Weyl chambers $w^{-1}C_{+}(\fg)$ with $w\in C_{+}(\fg)$.  

In Section \ref{sAsyI}, we study the asymptotic of $I_{t}(w)$ when 
$t\to \infty$. The detailed proof will be given in Section \ref{SJtda}. 

Finally, in Section \ref{sPfthm1}, we deduce  Theorems \ref{thm1} and \ref{thm2}.

\subsection{An application of Weyl's integral formula}\label{sWeylint}
We use the notation in Section \ref{Smainresult}. In particular, we have fixed a positive root system $R_{+}(\fk)\subset R(\fk)$ for 
$R(\fk)$. 
Recall that $\ft_{0}=\sqrt{-1}\ft$ is a Euclidean space of dimension $ r_{0 } $, and it is 
identified with its dual $\ft_{0}^{*}$ by the Euclidean metric $B_{|\ft_{0}}$. 
Recall also that $\pi^{\fk}$ is a real polynomial on $\ft_{0}$ defined by $R_{+}(\fk)$. 

\begin{defin}
	For $t>0$ and $\mu\in \ft_{0}$, set 
\begin{align}\label{eqIfgt}
I^{\fg}_{t}(\mu)=	
\int_{\ft_{0}}\pi^{\fk}\(Y^{}_0\)\widehat{A}\(\ad\(Y^{}_{0}\)_{|\fp}\)\exp\(\left\<\mu,Y^{}_0\right\>-\frac{\left|Y_{0}\right|^{2}}{2t}\) \frac{dY_0}{(2 \pi t)^{r_{0}/2}}. 
\end{align} 
\end{defin}

\begin{prop}
If $w\in W(\fk)$, for $t>0$ and $\mu\in 
\ft_{0}$, we have 
\begin{align}\label{eqwew}
	I^{\fg}_{t}(w\mu)= \e_{w} I^{\fg}_{t}(\mu). 
\end{align}  
\end{prop} 
\begin{proof}
If $w\in W(\fk)$, there is $k\in N_{K}(T)$ such that  when acting on $\ft_{0}$, 
	\begin{align}
		w=\Ad(k). 
	\end{align} 
Thus,
	\begin{align}\label{eqApw}
		\widehat{A}\(\ad\(w Y_{0}^{}\)_{|\fp}\)=\widehat{A}\(\ad\( 
		Y_{0}^{}\)_{|\fp}\). 
	\end{align} 
By \eqref{eqew}, \eqref{eqIfgt}, \eqref{eqApw},  and by a change 
of variable, we get our proposition. 
\end{proof} 

Let ${\rm vol}(K/T)$ be the Riemannian volume of $K/T$ with respect 
to the Riemannian metric induced by $-B_{|\fk}$. By \cite[Corollary 
7.27]{BGV}, we have 
\begin{align}\label{eqvolKT}
{\rm vol}(K/T)=\[\pi^{\fk}\(\frac{\varrho^{\fk}}{2\pi}\)\]^{-1}. 
\end{align} 

Recall that  $\tau^{E}$ is an irreducible representation of $K$ with highest $\lambda^{E}\in \ft_{0}$.

\begin{prop}
For  $t>0$, we have 
\begin{align}\label{eqTrGwi}
	\Tr_{G}\[\exp\(-\frac{t}{2}C^{\fg,X}\)\]=\frac{1}{(2\pi)^{m/2}}\[\pi^{\fk}\(\varrho^{\fk}\)\]^{-1}t^{-\frac{m+n-r_{0}}{2}}e^{-c_{\fg}t/2}I^{\fg}_{t}\(\lambda^{E}+\varrho^{\fk}\). 
\end{align}
In particular, 
\begin{align}
	I^{\fg}_{t}\(\lambda^{E}+\varrho^{\fk}\)>0. 
\end{align}
\end{prop} 
\begin{proof}
By Theorem \ref{thmB09}, Corollary \ref{corTrp}, and \eqref{eqvolKT}, it is enough to  show	
\begin{multline}\label{eqTrGwi1}
	\int_{\sqrt{-1}\fk}\frac{\widehat 
	A\(\ad\(\Yok\)_{|\fp}\)}{\widehat 
	A\(\ad\(\Yok\)_{|\fk}\)}\Tr\[\tau^{E}\(e^{-\Yok}\)\]\exp\(-\frac{1}{2t}\left|\Yok\right|^{2}\)d\Yok\\
	={\rm vol}(K/T)(2\pi t)^{r_{0}/2}I_{t}^{\fg}\(\lambda^{E}+\varrho^{\fk}\). 
\end{multline} 

Since the integrand of the left hand side of \eqref{eqTrGwi1} is 
$\Ad(K)$-invariant,  using Weyl's integration formula \cite[Lemma 11.4]{Knappsemi} on  
$\sqrt{-1}\fk$,  we get 
\begin{multline}\label{eqWeyl}
	\int_{\sqrt{-1}\fk}\frac{\widehat 
	A\(\ad\(\Yok\)_{|\fp}\)}{\widehat 
	A\(\ad\(\Yok\)_{|\fk}\)}\Tr\[\tau^{E}\(e^{-\Yok}\)\]\exp\(-\frac{1}{2t}\left|\Yok\right|^{2}\)d\Yok\\
	=\frac{\vol\(K/T\)}{\left|W\(\fk\)\right|} 
 \int_{\ft_{0}}\left|\pi^{\fk}\(Y^{}_0\)\right|^{2}\frac{\widehat 
	A\(\ad\(Y^{}_0\)_{|\fp}\)}{\widehat 
	A\(\ad\(Y^{}_0\)_{|\fk}\)}\Tr\[\tau^{E}\(e^{-Y^{}_0}\)\]\exp\(-\frac{1}{2t}\left|Y^{}_0\right|^{2}\)dY^{}_0. 
\end{multline} 
Since $\pi^{\fk}$ is real on $\ft_{0}$,  for $Y^{}_0\in \ft_{0}$, we have 
\begin{align}\label{eqpppi}
	\left|\pi^{\fk}\(Y^{}_0\)\right|^{2}=\[\pi^{\fk}\(Y^{}_0\)\]^{2}.
\end{align} 
By Weyl's character formula \cite[VI.1.7]{BrockerDieck}, for 
$Y_{0}^{}\in \ft_{0}$, we have 
	\begin{align}\label{eqWeyl2}
		\Tr\[\tau^{E}\(e^{-Y^{}_0}\)\]\prod_{\alpha\in 
		R_{+}(\fk)}\(e^{-\left\langle \alpha,Y^{}_0\right\rangle 
		/2}-e^{\left\langle \alpha,Y^{}_0\right\rangle 
		/2}\)=\sum_{w\in W\(\fk\)}\e_{w}e^{-\left\<\lambda^{E}+\varrho^{\fk},wY^{}_0\right\>}. 
	\end{align} 
From \eqref{eqew}, \eqref{eqpppi}, and \eqref{eqWeyl2},  if  
$Y_{0}^{}\in \ft_{0}$, we have 
\begin{align}\label{eqWeyl1}
	\begin{aligned}
	\left|\pi^{\fk}\(Y^{}_0\)\right|^{2}{\widehat 
	A^{-1}\(\ad\(Y^{}_0\)_{|\fk}\)}\Tr\[\tau^{E}\(e^{-Y^{}_0}\)\]
	=&\pi^{\fk}\(-Y^{}_0\)\sum_{w\in 
	W\(\fk\)}\e_{w}e^{-\left\<\lambda^{E}+\rho^{\fk},wY^{}_0\right\>}\\=&\sum_{w\in 
	W\(\fk\)}\pi^{\fk}\(-wY^{}_0\)e^{-\left\<\lambda^{E}+\varrho^{\fk},wY^{}_0\right\>}.
	\end{aligned}
\end{align} 
	
By \eqref{eqIfgt},  \eqref{eqApw}, \eqref{eqWeyl},  \eqref{eqWeyl1}, and by a change of variables, we get \eqref{eqTrGwi1}.  
\end{proof} 
\subsection{The functional $I_{t}^{\fg}(\mu,w)$}\label{sItww}
Recall that in Section \ref{Smainresult}, we have fixed a compatible  positive root system $R_{+}(\fg)$ in $R(\fg)$ such that \eqref{eqlambdaEinC} holds. 

%

	\begin{defin}
	For $t>0$,  $\mu\in \ft_{0}$,  and  $w\in W(\fg)$, set 
\begin{align}\label{eqImw}
I^{\fg}_{t}(\mu,w)=	
\int_{w^{-1} 
C_{+}(\fg)}\pi^{\fk}\(Y_{0}^{}\)\widehat{A}\(\ad\(Y_{0}^{}\)_{|\fp}\)\exp\(\left\<\mu,Y_{0}^{}\right\>-\frac{1}{2t}\left|Y_{0}^{}\right|^{2}\)\frac{dY_0}{(2 \pi t)^{{r_{0}}/{2}}}.
\end{align} 
\end{defin}
By \eqref{eqIfgt} and \eqref{eqImw},  we have 
	\begin{align}\label{eq:Itmu}
		I_{t}^{\fg}(\mu)=\sum_{w\in W(\fg)}I^{\fg}_{t}(\mu,w). 
	\end{align}

We write $w=w_{1}w_{2}$ as in  \eqref{eqwww12}. 

\begin{prop}\label{propIgw}
For $t>0$, $\mu\in \ft_{0}$, and $w\in W(\fg)$, we have  
\begin{align}\label{eqIgnw}
I^{\fg}_{t}(\mu,w)=	
\e_{w_{2}}\int_{C_{+}(\fg)}\pi^{\fk}\(w^{-1}_{1}Y_{0}^{}\)\widehat{A}\(\ad\(w^{-1}_{1}Y_{0}^{}\)_{|\fp}\)\exp\(\left\<w \mu,Y_{0}^{}\right\>-\frac{\left|Y_{0}^{}\right|^{2}}{2t}\)\frac{dY_{0}^{}}{(2\pi t)^{{r_{0}}/{2}}}. 
\end{align} 
In particular, 
	\begin{align}\label{eqew2Ig}
	\e_{w_{2}}	I_{t}^{\fg}(\mu,w)>0. 
	\end{align} 
\end{prop} 
\begin{proof}
	  Equation \eqref{eqIgnw} follows from \eqref{eqew}, \eqref{eqApw},  
	  \eqref{eqImw}, and a change of variables. 
	 By \eqref{eqwR+kR+g}, for $ Y_{0 }\in {\rm Int} \left(C_{+} \left(\mathfrak{g} \right)\right)$, we have  $\pi^{\mathfrak{k} }\(w_{1}^{-1} Y_{0 }  \)> 0 $.
	In particular, the integrand on the right hand side of 	 \eqref{eqIgnw} is positive on  ${\rm Int}\(C_{+}(\fg)\)$,  which gives \eqref{eqew2Ig}. 
\end{proof} 

	
For $x\in \bR$, set 
\begin{align}
	{\rm Td}(x)=\frac{x}{1-e^{-x}}. 
\end{align} 
Then, ${\rm Td}$ is a positive function such that for $x\in\bR$, 
	\begin{align}\label{eqATd}
		\widehat A(x)=e^{-x/2}{\rm Td}(x). 
	\end{align} 
Moreover, there is $C>0$ such that 	for $x\g0$, we have 
	\begin{align}\label{eqTdc}
{\rm Td}(x)\l C(1+|x|).\end{align}

\begin{defin}
	For $w_{1}\in W(\fg, \fk)$, set 
	\begin{align}\label{eqpipihat}
		\begin{aligned}
	\pi_{0}\(w_{1},Y_{0}\)&=\prod_{\alpha\in w_{1}R_{+}(\fk)}\left\langle 
	\alpha,Y_{0}\right\rangle \prod_{\alpha\in w_{1}R(\fp)\cap R_{+}(\fg)}\left\langle 
	\alpha,Y_{0}\right\rangle,\\
	\widehat{\pi}_{0}\(w_{1},Y_{0}\)&=\prod_{\alpha\in w_{1}R_{+}(\fk)}\left\langle 
	\alpha,Y_{0}\right\rangle \prod_{\alpha\in w_{1}R(\fp)\cap 
	R_{+}(\fg)}{\rm Td}\(\left\langle 
	\alpha,Y_{0}\right\rangle\). 
	\end{aligned}
	\end{align} 
\end{defin} 
Then, $\pi_{0}(w_{1},\cdot)$ is a  polynomial of degree $|R_{+}(\fg)|$ on $\ft_{0}$. 
By abuse of notation, we call  $\widehat \pi_{0}(w_{1},\cdot)$ has degree 
$|R_{+}(\fg)|$. Clearly, there is $C>0$  such that for  $Y_0^{}\in C_{+}(\fg)$, 
\begin{align}\label{eqTdexp}
0\l \widehat{\pi}_{0}\(w_{1},Y_{0}\) \l 
	C\(1+\left|Y_0^{}\right|\)^{\left|R_{+}(\fg)\right|}.
\end{align}

\begin{prop}\label{propAtoTod}
	For $w_{1}\in W(\fg, \fk)$ and $Y_0^{}\in {\ft_{0}}$,	we have 
\begin{align}\label{eqAexp1} 
	\pi^{\fk}\(w^{-1}_{1}Y_{0}^{}\) \widehat{A}\(\ad\(w_{1}^{-1}Y_0^{}\)_{|\fp}\)= 
 	\widehat\pi_{0}\(w_{1},Y_{0}\)\exp\(\left\<w_{1}\varrho^{\fk}-\varrho^{\fg},Y_0^{}\right\>\). 
\end{align} 
\end{prop} 
\begin{proof}
		By \eqref{eqpipihat}, it is enough to show
	\begin{align}\label{eqAexp}
	\widehat 
	A\(\ad\(w_{1}^{-1}Y_0^{}\)_{|\fp}\)= 
	\exp\(\left\<w_{1}\varrho^{\fk}-\varrho^{\fg},Y_0^{}\right\>\)
 	{\prod_{\substack{\alpha\in R(\fp)\\ w_{1}\alpha>0}} {\rm 
	Td}\(\left\langle w_{1}\alpha,Y_0^{}\right\rangle \)}.  
\end{align} 

We have a disjoint union  
\begin{align}\label{eqRfp}
	R(\fp)=\{\alpha\in R(\fp): w_{1}\alpha>0\}\sqcup \{\alpha\in 
	R(\fp): w_{1}\alpha<0\}. 
\end{align} 
Observes that the map $\alpha\in R(\fp)\to -\alpha \in R(\fp)$ 
interchanges the two subsets on the right hand side of \eqref{eqRfp}.  
Since $\widehat A$ is an even function, by the above observation, we 
have 
\begin{align}\label{eq526}
{	\widehat{A}}\(\ad\(w^{-1}_{1}Y_{0}^{}\)_{|\fp}\)= 
\prod_{\substack{\alpha\in R(\fp) \\ 
	w_{1}\alpha>0 }} {	\widehat{A}} \(
\left\<\alpha, w_{1}^{-1}Y_{0}\right\>\). 
\end{align} 
By \eqref{eqATd} and \eqref{eq526}, we  get  
\begin{align}\label{eqAtoTod0}
	\begin{aligned}
{	
\widehat{A}}\(\ad\(w^{-1}_{1}Y_{0}^{}\)_{|\fp}\)=&\exp\Bigg(-\frac{1}{2}\sum_{\substack{\alpha\in R(\fp) \\ 
	w_{1}\alpha>0 }}\left\<\alpha,w_{1}^{-1}Y_{0}^{}\right\>\Bigg)\prod_{\substack{\alpha\in 
	R(\fp)\\ w_{1}\alpha>0}}{\rm Td}\(\left\langle 
	\alpha,w^{-1}_{1}Y_{0}^{}\right\rangle \)\\
	=&\exp\Bigg(-\frac{1}{2}\sum_{\substack{\alpha\in R(\fp) \\ 
	w_{1}\alpha>0 }}\left\<w_{1}\alpha,Y_{0}^{}\right\>\Bigg)\prod_{\substack{\alpha\in 
	R(\fp)\\ w_{1}\alpha>0}}{\rm Td}\(\left\langle 
w_{1}	\alpha,Y_{0}^{}\right\rangle \).
\end{aligned}
\end{align} 


We claim that  
	\begin{align}\label{eqAtoTod}
-\frac{1}{2}\sum_{\substack{\alpha\in 
R(\fp)\\ 
	w_{1}\alpha>0}}\alpha=\varrho^{\fk}-w_{1}^{-1}\varrho^{\fg}. 
	\end{align} 
Indeed, we can write 
	\begin{align}\label{eqAToTod1}
		-\frac{1}{2}\sum_{\substack{\alpha\in R(\fp) \\ 
	w_1\alpha>0 }}\alpha=\frac{1}{2}\sum_{\substack{\alpha\in R(\fk) \\ 
	w_1\alpha>0 }}\alpha-\frac{1}{2}\Big(\sum_{\substack{\alpha\in R(\fp) \\ 
	w_1\alpha>0 }}\alpha+\sum_{\substack{\alpha\in R(\fk) \\ 
	w_1\alpha>0 }}\alpha\Big). 
	\end{align} 
By \eqref{eqwR+kR+g1}, we have  
	\begin{align}\label{eqAToTod15}
		\frac{1}{2}\sum_{\substack{\alpha\in R(\fk) \\ 
	w_{1}\alpha>0 }}\alpha=\varrho^{\fk}. 
	\end{align} 
By the observation given after Remark \ref{rerhott}, 	we have 
	\begin{align}\label{eqAToTod2}
		\frac{1}{2}\Big(\sum_{\substack{\alpha\in R(\fp) \\ 
	w_{1}\alpha>0 }}\alpha+\sum_{\substack{\alpha\in R(\fk) \\ 
	w_1\alpha>0 }}\alpha\Big)=\frac{1}{2}\sum_{\substack{\alpha\in 
	R(\fg)\\ 
	w_{1}\alpha>0 }}\dim \fg_{\alpha}\cdot \alpha=w_{1}^{-1}\varrho^{\fg}. 
	\end{align} 
By 	\eqref{eqAToTod1}-\eqref{eqAToTod2}, we get our claim \eqref{eqAtoTod}. 
		
		By \eqref{eqAtoTod0} and \eqref{eqAtoTod}, we get 
		\eqref{eqAexp} and finish the proof of our proposition. 
\end{proof} 

\begin{defin}\label{defmuw1}
If $\mu\in \ft_{0}$ and if $w=w_{1}w_{2}\in W(\fg)$ satisfying \eqref{eqwww12}, denote by 
\begin{align}\label{eqmuwul}
\ul \mu(w)=w\mu+w_{1}\varrho^{\fk}-\varrho^{\fg}. 
\end{align} 
If $w=1$, we write 
\begin{align}
\ul \mu= \mu+\varrho^{\fk}-\varrho^{\fg}. 
\end{align} 
\end{defin} 

\begin{cor}\label{corIgmuw1}
For $t>0$,  $\mu\in \ft_{0}$, and $w=w_{1}w_{2}\in W$ satisfying \eqref{eqwww12},  we have 
	\begin{align}\label{eqIgmuw1}
	I^{\fg}_{t}(\mu,w)=\e_{w_{2}}\int_{C_{+}(\fg)}	
\widehat\pi_{0}\(w_{1}, Y_{0}\)
	\exp\(\left\<\ul \mu(w),Y_{0}^{}\right\>-\frac{1}{2t}|Y_{0}^{}|^{2}\)\frac{dY_{0}^{}}{(2\pi t)^{r_{0}/2}}. 
	\end{align} 
\end{cor}
\begin{proof}
This is a consequence of Propositions \ref{propIgw} and \ref{propAtoTod}, and \eqref{eqmuwul}.
\end{proof} 	

Denote by $\pi_{0}^{\fg}(w_{1},\cdot)$, $\widehat \pi_{0}^{\fg}(w_{1},\cdot)$ the restrictions of 
$\pi_{0}(w_{1},\cdot)$, $\widehat\pi_{0}^{\fg}(w_{1},\cdot)$ to $\ft_{0}^{\fg}$. If $Y_{0}=Y_{0}^{\fg}+Y_{\fg}\in \mathfrak{t}_{0}^{\mathfrak{g} } \oplus \mathfrak{t}_{\mathfrak{g} } $, we have 
	\begin{align}\label{eqEdexp11}
&		\pi_{0}^{\fg}\(w_{1} , Y_{0}^{\fg}\)=	\pi_{0}\(w_{1} , 
		Y_{0}\), & \widehat \pi_{0}^{\fg}\(w_{1} , Y_{0}^{\fg}\)=	\widehat{\pi} _{0}\(w_{1} , Y_{0}\). 
	\end{align} 
Moreover, if $\mu_{0}=\mu_{0}^{\fg}+\mu_{\fg}$, we use the obvious notation 
	\begin{align}
		\ul\mu(w)=\ul\mu_{0}^{\fg}(w)+\mu_{\fg}. 
	\end{align} 


\begin{cor}
For $t>0$, $\mu_{0}=\mu_{0}^{\fg}+\mu_{\fg}$,  and $w=w_{1}w_{2}\in W$ satisfying \eqref{eqwww12}, we have
	\begin{multline}\label{eq532}
	I^{\fg}_{t}(\mu,w)=\e_{w_{2}} 
	\exp\(\frac{t}{2}\left|\mu_{\fg}\right|^{2}\)\\
	\times \int_{C_{0}^{\fg}}	
	\widehat \pi_{0}^{\fg}\(w_{1}, Y_{0}^{\fg}\)
	\exp\(\left\<\ul 
	\mu_{0}^{\fg}(w),Y_{0}^{\fg}\right\>-\frac{1}{2t}\left|Y_{0}^{\fg}\right|^{2}\)\frac{dY_{0}^{\fg}}{(2\pi t)^{r_{0}^{\fg}/2}}. 
	\end{multline} 	
\end{cor} 
\begin{proof}
	This is a consequence of \eqref{eqCgot} and Corollary \ref{corIgmuw1}. 
\end{proof}

\subsection{Intermediate results}\label{sAsyI}
%

  \begin{thm}\label{thm32}
	  For $w\in W(\fg)$, there exist $\alpha _{w}>0$, $\beta _{w}\in \frac{1}{2} \mathbf Z$,  and $\gamma_{w}\g0$ such that  as $t\to \infty$,
	  \begin{align}
		I^{\fg}_{t}\(\lambda^{E} + \varrho^{\mathfrak{k} }  ,w\)\sim  \e_{w_{2}}\alpha_{w}t^{\beta_{w}}e^{\gamma_{w}t}. 
	\end{align} 
 \end{thm}
 \begin{proof}
The proof of this Theorem will be presented in Section 
\ref{sPthm32}. 
\end{proof} 
 
	
Let $W^{1}_{0},W^{2}_{0}\subset W(\fg)$ be the Weyl group associated to 
 $\Delta_{0}^{1},\Delta_{0}^{2}$ defined in Section \ref{Smainresult}. 
 
 \begin{thm}\label{thm33}
	For $w\in W(\fg)$,  we have 
	 	\begin{align}\label{eqthm33}
		\gamma_{w}\l \gamma_{2},
	\end{align}
	where the equality holds if and only if $w\in W^{2}_{0}$. 
\end{thm} 
\begin{proof}
	The proof of this Theorem will be presented in Section 
\ref{sPthm33}. 
\end{proof} 
 
 \begin{thm}\label{thm34}
For $w\in W^{2}_{0}$, we have
	\begin{align}\label{eqthm34}
		\beta_{w}\l \beta_{1}, 
	\end{align} 
where the equality holds if and only if $w\in W^{1}_{0}$. 
 \end{thm} 
 \begin{proof}
	The proof of this Theorem will be presented in Section 
\ref{sPthm34}. 
\end{proof}

 \begin{thm}\label{thm35}
	 The following identity holds,  
	\begin{align}
		\sum_{w\in W_{0}^{1}}\e_{w_{2}}\alpha_{w}=\alpha_{0}. 
	\end{align} 
\end{thm} 
 \begin{proof}
	The proof of this Theorem will be presented in Section 
\ref{sPthm35}. 
\end{proof} 

\begin{thm}\label{thm:ldme}
	Under one of the equivalent conditions in Theorem \ref{thm2}, if $ w= 1$, there exists $ \epsilon_{0} > 0 $ such that as $ t\to \infty$,  
	\begin{align}\label{eq:aiis}
		I^{\fg}_{t}\(\lambda^{E} + \varrho^{\mathfrak{k} }  ,w\)\sim  \alpha_{0 }t^{\beta_{1}}e^{\gamma_{2}t}\left(1+\mathcal{O} \left(e^{-\epsilon_{0} t} \right)\right). 
	\end{align}
\end{thm}
\begin{proof}
The proof of this Theorem will be presented in Section \ref{sPthm32}.
\end{proof} 

\subsection{Proof of Theorems \ref{thm1} and \ref{thm2}}\label{sPfthm1}
We use (\ref{eq:Itmu}) to compute $I_{t}^{\fg}\(\lambda^{E}+\varrho^{\fk}\)$ as the sum over $I_{t}^{\fg}\(\lambda^{E}+\varrho^{\fk},w \)$. The asymptotics of each term $I_{t}^{\fg}\(\lambda^{E}+\varrho^{\fk},w \)$ can be computed using Theorem \ref{thm32}. 
The dominant terms as $t\to \infty$ are characterized by Theorems \ref{thm33}-\ref{thm35}.
We conclude that as $t\to \infty$, 
\begin{align}\label{eqItftda}
I_{t}^{\fg}\(\lambda^{E}+\varrho^{\fk}\) \sim \alpha_{0} 
t^{\beta_{1}} e^{\gamma_{2}t}. 
\end{align} 
By \eqref{eqa0},  \eqref{eqTrGwi},  and \eqref{eqItftda}, we get Theorem \ref{thm1}. 

Assume now one of the equivalent conditions in Theorem \ref{thm2} holds.
Then, $ W_{0 }^{2} $ is trivial. Accordingly, by Theorems \ref{thm32} and \ref{thm33}, if $ w= 1$, there is $ \epsilon_{0} > 0 $ such that  as $ t \to \infty$,
\begin{align}\label{eq:nhas}
	I_{t}^{\fg}\(\lambda^{E}+\varrho^{\fk}\) = I_{t}^{\fg}\(\lambda^{E}+\varrho^{\fk},w\) +\mathcal{O} \left(e^{-\left(\epsilon_{0} +\gamma_{2} \right)t} \right).
\end{align}
By Theorem \ref{thm:ldme} and by \eqref{eq:nhas}, we get Theorem \ref{thm2}.
 
\section{A functional $J_{t}$ and its asymptotic as $t\to 
\infty$}\label{SJtda}
In the whose section, we fix $\mu \in \ft_{0}$. 
Sometimes, we will not write down the dependence on $\mu \in \ft_{0}$, although most of the constants in estimates will depend on this parameter. 

This section is organised as follows. 
In Section \ref{sJt}, we introduce a functional $J_{t}$ on $\ft_{0}^{\fg}$, which is closely related to $I_{t}^{\fg}\left(w,\mu_{0} \right)$.

In Section \ref{sLTJ}, we state the main result of this section, which gives an asymptotic for $J_{t}$ as $t\to \infty$.
 
In Section \ref{sSCV}, we study the action $S_{w}$ (Definition \ref{def:S}) associated to the functional $J_{t}$.

In Sections \ref{sLoc}-\ref{ssLast}, we establish the asymptotics of $J_{t}$ using Laplace's method.

Finally, in Sections \ref{sPthm32}-\ref{sPthm35}, we deduce Theorems \ref{thm32}-\ref{thm:ldme}.

\subsection{A functional $J_{t}$}\label{sJt}
	

%

\begin{defin}\label{def:S}
For  $w\in W(\fg)$ and $Y_{0}^{\fg}\in \ft_{0}^{\fg}$, set 
		\begin{align}\label{eqS}
		S_{w}\(Y_{0}^{\fg}\)=\frac{1}{2}\left|Y_{0}^{\fg}\right|^{2}-\left\langle Y_{0}^{\fg},\ul\mu_{0}^{\fg}(w)\right\rangle.
	\end{align} 
\end{defin} 
By \eqref{eqS},  we have  
	\begin{align}\label{eqSS2}
		S_{w}\(Y_{0}^{\fg}\)=\frac{1}{2}\left|Y_{0}^{\fg}-\ul\mu_{0}^{\fg}(w)\right|^{2}-\frac{1}{2}\left|\ul \mu_{0}^{\fg}(w)\right|^{2}.
	\end{align} 

%
%
%

\begin{defin}
		If $t>0$, set 
			\begin{align}\label{eqItv}
		J_{t}\(w,\mu_{0}^{\fg}\)=\(\frac{t}{2\pi}\)^{r_{0}^{\fg}/2}\int_{ C_{0}^{\fg}}	
	\widehat \pi_{0}^{\fg}\(w_{1}, tY_{0}^{\fg}\)\exp\(-tS_{w}\(Y_{0}^{\fg}\)\)dY_{0}^{\fg}.
	\end{align}
\end{defin}

\begin{re}
After a rescaling on the variable $Y_{0}^{\fg}$, we have 
	\begin{align}\label{eqItvr}
	J_{t}\(w,\mu_{0}^{\fg}\)=\int_{ C_{0}^{\fg}}	
	\widehat \pi_{0}^{\fg}\(w_{1}, Y_{0}^{\fg}\)\exp\(\left\langle \ul \mu_{0}^{\fg}(w),
	Y_{0}^{\fg}\right\rangle -\frac{|Y_{0}^{\fg}|^{2}}{2t}\) \frac{d Y_{0}^{\fg}}{(2\pi t)^{ 
	r_{0}^{\fg}/2}}. 
	\end{align} 
By \eqref{eq532} and \eqref{eqItvr}, we have 
	\begin{align}\label{eqIJwmu}
		I_{t}^{\fg} \(w, 
		\mu_{0}\)=\e_{w_{2}}\exp\(\frac{t}{2}\left|\mu_{\fg}\right|^{2}\)J_{t}\(w,\mu_{0}^{\fg}\). 
	\end{align} 
\end{re}

\subsection{Large time behavior of $J_{t}$}\label{sLTJ}
Let $\Delta_{0}^{1}(w) \subset \Delta_{0}^{2}(w)$ be two subsets of $\Delta_{0}^{\fg}$ associated to $\ul\mu_{0}^{\fg}(w)$ as in the Proposition \ref{propLC1}. 
We have the decomposition 
	\begin{align}\label{eqt012w}
		\ft_{0}^{\fg}=\ft_{0}^{1}(w)\oplus\ft_{1}^{2}(w)\oplus \ft_{2}^{\fg}(w).
	\end{align}
If $\widehat C_{0}^{1}(w)$ and $C_{1}^{2}(w)$ are the obvious associated cones, we have 
\begin{align}
&\ul\mu_{0}^{\fg}(w)=\ul\mu_{0}^{1}(w)+\ul\mu_{2}^{\fg}(w), & \ul\mu_{0}^{1}(w)\in -{\rm Int}\(\check 
C_{0}^{1}(w)\), &&\ul\mu_{2}^{\fg}(w)\in {\rm Int}\(C_{2}^{\fg}(w)\). 
\end{align} 

Let $\fq^{1}(w)\subset \fq^{2}(w)\subset \fg_{\bC}$ be the associated standard parabolic subalgebras.
By \eqref{eqflq0}, with obvious notations, we have 
\begin{align}
\fq^{0}=\fl^{0}_{\bC}\oplus \fu_{0}^{1}(w)\oplus \fu_{1}^{2}(w)\oplus 
\fu_{2}(w). 
\end{align} 


\begin{defin}\label{defA0123}
	Put  
		\begin{align}\label{eqA0123}
		&A_{0}^{1}(w)=\left\{f\in w_{1}R_{+}(\fk): 
		f_{|\ft_{1}^{\fg}(w)}=0\right\}, \notag\\
		&A_{1}^{2}(w)=\left\{f\in 
		w_{1}R_{+}(\fk): f_{|\ft_{1}^{\fg}(w)}\neq0, 
		f_{|\ft_{2}^{\fg}(w)}=0\right\},\\&A_{2}(w)=\left\{f\in 
		w_{1}R_{+}(\fk): f_{|\ft_{2}^{\fg}(w)}\neq0\right\}. \notag
	\end{align} 
	Similarly, we define $B_{0}^{1}(w), B_{1}^{2}(w), B_{2}(w)$ by replacing $w_{1}R_{+}(\fk)$ with $w_{1}R(\fp)\cap R_{+}(\fg)$ in the above definitions.
\end{defin} 

We have disjoint unions
	\begin{align}
		\begin{aligned}
	w_{1}R_{+}(\fk)&=	A_{0}^{1}(w)\sqcup A_{1}^{2}(w)\sqcup 
	A_{2}(w),\\ w_{1}R(\fp)\cap R_{+}(\fg)&=	B_{0}^{1}(w)\sqcup 
	B_{1}^{2}(w)\sqcup 
	B_{2}^{}(w).
	\end{aligned}
	\end{align} 

\begin{prop}\label{propABB}
If $w\in W(\fg)$, we have 
\begin{align}\label{eqABB}
&\left|A_{0}^{1}(w)\right|+\left|B_{0}^{1}(w)\right|= \dim 
\fu_{0}^{1}(w),\notag\\&\left|A_{1}^{2}(w)\right|+\left|B_{1}^{2}(w)\right|= \dim 
\fu_{1}^{2}(w), \\&\left|A_{2}(w)\right|+\left|B_{2}(w)\right|= \dim 
\fu_{2}(w). \notag
\end{align} 
\end{prop} 
\begin{proof}
Let us show that last equations in \eqref{eqABB}. 
Since $R(\fk)=R_{+}(\fk)\sqcup\( -R_{+}(\fk)\)$, and since the condition $ f _{|\ft_{2}^{\fg}(w)}\neq0$ is preserved when $ f $ is replaced by $- f $, by \eqref{eqA0123}, we have  
\begin{align}\label{eqA2w}
\left| A_{2}(w)\right|=\frac{1}{2}\left|\left\{ f \in 
w_{1}R\(\fk\): 
 f _{|\ft_{2}(w)\neq0}\right\}\right|. 
\end{align} 
Similarly, we have 
\begin{align}\label{eqB2w}
\left | B_{2}(w)\right|=\frac{1}{2}\left|\left\{ f \in 
w_{1}R\(\fp\): 
 f _{|\ft_{2}(w)}\neq0\right\}\right|.
\end{align} 
By Proposition \ref{propRg}, \eqref{eqA2w}, and \eqref{eqB2w}, 
using $w_{1}R(\fg)=R(\fg)$, we have 
\begin{align}\label{eqAB2w}
\left | A_{2}(w)\right|+\left | B_{2}(w)\right|=\frac{1}{2} \sum_{\alpha\in 
R(\fg): \alpha_{|\ft_{2}(w)}\neq0} \dim 
\fg_{w_{1}^{-1}\alpha}. 
\end{align} 
By Proposition \ref{propWgainv} and \eqref{eqAB2w}, we have 
\begin{align}
\left | A_{2}(w)\right|+\left | B_{2}(w)\right|=\frac{1}{2} \sum_{\alpha\in 
R(\fg): \alpha_{|\ft_{2}(w)}\neq0} \dim 
\fg_{\alpha}= \sum_{\alpha\in 
R_{+}(\fg): \alpha_{|\ft_{2}(w)}\neq0} \dim 
\fg_{\alpha},
\end{align} 
from which we get the last equation in \eqref{eqABB}.

The second equation in \eqref{eqABB} can be established in a similar way.
The first equation in \eqref{eqABB} then follows from the last two equations.
\end{proof} 

\begin{defin}\label{definQQQ}
	For $Y_{0}^{1}\in \ft_{0}^{1}(w)$, set  
		\begin{align}\label{eqQ123}
	\widehat 	\pi_{0}^{1}\(w, Y_{0}^{1}\)=\prod_{f\in A_{0}^{1}(w)}\left\langle f , Y_{0}^{1}
		\right\rangle \prod_{f\in B_{0}^{1}(w)}{\rm Td}\(\left\langle f 
		, Y_{0}^{1}\right\rangle\). 
	\end{align} 
For $Y_{1}^{2}\in \ft_{1}^{2}(w)$, set 
			\begin{align}
		\pi_{1}^{2}\(w, Y_{1}^{2}\)=\prod_{f\in A_{1}^{2}(w) }\left\langle f , Y_{1}^{2}
		\right\rangle \prod_{f\in B_{1}^{2}(w) }\left\langle f , 
		Y_{1}^{2}\right\rangle. 
	\end{align} 
For $Y_{2}^{\fg}\in \ft_{2}^{\fg}(w)$, set 
			\begin{align}\label{eqQ1235}
		\pi_{2}^{\fg}\(w,Y_{2}^{\fg}\)=\prod_{f\in A_{2}^{\fg}(w) }\left\langle f , Y_{2}^{\fg}
		\right\rangle \prod_{f\in B_{2}^{\fg} (w)}\left\langle f 
		,Y_{2}^{\fg} \right\rangle. 
	\end{align} 
\end{defin} 

By Proposition \ref{propABB}, the degrees of the functions in \eqref{eqQ123}-\eqref{eqQ1235} are respectively $\dim \fu_{0}^{1}(w)$, $\dim \fu_{1}^{2}(w)$, and $\dim \fu_{2}(w)$.

For simplicity of notation, we denote by $P_{0}^{1}$, $P_{1}^{2}$, $P_{2}^{\fg}$ the orthogonal projections with respect to the decomposition \eqref{eqt012w}. 
We have the following inclusion, 
 	\begin{align}\label{eqABCCC}
A_{0}^{1}(w)\cup B_{0}^{1}(w)&\subset  
 \check C_{0}^{1}(w)\backslash \{0\},\notag\\
		P_{1}^{2}\(A_{1}^{2}(w)\cup B_{1}^{2}(w)\)&\subset \check 
		C_{1}^{2}(w)\backslash \{0\}, \\
		P_{2}^{\fg}\(A_{2}^{\fg}(w)\cup B_{2}^{\fg}(w)\)&\subset 
\check		C_{2}^{\fg}(w)\backslash \{0\}. \notag
	\end{align} 
 In particular, $\widehat 	\pi_{0}^{1}$, $\pi_{1}^{2}$, and 
 $\pi_{2}^{\fg}$ satisfy similar properties as $\pi_{0}$ and 
 $\widehat\pi_{0}$. 

\begin{defin}
	Define 
		\begin{align}\label{eqa101212}
			\begin{aligned}
	\alpha_{0}^{1}\left(w,\mu_0^{\fg}\right)&=	\int_{Y_{0}^{1}\in 
		C_{0}^{1}(w)}\widehat \pi_{0}^{1}\(w, Y_{0}^{1}\)\exp\(\left\langle 
		Y_{0}^{1},\ul\mu_{0}^{1}(w)\right\rangle 
		\)\frac{dY_{0}^{1}}{(2\pi)^{r_{0}^{1}(w)/2}},\\
		\alpha_{1}^{2}\left(w,\mu_0^{\fg}\right)&=\int_{Y_{1}^{2}\in 
		C_{1}^{2}(w)}\pi_{1}^{2}\(w, 
		Y_{1}^{2}\)\exp\(-\frac{1}{2}\left|Y_{1}^{2}\right|^{2}\)\frac{dY_{1}^{2}}{(2\pi)^{r_{1}^{2}(w)/2}},\\
		\alpha_{2}^{\fg}\left(w,\mu_0^{\fg}\right)&=\pi_{2}^{\fg}\(w,\ul \mu_{2}^{\fg}(w)\). 
	\end{aligned} 
		\end{align} 
	Set
		\begin{align}\label{eqrAc}
		\alpha\(w,\mu_{0}^{\fg}\)=\alpha_{0}^{1}(w)\alpha_{1}^{2}(w)\alpha_{2}^{\fg}(w). 
	\end{align} 
\end{defin} 

The integrands in $\alpha_{0}^{1}(w)$ and $\alpha_{1}^{2}(w)$ are integrable, and is positive on the corresponding open cones by \eqref{eqABCCC}. 
Then, 
	\begin{align}
		&\alpha_{0}^{1}(w)>0, &\alpha_{1}^{2}(w)>0, 
		&&\alpha_{2}^{\fg}(w)>0,
	\end{align} 
	so that 
		\begin{align}
		\alpha\(w,\mu_{0}^{\fg}\)>0. 
	\end{align} 
	
\begin{defin}
	Set
	\begin{align}\label{eqrmu0}
	\beta\(w,\mu_{0}^{\fg}\)=-\frac{1}{2}r_{0}^{1}(w)+\frac{1}{2}\dim_{\mathbf{C} }  
	\fu_{1}^{2}(w)+\dim_{\mathbf{C} } \fu_{2}^{\fg}(w)\in \frac{1}{2} \mathbf{Z}. 		
	\end{align} 
\end{defin}

	
	
The main results of this section are the following. 

\begin{thm}\label{thmIt}
Given $ \mu_{0}^{\mathfrak{g} }\in \mathfrak{t}_{0}^{\mathfrak{g} } $ and $ w\in W\left(\mathfrak{g} \right)$, as $t\to \infty$, we have 
	\begin{align}\label{eqItasy}
		J_{t}\(w,\mu_{0}^{\fg}\)\sim  \alpha\(w,\mu_{0}^{\fg}\)
		t^{\beta\(w,\mu_{0}^{\fg}\)}\exp\(\frac{t}{2}\left|\ul\mu_{2}^{\fg}(w)\right|^{2}\). 
	\end{align} 
\end{thm}
\begin{proof}
	The proof of our theorem will be given in Sections \ref{sSCV}-\ref{sTinfty}. 
	We will apply Laplace's method, which relies on a detailed analysis of the minimum of $S_{w}$ on $C_{0}^{\fg}$ and its local behavior near the minimal point.
\end{proof}

\begin{thm}\label{thm:hjmy}
 Given $ \mu_{0}^{\mathfrak{g} }\in \mathfrak{t}_{0}^{\mathfrak{g} } $ and $ w= 1$ such that $ \Delta_ {0}^{1} = \Delta_{0}^{2} = \varnothing $, there exits $\epsilon_0>0$ such that as $ t \to \infty $, we have 
 \begin{align}\label{eq:3jnt}
	J_{t}\(1,\mu_{0}^{\fg}\) \sim \alpha\(1,\mu_{0}^{\fg}\)	t^{\beta\(1,\mu_{0}^{\fg}\)}\exp\(\frac{t}{2}\left|\ul\mu_{0 }^{\fg}\right|^{2}\)\left(1+\mathcal{O} \left(e^{-\epsilon_{0} t} \right)\right).
 \end{align}
\end{thm}

\begin{re}
	When $\Delta_{0}^{2}(w)=\varnothing$, i.e., 
	$\ul\mu_{0}^{\fg}(w)=\ul\mu_{2}^{\fg}(w)$ is in the  interior of 
	$C_{0}^{\fg}$. By \eqref{eqItasy}, as $t\to 
	\infty$, we have 
		\begin{align}
		J_{t}\(w,\mu_{0}^{\fg}\)\sim  \pi_{0}^{\fg}\(w,\ul\mu_{0 }^{\mathfrak{g} } (w)\) t^{\dim_{\mathbf{C}} \fu_{0}}\exp\(\frac{t}{2}\left|\ul\mu_{0 }^{\fg}(w)\right|^{2}\),
	\end{align} 
which is  a well-known consequence of the Laplace method \cite[Section IV.2.5]{Dieudonne68}. 
\end{re}

\begin{re}
	When $\Delta_{0}^{1}(w)=\Delta_{0}^{\fg}$, i.e., $\ul\mu_{0}^{\fg}(w)$ is in the 
	interior of $-\check C_{0}^{\fg}$ and $\ul\mu_{2}^{\fg}(w)=0$. Then,
	$\exp\(\left\langle Y_{0}^{\fg},\ul\mu_{0}^{\fg}(w)\right\rangle \)$ is 
	integrable on $C_{0}^{\fg}$. 	By \eqref{eqItasy}, as 	$t\to \infty$, we have  
	\begin{align}
	J_{t}\(w,\mu_{0}^{\fg}\) \sim t^{-r_{0}^{\fg}/2}  \int_{Y_{0}^{\fg}\in C_{0}^{\fg}}	
	\widehat \pi_{0}^{\fg}\(w, Y_{0}^{\fg}\)\exp\(\left\langle 
	Y_{0}^{\fg},\ul\mu_{0}^{\fg}(w)\right\rangle \)\frac{dY_{0}^{\fg}}{(2\pi)^{r_{0}^{\fg}/2}},
	\end{align} 
	which can be  obtained  directly from  \eqref{eqItvr}. 
\end{re}

\subsection{Study of the action $S_{w}$}\label{sSCV}
By \eqref{eqv00} and \eqref{eqSS2}, we see that  
	\begin{align}
		\min_{C_{0}^{\fg}} S_{w}= 
		S_{w}\(\ul\mu_{2}^{\fg}(w)\)=-\frac{1}{2}\left|\ul\mu_{2}^{\fg}(w)\right|^{2}.
	\end{align} 

Let us study the local behaviors of $S_{w}$ near $\ul\mu_{2}^{\fg}(w)$. 
In the statement of the following proposition, we use  \eqref{eqy123} for the expression $Y_{0}^{\fg}=Y_{1}+Y_{2}+Y_{3}$. 

\begin{prop}
	For $Y_{0}^{\fg}\in \ft_{0}^{\fg}$, we 	have 
		\begin{multline}\label{eqS0}
		S_{w}\(Y_{0}^{\fg}+\ul\mu_{2}^{\fg}(w)\)=S_{w}\(\ul\mu_{2}^{\fg}(w)\)+\frac{1}{2}\left|P_{0}^{2}Y_{1}\right|^{2}	+\frac{1}{2}\left|P_{1}^{2}Y_{2}\right|^{2}	+\left\<P_{1}^{2}Y_{1},P_{1}^{2}Y_{2}\right\>\\+\frac{1}{2}\left|P_{2}^{\fg}\(Y_{1}+Y_{2}\)+Y_{3}\right|^{2}-\left\<P_{0}^{1}Y_{1},\ul\mu_{0}^{1}(w)\right\>. 
	\end{multline} 
	If 	$Y_{0}^{\fg}+\ul\mu_{2}^{\fg}(w)\in C_{0}^{\fg}$, we have 
	\begin{align}\label{eqPY23p}
&	\left\<P_{1}^{2}Y_{1},P_{1}^{2}Y_{2}\right\>\g 
0,&\left\<P_{0}^{1}Y_{1},\ul\mu_{0}^{1}(w)\right\>\l 0. 
\end{align} 
\end{prop} 
\begin{proof}
		By \eqref{eqS}, we have 
	\begin{align}\label{eqS1}
		S_{w}\(Y_{0}^{\fg}+\ul\mu_{2}^{\fg}(w)\)=S_{w}\(\ul\mu_{2}^{\fg}(w)\)+\frac{1}{2}\left|Y_{0}^{\fg}\right|^{2}-\left\<Y_{0}^{\fg},\ul\mu_{0}^{1}(w)\right\>. 
	\end{align} 

Using the trivial  formula, 
	\begin{align}
		\left|Y_{0}^{\fg}\right|^{2}=\left|P_{0}^{2}Y_{0}^{\fg}\right|^{2}+\left|P_{2}^{\fg}Y_{0}^{\fg}\right|^{2},
	\end{align} 
and 
	\begin{align}
		&P_{0}^{2}Y_{0}^{\fg}=P_{0}^{2}Y_{1}+P_{1}^{2}Y_{2}, & 
		P_{2}^{\fg}Y_{0}^{\fg}=P_{2}^{\fg}(Y_{1}+Y_{2})+Y_{3}, 
	\end{align} 
we get 
	\begin{align}\label{eqS15}
	\frac{1}{2}\left|Y_{0}^{\fg}\right|^{2}=	
	\frac{1}{2}\left|P_{0}^{2}Y_{1}\right|^{2}	
	+\frac{1}{2}\left|P_{1}^{2}Y_{2}\right|^{2}	
	+\left\<P_{1}^{2}Y_{1},P_{1}^{2}Y_{2}\right\>+\frac{1}{2}\left|P_{2}^{\fg}\(Y_{1}+Y_{2}\)+Y_{3}\right|^{2}.
	\end{align} 
	
By \eqref{eqy123} and \eqref{eqvFG}, we have 
\begin{align}\label{eqS2}
	\left\<Y_{0}^{\fg},\ul\mu_{0}^{1}\left(w\right)\right\>=\left\<P_{0}^{1}Y_{1},\ul\mu_{0}^{1}\left(w\right)\right\>. 
\end{align}

By \eqref{eqS1}, \eqref{eqS15}, and \eqref{eqS2}, we 
get \eqref{eqS0}. 
	
Since $Y_0^{\fg}+\mu_{2}^{\fg}\in C_{0}^{\fg}$, 
$Y_{1}$ and $Y_{2}$ are non negative linear combinations of $\omega_{\alpha}$ with $\alpha\in \Delta_{0}^{2}(w)$, which gives the first estimate of \eqref{eqPY23p}.
Similarly, $P_{0}^{1}Y_{1}\in C_{0}^{1}$, which gives the second estimate of  \eqref{eqPY23p}. 
\end{proof} 

\subsection{Localisation of the problem} \label{sLoc}
  For 
$\e>0$, put 
\begin{align}\label{eqBve}
	B\(\ul\mu_{2}^{\fg}(w),\e\)=\left\{\ul\mu_{2}^{\fg}(w)+\sum_{\alpha\in 
	\Delta_{0}^{\fg}}y^{\alpha}\omega_{\alpha}\in \ft_{0}^{\fg}: 
	\left|y^{\alpha}\right|<\e\right\}.
\end{align} 
 Set 
	\begin{align}\label{eqIt3}
	J_{\e, t}\(w,\mu_{0}^{\fg}\)=\(\frac{t}{2\pi}\)^{r_{0}^{\fg}/2}\int_{ C_{0}^{\fg}\cap B\(\ul\mu_{2}^{\fg}(w),\e\)}	
	\widehat \pi_{0}^{\fg}\(w_{1}, tY_{0}^{\fg}\)\exp\(-tS_{w}\(Y_{0}^{\fg}\)\)dY_{0}^{\fg}.
	\end{align} 
	
 \begin{prop}\label{propLocal}
Given $\e>0$, there exist $\eta>0,C>0$ such that for $t\g1$, 
	\begin{align}\label{eqLocal}
		\left|J_{t}\(w,\mu_{0}^{\fg}\)-J_{\e,t}\(w,\mu_{0}^{\fg}\)\right|\l  C 
		\exp\(-tS_{w}\(\ul\mu_{2}^{\fg}(w)\)-t\eta\).
	\end{align} 
 \end{prop} 	
\begin{proof}	
By \eqref{eqSS2}, \eqref{eqItv},  and \eqref{eqIt3}, 	we have 
		\begin{multline}\label{eqLa1}
		J_{t}\(w,\mu_{0}^{\fg}\)-J_{t,\e}\(w,\mu_{0}^{\fg}\)
		=	\(\frac{t}{2\pi}\)^{r_{0}^{\fg}/2}
		\exp\(\frac{t}{2} \left|\ul\mu_{0}^{\fg}(w)\right|^{2}\)\\
		\times \int_{C_{0}^{\fg}\setminus B\(\ul\mu_{2}^{\fg}(w),\e\)} \widehat \pi_{0 }^{\mathfrak{g} } \(w_{1}, tY_{0}^{\fg}\)\exp\(-\frac{t}{2}\left|Y_{0}^{\fg}-\ul\mu_{0}^{\fg}(w)\right|^{2}\)dY_{0}^{\fg}. 
	\end{multline} 
	
By the unicity of the projection $\mu_{2}^{\fg}$, we have 
\begin{align}\label{eqLa2}
\min_{Y_{0}^{\fg}\in C^{\fg}_{0}\setminus B\(\ul\mu_{2}^{\fg}(w),\e\)}		
\left|Y_{0}^{\fg}-\ul\mu_{0}^{\fg}(w)\right|>\left|\ul\mu_{0}^{1}(w)\right|. 
	\end{align} 
By \eqref{eqLa2}, there exist $\eta>0$ and $\delta\in (0,1)$ such 
that for $Y_{0}^{\fg}\in 
	C_{0}^{\fg}\setminus B\(\ul\mu_{2}^{\fg}(w),\e\)$, 
	\begin{align}\label{eqLa3}
	\frac{1}{2}	\left|Y_{0}^{\fg}-\ul\mu_{0}^{\fg}(w)\right|^{2}\g \frac{\delta}{2} 
		\left|Y_{0}^{\fg}-\ul\mu_{0}^{\fg}(w)\right|^{2}+\frac{1}{2}\left|\ul\mu_{0}^{1}(w)\right|^{2}+2\eta. 
	\end{align} 

By \eqref{eqLa1} and \eqref{eqLa3}, we have 
	\begin{multline}
	\left|J_{t}\(w,\mu_{0}^{\fg}\)-J_{t,\e}\(w,\mu_{0}^{\fg}\)\right|
	\l 
	\(\frac{t}{2\pi}\)^{r_{0}^{\fg}/2}\exp\(-t\(S_{w}\(\ul\mu_{2}^{\fg}(w)\)+2\eta\)\)\\
	\times \int_{ C_{0}^{\fg}\setminus B\(\ul\mu_{2}^{\fg}(w),\e\)}	
	\widehat{\pi} _{0}^{\fg}\(w_{1}, tY_{0}^{\fg}\)\exp\(-\frac{\delta 
	t}{2}\left|Y_{0}^{\fg}-\ul\mu_{0}^{\fg}(w)\right|^{2}\)dY_{0}^{\fg}.
	\end{multline} 
	
It remains to show that  the term 
	\begin{align}\label{eqL5}
 	t^{r_{0}^{\fg}/2}	\exp\(-\eta t\)\int_{ C_{0}^{\fg}\setminus B\(\ul\mu_{2}^{\fg}(w),\e\)}	
	\widehat{\pi} _{0}^{\fg}\(w_{1}, tY_{0}^{\fg}\)\exp\(-\frac{\delta 
	t}{2}\left|Y_{0}^{\fg}-\ul\mu_{0}^{\fg}(w)\right|^{2}\)dY_{0}^{\fg}
	\end{align} 
	is uniformly bounded when  $t \g 1$. 
	Indeed, by  \eqref{eqTdexp} and \eqref{eqEdexp11}, if $Y_{0}^{\fg}\in C_{0}^{\fg}$, we have 
	\begin{align}\label{eqLa4}
		\left|\widehat \pi\(w_{1}, tY_{0}^{\fg}\)\right|\l C\(1+\left|tY_{0}^{\fg}\right|\)^{\left|R_{+}(\fg)\right|}\l  
		C\(1+t\)^{\left|R_{+}(\fg)\right|}\(1+\left|Y_{0}^{\fg}\right|\)^{\left|R_{+}(\fg)\right|}. 
	\end{align} 
By \eqref{eqLa4}, for $t\g 1$, we have 
	\begin{multline}\label{eqL6}
		\int_{ C_{0}^{\fg}\setminus B\(\ul\mu_{2}^{\fg}(w),\e\)}\widehat{\pi} _{0}^{\fg}\(w_{1}, tY_{0}^{\fg}\)\exp\(-\frac{\delta 
	t}{2}\left|Y_{0}^{\fg}-\ul\mu_{0}^{\fg}(w)\right|^{2}\)dY_{0}^{\fg}\\
\l C 	\left(1+t\right)^{\left|R_{+}(\fg)\right|}\int_{ \ft_{0}^{\fg}}\(1+\left|Y_{0}^{\fg}\right|\)^{\left|R_{+}(\fg)\right|}\exp\(-\frac{\delta 
	}{2}\left|Y_{0}^{\fg}-\ul\mu_{0}^{\fg}(w)\right|^{2}\)dY_{0}^{\fg}.
	\end{multline} 
By \eqref{eqL6}, we get the claim   \eqref{eqL5}, which finishes the proof of our proposition. 
\end{proof} 

\subsection{A rescaling on the normal coordinates} \label{sTinfty}
We fix now $\e>0$ small enough such that   
\begin{align}\label{eqCvbe}
C_{0}^{\fg}\cap B\(\ul\mu_{2}^{\fg}(w),\e\)=	
\left\{\ul\mu_{2}^{\fg}(w)+\sum_{\alpha\in 
\Delta_{0}^{\fg}}y^{\alpha}\omega_{\alpha}\in \ft_{0}^{\fg}: 
\substack{\textstyle
 0\l y^{\alpha} < \e \text{ 
if } \alpha\in \Delta_{0}^{2}(w)\\ \textstyle \left|y^{\alpha}\right|<\e \text{ 
if } \alpha\notin \Delta_{0}^{2}(w)\hfill}\right\}. 
\end{align} 

We introduction a nonhomogenous rescaling
	\begin{align}\label{eqCvbe1}
	Y_{0}^{\fg}=\ul\mu_{2}^{\fg}(w)+\frac{Y_{1}}{t}+\frac{Y_{2}+Y_{3}}{\sqrt{t}},
	\end{align} 
where $Y_{1},Y_{2}$ and $Y_{3}$ are as in \eqref{eqy123}.  Set
	\begin{align}\label{eqCvbe2}
		\begin{aligned}
		C_{\e,t}(w)=&\left\{\sum_{\alpha\in \Delta_{0}^{\fg}} 
		y^{\alpha}\omega_{\alpha}\in \ft_{0}^{\fg}: \substack{\textstyle 0\l 
		 y^{\alpha}< \e t \text{ if } \alpha \in \Delta_{0}^{1}(w) \hfill\\ \textstyle
		 0\l y^{\alpha} <  
		 \e \sqrt{t} \text{ if } \alpha\in \Delta_{0}^{2}(w)\setminus 
		 \Delta_{0}^{1}(w), \\ \textstyle
		 \left|y^{\alpha}\right|< \e \sqrt{t} \text{ if } 
		 \alpha\in \Delta_{0}^{\fg}\backslash\Delta_{0}^{2}(w)\hfill }\right\},\\
		 C_{\infty}(w)=&\left\{\sum_{\alpha\in \Delta_{0}^{\fg}} 
		 y^{\alpha}\omega_{\alpha}\in \ft_{0}^{\fg}: y^{\alpha}\g 0 \text{ if 
		 }\alpha\in \Delta_{0}^{2}(w)\right\}.
		 \end{aligned}
	\end{align} 
Then,
\begin{align}\label{eqCetinfty}
	C_{\e,t}\left(w\right)\subset C_{\infty}\left(w\right). 
\end{align} 
By \eqref{eqCvbe} and \eqref{eqCvbe2}, we see that 	
\begin{align}\label{eqeqCC}
	\ul\mu_{2}^{\fg}(w)+\frac{Y_{1}}{t}+\frac{Y_{2}+Y_{3}}{\sqrt{t}}\in 
	C_{0}^{\fg}\cap B\(\ul\mu_{2}^{\fg}(w),\e\) \iff	Y_{1}+Y_{2}+Y_{2}\in C_{\e,t}. 
\end{align}

By \eqref{eqS0}, \eqref{eqIt3}, and by  the change of variables
\eqref{eqCvbe1}, we have 
	\begin{multline}
		J_{\e,t}\(w,\mu_{0}^{\fg}\)=t^{-r_{0}^{1}(w)/2}\exp\(-tS_{w}\(\ul\mu_{2}^{\fg}(w)\)\)\\\times\int_{Y_{1}+Y_{2}+Y_{3}\in C_{\e,t}\left(w\right)} 
		\widehat \pi_{0}^{\fg}\(w_{1}, t\ul\mu_{2}^{\fg}(w)+Y_{1}+\sqrt{t}\({Y_{2}+Y_{3}}\)\)	
		\exp\Bigg(-\frac{1}{2t}\left|P_{0}^{2}Y_{1}\right|^{2}	
-\frac{1}{2}\left|P_{1}^{2}Y_{2}\right|^{2}\\	- \frac{1}{\sqrt{t}}
\left\<P_{1}^{2}Y_{1},P_{1}^{2}Y_{2}\right\>
-\frac{1}{2}\left|P_{2}^{\fg}\(\frac{Y_{1}}{\sqrt{t}}+Y_{2}\)+Y_{3}\right|^{2}+\left\<P_{0}^{1}Y_{1},\ul\mu_{0}^{1}(w)\right\>\Bigg)\frac{dY_{0}^{\fg}}{(2\pi)^{r_{0}^{\fg}/2}}.
	\end{multline} 
		
After translating the variable $Y_{3}$  by 
$-P_{2}^{\fg}\(Y_{1}/\sqrt{t}+Y_{2}\)$, we get 
	\begin{multline}\label{eqdc12}
J_{\e,t}\(w,\mu_{0}^{\fg}\)=t^{-r_{0}^{1}(w)/2}\exp\(-tS_{w}\(\ul\mu_{2}^{\fg}(w)\)\)\\
\times\int_{Y_{1} +Y_{2} +Y_{3}\in \widetilde{C}_{\e,t}(w)} 
	\widehat \pi_{0}^{\fg}\(w_{1}, t\ul\mu_{2}^{\fg}(w)+P_{0}^{2}\(Y_{1}+\sqrt{t}Y_{2}\)+\sqrt{t}Y_{3}\)
	\exp\Bigg(-\frac{1}{2t}\left|P_{0}^{2}Y_{1}\right|^{2}	\\
-\frac{1}{2}\left|P_{0}^{2}Y_{2}\right|^{2}	- \frac{1}{\sqrt{t}}
\left\<P_{0}^{2}Y_{1},P_{0}^{2}Y_{2}\right\>
-\frac{1}{2}\left|Y_{3}\right|^{2}+\left\<P_{0}^{1}Y_{1},\ul\mu_{0}^{1}(w)\right\>\Bigg)\frac{dY_{0}^{\fg}}{(2\pi)^{r_{0}^{\fg}/2}}.
		\end{multline}
where $\widetilde{C}_{\e,t}(w)$ is some domain obtained by 
translation.  Since $C_{\infty}(w)$ is stable under translation by an element of 
$\ft_{2}^{\fg}$. By \eqref{eqCetinfty}. we know that 
	\begin{align}
		\widetilde{C}_{\e,t}(w)\subset {C}_{\infty}(w). 
	\end{align} 


	
\begin{prop}\label{propPQ}
There is $C>0$ such that for all $t\g 1$ and $Y_{0}^{\fg}=Y_{1}+Y_{2}+Y_{3}\in \widetilde{C}_{\e,t}(w), $ we have 
\begin{multline}\label{eqconvd1}
t^{-\dim_{\mathbf C} \fu_{1}^{2}(w)/2-\dim_{\mathbf C} \fu_{2}^{\fg}(w)}\left|	\widehat \pi_{0}^{\fg}\(w_{1}, t\ul\mu_{2}^{\fg}(w)+P_{0}^{2}\(Y_{1}+\sqrt{t}Y_{2}\)+\sqrt{t}Y_{3}\)\right|\\
\l C(1+|Y|)^{\left|R_{+}(\fg)\right|}.
	\end{multline} 
As $t\to \infty$, we have  the pointwise  convergence in the interior of $C_{\infty}(w)$, 
	\begin{multline}\label{eqcond2}
		t^{-\dim_{\mathbf C} \fu_{1}^{2}(w)/2-\dim_{\mathbf C} \fu_{2}^{\fg}(w)}{\widehat 
		\pi_{0}^{\fg}\(w_{1}, t\ul\mu_{2}^{\fg}(w)+P_{0}^{2}\(Y_{1}+\sqrt{t}Y_{2}\)+\sqrt{t}Y_{3}\)}\\
		\to 
	\widehat{\pi}_{0}^{1}\(w, P_{0}^{1}Y_{1}\)	{\pi}_{1}^{2}\(w, P_{1}^{2} 
	Y_{2}\){\pi}_{2}^{\fg}\(w, \ul\mu_{2}^{\fg}(w)\). 
	\end{multline} 
\end{prop} 
\begin{proof}
	Recall that $\widehat \pi_{0}^{1}\left(w,\cdot\right),\widehat \pi_{1}^{2}\left(w,\cdot\right),\widehat \pi_{2}^{\mathfrak{g} }\left(w,\cdot\right)$ introduced in Definition \ref{definQQQ} are defined by roots in $ R\left(\mathfrak{g} \right)$.
	They extend naturally to functions on $\ft_{0}^{\fg}$, which will still denote by $\widehat \pi_{0}^{1}\left(w,\cdot\right),\widehat \pi_{1}^{2}\left(w,\cdot\right),\widehat \pi_{2}^{\mathfrak{g} }\left(w,\cdot\right)$.	
	Then,  
	\begin{align}\label{eqPQ123}
		\widehat \pi_{0 }^{\mathfrak{g} }\left(w_{1} ,\cdot\right) =\widehat \pi_{0}^{1}\left(w,\cdot\right)\widehat \pi_{1}^{2}\left(w,\cdot\right)\widehat \pi_{2}^{\fg}\left(w,\cdot\right).
	\end{align} 	
Clearly,
\begin{align}\label{eqPQP}
&\widehat \pi_{0}^{1}\(w,Y_{0}^{\fg}\)=\widehat \pi_{0}^{1}\(w,P_{0}^{1}Y_{0}^{\fg}\),&	
\widehat \pi_{1}^{2}\(w,Y_{0}^{\fg}\)=\widehat \pi_{1}^{2}\(w,P_{0}^{2}Y_{0}^{\fg}\). 
\end{align} 
	
By \eqref{eqPQ123} and \eqref{eqPQP}, we have 
	\begin{multline}\label{eqPQ3}
	\widehat	\pi_{0 }^{\mathfrak{g} }  \(w_{1}, t\ul\mu_{2}^{\fg}(w)+P_{0}^{2}\(Y_{1}+\sqrt{t}Y_{2}\)+\sqrt{t}Y_{3}\) \\ =	\widehat \pi _{0}^{1}\(w,P^{1}_{0}Y_{1}\) \widehat	\pi _{1}^{2}\(w,P_{0}^{2}\(Y_{1}+\sqrt{t}Y_{2}\)\)  \widehat \pi _{2}^{\fg}\(w,t\ul\mu_{2}^{\fg}(w)+P_{0}^{2}\(Y_{1}+\sqrt{t}Y_{2}\)+\sqrt{t}Y_{3}\). 
	\end{multline}
By our definition of $\widetilde C_{\e,t}\left(w\right)$ and by \eqref{eqeqCC}, we have 
	\begin{align}
t\ul\mu_{2}^{\fg}(w)+P_{0}^{2}\(Y_{1}+\sqrt{t}Y_{2}\)+\sqrt{t}Y_{3}\in 
C_{0}^{\fg}. 		
	\end{align} 
  Therefore, 
  	\begin{align}\label{eqPQ3111}
&P_{0}^{1}Y_{1}\in C_{0}^{1}\left(w\right), 
&P_{0}^{2}\(Y_{1}+\sqrt{t}Y_{2}\)\in 
C^2_{0}\left(w\right).
\end{align} 
By estimates which are similar to \eqref{eqTdexp}, and by \eqref{eqPQ3}-\eqref{eqPQ3111}, we get \eqref{eqconvd1}. 
		
Observe that if $Y_{0}^{\fg}=Y_{1}+Y_{2}+Y_{3}$ is in the interior of 
$ C_{\infty}(w)$, if  $f\in 
A_{1}^{2}(w)\cup 
B_{1}^{2}(w)$ and $f'\in A_{2}(w)\cup B_{2}(w)$, then 
	\begin{align}
 &\left\<f,P_{0}^{2}Y_{2}\right\>>0, &		\left\langle 
 f',\ul\mu_{2}^{\fg}(w)\right\rangle >0. 
	\end{align} 
Moreover,  as $t\to 
\infty$, we have 
	\begin{align}
		\begin{aligned}
		\left\<f,P_{0}^{2}\(\sqrt{t}Y_{2}+Y_{3
		}\)\right\>&\sim \sqrt{t} 
		\left\<f,P_{0}^{2}Y_{2}\right\rangle,\\
		\left\<f',t\ul\mu_{2}^{\fg}(w)+P_{0}^{2}\(Y_{1}+\sqrt{t}Y_{2}\)+\sqrt{t}Y_{3}\right\>&\sim t\left\<f^{\prime } ,\ul\mu_{2}^{\fg}(w)\right\rangle. 
		\end{aligned}
	\end{align} 
By \eqref{eqPQ3} and by the above   observations, and by $P_{0}^{2}Y_{2}=P_{1}^{2}Y_{2}$,  we get \eqref{eqcond2}. 
\end{proof}

\begin{proof}[The proof of Theorem \ref{thmIt}.]
	By \eqref{eqPY23p}, and by our definition of $\widetilde C_{\e, 
	t}(w)$, if 
	$Y_{0}^{\fg}=Y_{1}+Y_{2}+Y_{3}\in \widetilde C_{\e, t}(w)$, we have   
\begin{multline}\label{eqGaul}
		\exp\(-\frac{1}{2t}\left|P_{0}^{2}Y_{1}\right|^{2}	
-\frac{1}{2}\left|P_{1}^{2}Y_{2}\right|^{2}	- \frac{1}{\sqrt{t}}
\left\<P_{1}^{2}Y_{1},P_{1}^{2}Y_{2}\right\>
-\frac{1}{2}\left|Y_{3}\right|^{2}+\left\<P_{0}^{1}Y_{1},\ul\mu_{0}^{1}(w)\right\>\)\\
\l \exp\(\left\<P_{0}^{1}Y_{1},\ul\mu_{0}^{1}(w)\right\>	
-\frac{1}{2}\left|P_{1}^{2}Y_{2}\right|^{2}
-\frac{1}{2}\left|Y_{3}\right|^{2}\).\end{multline} 

Since the function $	\(1+\left|Y_{0}^{\fg}\right|\)^{|R_{+}(\fg)|}
\exp\(\left\<P_{0}^{1}Y_{1},\ul\mu_{0}^{1}\left(w\right)\right\>	-\frac{1}{2}\left|P_{0}^{2}Y_{2}\right|^{2}-\frac{1}{2}\left|Y_{3}\right|^{2}\)
$ is integrable on $C_{\infty}(w)$, by  Proposition \ref{propPQ}, by 
	\eqref{eqGaul},  and by the 
dominated convergence theorem, as $t\to \infty$,  we get 
\begin{multline}\label{eqdc11}
	t^{-\dim_{\mathbf C} \fu_{1}^{2}(w)/2-\dim_{\mathbf C}  
	\fu_{2}^{\fg}(w)}\int_{\widetilde{C}_{\e,t}(w)} 
		\widehat \pi_{0}^{\fg}\(w_{1}, t\ul\mu_{2}^{\fg}(w)+P_{0}^{2}\(Y_{1}+\sqrt{t}Y_{2}\)+\sqrt{t}Y_{3}\)\\
		\exp\(-\frac{1}{2t}\left|P_{0}^{2}Y_{1}\right|^{2}	
-\frac{1}{2}\left|P_{1}^{2}Y_{2}\right|^{2}	- \frac{1}{\sqrt{t}}
\left\<P_{1}^{2}Y_{1},P_{1}^{2}Y_{2}\right\>
-\frac{1}{2}\left|Y_{3}\right|^{2}+\left\<P_{0}^{1}Y_{1},\ul\mu_{0}^{1}(w)\right\>\)\frac{dY_{0}^{\fg}}{(2\pi)^{r_{0}^{\fg}/2}}\\
		\to \pi _{2}^{\fg}\(w,\ul\mu_{2}^{\fg}(w)\)\int_{C_{\infty}(w) }	
		\widehat {\pi}_{0}^{1}\(w, P_{0}^{1}Y_{1}\)	{\pi}_{1}^{2}\(w,P_{1}^{2} 
	Y_{2}\) 
    \\
   \times \exp\(\left\<P_{0}^{1}Y_{1},\ul\mu_{0}^{1}(w)\right\>	
-\frac{1}{2}\left|P_{1}^{2}Y_{2}\right|^{2}
-\frac{1}{2}\left|Y_{3}\right|^{2}\)\frac{dY_{0}^{\fg}}{(2\pi)^{r_{0}^{\fg}/2}}. 
\end{multline} 
Since the isomorphism \eqref{eqpyFG} sends $C_{\infty}$ to 
$C_{0}^{1}\times C_{1}^{2}\times \ft_{2}^{\fg}$ , by 
Proposition \ref{propintVVp},  the last two lines of \eqref{eqdc11} is 
equal to $\alpha\(w,\mu_{0}^{\fg}\)$ given in \eqref{eqrAc}. 
	
By \eqref{eqLocal}, \eqref{eqdc12}, \eqref{eqdc11}, and the above observation, we get \eqref{eqItasy}. The proof of  Theorem \ref{thmIt} is completed. 
\end{proof} 

\subsection{The case $ G$ is equal rank, $ w= 1$,  and $ \Delta_{0}^{1} = \Delta_{0}^{2} = \varnothing $}\label{ssLast}
Now we assume that $ G$ is equal rank, $ w= 1$ and $ \Delta_{0}^{1} = \Delta_{0}^{2} = \varnothing $.
Since $ w= 1$, the notation like $ J_{t} \left(w,\cdot\right), \pi _{0 }^{\mathfrak{g} } \left(w_{1}, \cdot\right),\widehat{\pi}  _{0 }^{\mathfrak{g} } \left(w_{1}, \cdot\right),S_{w} $ will be denoted by $ J\left(\cdot\right), \pi _{0 }^{\mathfrak{g} } \left( \cdot\right),\widehat{\pi} _{0 }^{\mathfrak{g} } \left( \cdot\right),S $.

Then, 
\begin{align}\label{eq:ewxn}
 \underline{\mu }_{0 }^{\mathfrak{g} }= \underline{\mu }_{2}^{\mathfrak{g} }  \in {\rm Int}\left(C_{0 }^{\mathfrak{g} } \right).
\end{align}
Moreover, 
\begin{align}\label{eq:hniv}
		\min_{C_{0}^{\fg}} S= 
		S\(\ul\mu_{0 }^{\fg}\)=-\frac{1}{2}\left|\ul\mu_{0 }^{\fg}\right|^{2}.
\end{align}

For $ \epsilon > 0 $ small enough, we have 
\begin{align}\label{eq:telf}
	B\(\ul\mu_{2}^{\fg},\e\) \subset  C_{0}^{\fg}.
\end{align}

Put 
\begin{align}\begin{aligned}\label{eq:dddj}
	{K}_{ t}\left(\mu_{0}^{\mathfrak{g} } \right) &=\(\frac{t}{2\pi}\)^{r_{0}^{\fg}/2}\int_{ \mathfrak{t}_{0}^{\mathfrak{g} } }	
  \pi_{0}^{\fg}\( tY_{0}^{\fg}\)\exp\(-tS\(Y_{0}^{\fg}\)\)dY_{0}^{\fg},\\
 {K}_{\epsilon, t}\left(\mu_{0}^{\mathfrak{g} } \right) &=\(\frac{t}{2\pi}\)^{r_{0}^{\fg}/2}\int_{ B\(\ul\mu_{2}^{\fg}(w),\e\)} \pi_{0}^{\fg}\(tY_{0}^{\fg}\)\exp\(-tS\(Y_{0}^{\fg}\)\)dY_{0}^{\fg}.
\end{aligned}\end{align}

\begin{prop}\label{prop:kulj}
	Given $\e>0$, there exist $\eta>0,C>0$ such that for $t\g1$, 
	\begin{align}\begin{aligned}\label{eq:mdnp}
	 \left| J_{\epsilon, t}\left(\mu_{0}^{\mathfrak{g} } \right) - K_{\epsilon, t} \left(\mu_{0}^{\mathfrak{g} } \right) 		\right| &\l  C  \exp\(-tS\(\ul\mu_{0 }^{\fg}\)-t\eta\),	\\
	 \left| K_{\epsilon, t}\left(\mu_{0}^{\mathfrak{g} } \right)  - K_{ t} \left(\mu_{0}^{\mathfrak{g} } \right) 		\right| &\l  C  \exp\(-tS\(\ul\mu_{0 }^{\fg}\)-t\eta\).
	\end{aligned}\end{align}
\end{prop}
\begin{proof}
 The proof of the second equation of \eqref{eq:mdnp} is the same as the one given in \eqref{propLocal}.
 
 Let us show the first one. 
 There exists $ C> 0 $ such that for $ Y_{0 }^{\mathfrak{g} }\in {\rm Int}\left(C_{0 }^{\mathfrak{g} }\right) $, we have 
 \begin{align}\label{eq:2anr}
0 \l	\frac{1}{\prod_{f\in R_{+} \left(\mathfrak{g} \right)} \left(1-e^{-\left\langle Y_{0 }^{\mathfrak{g} } , f \right\rangle } \right)} -1 \l C \sum_{\alpha\in \Delta_{0}^{\mathfrak{g} } }^{} \frac{e^{-\left\langle Y_{0 }^{\mathfrak{g} } , \alpha \right\rangle } }{\prod_{f\in R_{+} \left(\mathfrak{g} \right)} \left(1-e^{-\left\langle Y_{0 }^{\mathfrak{g} } , f \right\rangle } \right)}.
 \end{align}
Multiplying \eqref{eq:2anr} by $ \pi_{0}^{\fg}\(Y_{0}^{\fg}\) $, we get
\begin{align}\label{eq:nzod}
 0 \l \widehat{\pi }_{0}^{\mathfrak{g} } \left(Y_{0 }^{\mathfrak{g} }\right) - \pi_{0}^{\fg}\(Y_{0}^{\fg}\) \l C \widehat{\pi }_{0}^{\mathfrak{g} } \left(Y_{0 }^{\mathfrak{g} }\right) \sum_{\alpha\in \Delta_{0}^{\mathfrak{g} } }^{} e^{-\left\langle Y_{0 }^{\mathfrak{g} } , \alpha \right\rangle }.
\end{align}

By \eqref{eq:nzod}, we have 
\begin{align}\label{eq:bqcr}
 0 \l J_{\epsilon, t}\left(\mu_{0}^{\mathfrak{g} } \right) - K_{\epsilon, t} \left(\mu_{0}^{\mathfrak{g} } \right) \l  C \sum_{\lambda\in \Delta_{0}^{\mathfrak{g} } }^{} J_{\epsilon, t} \left(\mu_{0}^{\mathfrak{g} } -\alpha \right).
\end{align}
By Proposition \ref{propv22delta}, using our asymptotic formula (\ref{eqItasy}) on $J_{\epsilon, t} \left(\mu_{0}^{\mathfrak{g} } -\alpha \right)$ and \eqref{eq:bqcr}, we get the first estimate in \eqref{eq:mdnp}.
\end{proof}

\begin{prop}\label{prop:d3bs}
 For $ t> 0 $, we have 
 \begin{align}\label{eq:rlqk}
	K_{ t} \left(\mu_{0}^{\mathfrak{g} } \right)= \pi_{0}^{\mathfrak{g} } \left(t \ul \mu_{0}^{\mathfrak{g} } \right)\exp\left({\frac{t}{2} \left| \ul \mu_{0}^{\mathfrak{g} }  \right|^{2} }\right).
 \end{align}
\end{prop}
\begin{proof}
	After a rescaling on the variable $Y_{0}^{\fg}$, we have an analogue of \eqref{eqItvr}, 
	\begin{align}\label{eqItvr1}
	K_{t}\(\mu_{0}^{\fg}\)=\int_{ \mathfrak{t}_{0 }^{\mathfrak{g} } }	
 \pi_{0}^{\fg}\(Y_{0}^{\fg}\)\exp\(\left\langle \ul \mu_{0}^{\fg},
	Y_{0}^{\fg}\right\rangle -\frac{\left|Y_{0}^{\fg}\right|^{2}}{2t}\) \frac{d Y_{0}^{\fg}}{(2\pi t)^{ 
	r_{0}^{\fg}/2}}. 
	\end{align} 
  If $ \Delta^{\mathfrak{t}_{0}^{\mathfrak{g} } } $ is the Laplacian on the Euclidean space $ \mathfrak{t}_{0} $, using the heat equation, we have 
	\begin{align}\label{eq:eowo}
		K_{t}\(\mu_{0}^{\ft}\)=e^{\frac{t}{2} \left|\ul \mu_{0}^{\mathfrak{g} }  \right|^{2} } \left(e^{t\Delta^{\mathfrak{t}_{0}^{\mathfrak{g} }  }} \pi_{0}^{\fg}\right)\left(t\ul\mu_{0}^{\mathfrak{g} } \right).
	\end{align}
	Since $ \pi_{0}^{\fg}$ is a harmonic polynomial \cite[(7.5.22)]{B09}, we have 
	\begin{align}\label{eq:rfah}
		e^{t\Delta^{\mathfrak{t}_{0}^{\mathfrak{g} }  }} \pi_{0}^{\fg}= \pi_{0}^{\fg}.
	\end{align}
	By \eqref{eq:eowo} and \eqref{eq:rfah}, we get our proposition. 
\end{proof}

\begin{proof}[Proof of Theorem \ref{thm:hjmy}]
 Theorem \ref{thm:hjmy} is a consequence of Propositions \ref{propLocal}, \ref{prop:kulj}, and \ref{prop:d3bs}.
\end{proof}

\subsection{Proof of Theorems \ref{thm32} and \ref{thm:ldme}}\label{sPthm32}
%
By \eqref{eqIJwmu} and \eqref{eqItasy}, as $ t \to \infty$, we have 
\begin{align}\label{eqIaas}
I_{t}^{\fg}\(\mu, w\) \sim \e_{w_{2}}\alpha\(w,\mu_{0}^{\fg}\) 
t^{\beta\(w,\mu_{0}^{\fg}\)}\exp\(\frac{t}{2}\left|\ul\mu_{2}(w)\right|^{2} \).
\end{align} 

Applying \eqref{eqIaas} to the case where
\begin{align}\label{eq:vygl}
 \mu  = \lambda^{E} + \varrho^{\mathfrak{k} }, 
\end{align}
we get Theorem \ref{thm32} with 
\begin{align}\label{eq:eri2}
 &\alpha _{w} = \alpha\(w,\mu_{0}^{\fg}\), & \beta_{w} = \beta\(w,\mu_{0}^{\fg}\),&& \gamma_{w} = \frac{1}{2}\left|\ul\mu_{2}(w)\right|^{2}.
\end{align}

Using \eqref{eq:3jnt} instead, by a similar argument, we get Theorem \ref{thm:ldme}.
 
\subsection{Proof of Theorem \ref{thm33}}\label{sPthm33}	
In this section, we assume that 
\begin{align}\label{eq:oanz}
& \mu\in {\rm Int}\left(C_{+} \left(\mathfrak{k} \right)\right),& \mu +\varrho^{\mathfrak{k} } \in C_{+} \left(\mathfrak{g} \right).
\end{align}
Recall that $ \ul\mu(w)$ and $ \underline{\mu } $ are defined in Definition \ref{defmuw1}.
	
\begin{prop}\label{propwpp12}
	If $w=w_{1}w_{2}\in W(\fg)$ satisfying \eqref{eqwww12},  we have 
	\begin{align}\label{eqlw12}
	&\mu+\varrho^{\fk}-w_{1}\(\mu+\varrho^{\fk}\)\in 
	\check C_{0}^{\fg},
	&w_{1}\mu-w\mu\in \check C_{0}^{\fg}. 
\end{align} 
	In particular, 
	\begin{align}\label{eqlw13}
	\ul\mu-\ul\mu(w)\in 
	\check C_{0}^{\fg}. 
\end{align} 
	\end{prop} 
\begin{proof}
Since $\mu+\varrho^{\fk}\in C_{+}(\fg)$, the first relation  in \eqref{eqlw12} is a consequence of \eqref{eqw-}. 
Similarly, we have 
\begin{align}
	\mu-w_{2}\mu\in 
\check C_{+}(\fk). 
\end{align}  
Since $w_{1}^{-1}C_{+}(\fg)\subset C_{+}(\fk)$, then 
$\mu-w_{2}\mu$ is 
nonnegative on $w_{1}^{-1}C_{+}(\fg)$, from which we get the second relation  in  \eqref{eqlw12}. 
Taking the sum of two relations in \eqref{eqlw12}, we get \eqref{eqlw13}.
\end{proof} 

\begin{prop}\label{prop:616}
	If $ \mu\in \mathfrak{t}_{0} $ such that \eqref{eq:oanz} holds, then
\begin{align}\label{eqll1}
\left|\underline{\mu } _{2}\right|\g 
\left|\underline{\mu } _{2}\left(w\right)\right|, 
\end{align} 
where the equality holds if and only if 
\begin{align}\label{eqw12inW20}
	w\in W^{2}_{0}. 
\end{align} 
If one of the above equivalent condition holds, then
\begin{align}\label{eqD02wD}
&	\Delta_{0}^{1} \subset \Delta_{0}^{1}\left(w\right),&\Delta_{0}^{2} = \Delta_{0}^{2}(w). 
\end{align} 
\end{prop} 
\begin{proof}
The relations \eqref{eqll1} and \eqref{eqD02wD} are consequences of \eqref{eqLanC}, \eqref{eqDinD}, and \eqref{eqlw13}.

By \eqref{eqLanC0}, the equality in \eqref{eqll1} holds  if and only if 
%
\begin{align}\label{eqll2}
 		\left\langle \ul\mu_{0}^{\fg}-\ul\mu_{0}^{\fg}(w),\ul\mu_{2}^{\fg}\right\rangle =0. 
	\end{align} 
	By \eqref{eqlw12}, we see that \eqref{eqll2} is  equivalent to 
	\begin{align}\label{eqll3}
	&	\left\langle \mu_{0}^{\fg}+\varrho^{\fk}- 
		w_{1}\(\mu_{0}^{\fg}+\varrho^{\fk}\),\ul\mu_{2}^{\fg}\right\rangle =0,
&				\left\langle 
		\mu_{0}^{\fg}-w_{2}\mu_{0}^{\fg},w_{1}^{-1}\ul\mu_{2}^{\fg}\right\rangle =0. 
	\end{align} 

By \eqref{eqwww}, the first equation in \eqref{eqll3} is equivalent to 
$w_{1}=w'_{1}w''_{1}$ such that 
\begin{align}\label{eqw1pp}
&	w'_{1} 
\ul\mu^{\fg}_{2}=\ul\mu^{\fg}_{2},& 
w_{1}''\(\mu_{0}^{\fg}+\varrho^{\fk}\)=\mu_{0}^{\fg}+\varrho^{\fk}. 
\end{align} 
By Proposition \ref{propCVogan} and by the second equation in \eqref{eqw1pp}, we see that  $w_{1}''$ fixes the vector $\underline{\mu} ^{\fg}_{2}$. 
Combining with the first equation of \eqref{eqw1pp}, we see that $ w_{1} $ fixes $\underline{\mu} ^{\fg}_{2}$.
By Chevalley's Lemma, $w_{1}\in W_{0}^{2}$. 
On the anther hands, it is immediate that if $ w_{1}\in W_{0 }^{2} $, the first equation in \eqref{eqll3} holds.
In summary, we have shown that the first equation in \eqref{eqll3} holds if and only if $ w_{1}\in W_{0 }^{2} $.

Since $ w_{1}^{-1} \underline{\mu }_{2}^{\mathfrak{g} }\in w_{1}^{-1} C_{+} \left(\mathfrak{g} \right) \subset C_{+} \left(\mathfrak{k} \right) $, by Proposition \ref{propuvC}, using the fact that $\mu_{0}^{\fg}$ is in the open Weyl chamber ${\rm Int}\(C_{+}(\fk)\)$, we see that the second equation in (\ref{eqll3}) is equivalent to the property that $w_{2}$ fixes the vector $w_{1}^{-1}\ul\mu^{\fg}_{2}$. 

By the above observations, we see that \eqref{eqll3} is equivalent to 
\begin{align}
w_{1},w_{2}\in W_{0}^{2}. 
\end{align} 
By Proposition \ref{propw12l1}, this is equivalent to $w\in 
W_{0}^{2}$. 
\end{proof} 

\begin{proof}[Proof of Theorem \ref{thm33}]
	Applying Proposition \ref{prop:616} to $ \mu= \lambda^{E} + \varrho^{\mathfrak{k} }  $, we get Theorem \ref{thm33}. 
\end{proof}

\subsection{Proof of Theorem \ref{thm34}}\label{sPthm34}
In this section, we will assume that \eqref{eq:oanz} holds and  $ w\in W_{0}^{2}.$ 
By \eqref{eqD02wD}, $ \mathfrak{t}^{2}_{0 } \left(w\right)= \mathfrak{t}^{2}_{0} $ is independent $ w\in W_{0 }^{2} $, while $ \mathfrak{t}^{1}_{0 } \left(w\right)$ still depends on $ w\in W_{0 }^{2} $.
We have the decomposition 
\begin{align}
\ft_{0}^{2}=\ft_{0}^{1}(w)\oplus \ft_{1}^{2}(w). 
\end{align} 


By \eqref{eqrmu0}, we have 
\begin{align}\label{eqbw1}
\beta\(w,\mu_{0}^{\fg}\)=-r_{0}^{1}(w)+\frac{1}{2}\dim_{\mathbf{C} } \fu_{1}^{2}(w)+\dim_{\mathbf{C} } \fu_{2}. 
\end{align}

\begin{prop}\label{prop:rsfs}
 If $ \mu\in \mathfrak{t}_{0} $ such that \eqref{eq:oanz} holds and if $ w\in W_{0 }^{2} $, then 
 \begin{align}\label{eq:skyg}
	\beta\left(w,\mu_{0}^{\fg}\right)\l \beta\left(1,\mu_{0}^{\fg}\right),
 \end{align}
 where the equality holds if and only if 
 \begin{align}\label{eq:nxix}
	w\in W_{0 }^{1}. 
 \end{align}
\end{prop}
\begin{proof}
By the first equation of \eqref{eqD02wD} and by \eqref{eqbw1}, we have  
\begin{align}
 \beta\left(w,\mu_{0}^{\fg}\right)\l \beta\left(1,\mu_{0}^{\fg}\right),
\end{align} 
where the equality holds if and only if $\Delta_{0}^{1}=\Delta_{0}^{1}(w)$.
This is also equivalent to 
	\begin{align}\label{eqaw123}
 		\left\langle 
		\underline{\mu }_0^{\fg} -\underline{\mu }_0^{\fg} \left(w\right),\omega_{\alpha}\right\rangle =0, \text{ for all }\alpha\in \Delta_{0}^{2}\backslash \Delta_{0}^{1}.
	\end{align} 

By \eqref{eqlw13}, the condition \eqref{eqaw123} is equivalent to 
	\begin{align}\label{eq684}
		\begin{aligned}
		\left\langle 
	\mu_{0}^{\fg}+\varrho^{\fk}-	
	w_{1}\(\mu_{0}^{\fg}+\varrho^{\fk}\),\omega_{\alpha}\right\rangle =0, 
	&\text{ for all }\alpha\in \Delta_{0}^{2}\backslash \Delta_{0}^{1}.
 \\
	\left\langle 
	w_{1}\mu_{0}^{\fg}-	
	w\mu_{0}^{\fg},\omega_{\alpha}\right\rangle =0, &\text{ for all }\alpha\in \Delta_{0}^{2}\backslash \Delta_{0}^{1}.
	\end{aligned}
	\end{align} 
	
Using the first equation of \eqref{eq:lsos} and arguments similar to those given below \eqref{eqw1pp}, we see that the first equation of \eqref{eq684} is equivalent to $w_{1}\in W_{0}^{1}$.

If $w_{1}\in W_{0}^{1}$, the second equation becomes 
	\begin{align}
	\left\langle 
	\mu_{0}^{\fg}-	
	w_{2}\mu_{0}^{\fg},\omega_{\alpha}\right\rangle =0, &\text{ for all }\alpha\in \Delta_{0}^{2}\backslash \Delta_{0}^{1}.
	\end{align} 
	Since $\mu_{0}^{\fg}\in {\rm Int}\(C_{+}(\fk)\)$, $w_{2}$ fixes 
	$\omega_{\alpha}$ with $\alpha\in \Delta_{0}^{2}\backslash 
	\Delta_{0}^{1}$. This means $w_{2}\in W_{0}^{1}$.
	
	By the above observation, we see that 
	\begin{align}\label{eq:clgh}
	 \Delta_{0}^{1}=\Delta_{0}^{1}(w) \iff w_{1}, w_{2}\in W_{0 }^{1},  
	\end{align}
	which is also equivalent to $ w \in W_{0 }^{1}$ by Proposition \ref{propw12l1}.
\end{proof} 

%
%

\begin{proof}[Proof of Theorem \ref{thm34}]
	Applying Proposition \ref{prop:rsfs} to $ \mu= \lambda^{E} + \varrho^{\mathfrak{k} }  $, we get Theorem \ref{thm34}. 
\end{proof}

\subsection{Proof of Theorem \ref{thm35}}\label{sPthm35}
In this section, we will assume that \eqref{eq:oanz} holds and  $ w\in W_{0}^{1}.$ 
Then, $ \Delta_{0}^{1}\left(w\right), \Delta_{0}^{2} \left(w\right) $ are independent of $ w\in W_{0 }^{1} $, so that the decomposition  
\begin{align}
\ft=\ft_{0}^{1}\oplus \ft_{1}^{2}\oplus \ft_{2}
\end{align} 
is also independent of $w\in W_{0 }^{1} $.

By \eqref{eqa101212} and by the above observation, $\alpha_{1}^{2}\(w,\mu_{0}^{\fg}\)$ and $\alpha_{2}\(w,\mu_{0}^{\fg}\)$ do not depend on $w\in 
W_{0}^{1}$, so that 
	\begin{align}\label{eq689}
	&	\alpha_{1}^{2}\(w,\mu_{0}^{\fg}\)=\alpha_{1}^{2}\left(1,\mu_{0}^{\mathfrak{g} } \right), 
		&\alpha_{2}\(w,\mu_{0}^{\fg}\)=\alpha_{2}\left(1,\mu_{0}^{\mathfrak{g} } \right). 
	\end{align} 

Recall that $ \mathfrak{k}^{1}_{s}, \mathfrak{p}^{1}, \mathfrak{m}^{1} $ are defined in \eqref{eqmfks}. 
Let $ \mu_{0 }^{1} $ be the $ \ft_{0 }^{1} $-component of $ \mu$. 

\begin{prop}
If $ \mu\in \mathfrak{t}_{0} $ such that \eqref{eq:oanz} holds and if $w=w_{1}w_{2}\in W^{1}_{0} $ satisfying \eqref{eqwww12}, then the following integral converges so that  
\begin{align}\label{eq690}
	\e_{w_{2}}\alpha_{0}^{1}\(w,\mu_{0}^{\mathfrak{g} } \)=\int_{w^{-1}C_{0}^{1}}\pi^{\fk^{1}_{s} }\(Y_{0}^{1}\)\widehat{A}\(\ad\(Y_{0}^{1}\)_{|\fp^{1}}\)\exp\(\left\<{\mu} _{0}^{1},Y_{0}^{1}\right\>\)\frac{dY_0}{(2 \pi )^{r_{0}^{1}/2}}.
\end{align}
In particular, the following integral is well-defined so that 
	\begin{align}\label{eq691}
	\sum_{w\in W_{0}^{1}}	
	\e_{w_{2}}\alpha_{0}^{1}\(w,\mu_{0}^{\mathfrak{g} } \)=\int_{\ft_{0}^{1}}\pi^{\fk^{1}_s}\(Y_{0}^{1}\)\widehat{A}\(\ad\(Y_{0}^{1}\)_{|\fp^{1}}\)\exp\(\left\<{\mu} _{0}^{1},Y_{0}^{1}\right\>\)\frac{dY_0}{(2 \pi )^{r_{0}^{1}/2}}.
	\end{align} 
\end{prop} 
\begin{proof}
By \eqref{eqk1rho}, Definition \ref{defA0123}, and by the first formula of \eqref{eqa101212}, we have 
\begin{multline}
\alpha_{0}^{1}\left(w,\mu_{0}^{\mathfrak{g} } \right)=\int_{Y_{0}^{1}\in 
		C_{0}^{1}}\prod_{f\in w_{1} 
		R_{+}\(\fk^{1}_{s} \)} \left\<f,Y_{0}^{\fg}\right\>
 	{\prod_{{f\in w_{1} R\(\fp^{1}\) \cap R_{+} \left(\mathfrak{m}^{1} \right)}} {\rm 
	Td}\(\left\langle f,Y_{0}^{\fg}\right\rangle \)}\\
	\times \exp\(\left\langle	Y_{0}^{1},w\mu_{0}^{1}+w_{1}\varrho^{\fk^{1}_{s}}-\varrho^{\mathfrak m^{1}}\right\rangle 
		\)\frac{dY_{0}^{1}}{(2\pi)^{r_{0}^{1}/2}}.		
\end{multline} 
Proceeding as in the proof of \eqref{eqIgmuw1}, we know that the right hand side of \eqref{eq690} is well-defined, so that \eqref{eq690} holds.

Taking of sum of \eqref{eq690} over $ w\in W_{0 }^{1} $, we get \eqref{eq691}.
\end{proof} 

Assume now $\mu=\lambda^{E}+\varrho^{\fk}$.  Recall that $\left(\tau^{E,1}, E^{1}\right)$ is the $K^{1}_{s}$-representation introduced in Definition \ref{deftauE1}. 

\begin{prop}
The following identity holds, 
\begin{align}\label{eq693}
\sum_{w\in 
W_{0 }^{1}}\e_{w_{2}}\alpha_{0}^{1}\left(w,\mu_{0}^{\mathfrak{g} }  \right)=\pi^{\fk^{1}_{s}}\(\varrho^{\fk^{1}_{s}}\)\int_{ \sqrt{-1}{\fk}^{1}_{s}}\frac{\widehat 
		A\(\ad\(Y^{\fk^1_{s}}_{0}\)_{|{\fp}^{1}_{s}}\)}{\widehat 
A\(\ad\(Y^{\fk^{1}_{s}}_{0}\)_{\fk^{1}_{s}}\)}\Tr\[\tau^{E,1}\(e^{-Y^{\fk^1_{s}}_{0}}\)\]\frac{dY_{0}^{\fk^{1}_{s}}}{(2\pi)^{n_{s}^{1}/2}}. 
\end{align}
\end{prop} 
\begin{proof}
Proceeding as the proof of \eqref{eqTrGwi1}, using \eqref{eq691}, we get our proposition.
\end{proof} 

\begin{proof}[Proof of Theorem \ref{thm35}]
 Applying \eqref{eq689} to the case $\mu=\lambda^{E}+\varrho^{\fk}$, by \eqref{eq313} and \eqref{eqa101212}, we have 
 \begin{align}\label{eq:g2md}
	&\alpha_{1}^{2}\left(1,\mu_{0}^{\mathfrak{g} } \right)= \alpha_{1}^{2}, &\alpha_{2}\left(1,\mu_{0}^{\mathfrak{g} } \right)= \alpha_{2}.
 \end{align}
 Moreover, the right hand side of \eqref{eq693} coincides with $ \alpha_{0}^{1} $.
 Theorem \ref{thm35} follows from the above observations.
\end{proof}

\end{document}